\documentclass[11pt]{article}
\usepackage[inner=3.4 cm,outer=3.4 cm,top=3.0 cm,bottom=2.0 cm]{geometry}

\usepackage{amsmath,amssymb,amsfonts,amsthm}
\usepackage{setspace}
\newcounter{item}[section]
\newcounter{kirshr}
\newcounter{kirsha}
\newcounter{kirshb}
\newenvironment{enumroman}{\setcounter{kirshr}{1}
\begin{list}{(\roman{kirshr})}{\usecounter{kirshr}} }{\end{list}}
\newenvironment{enumarab}{\setcounter{kirshb}{1}
\begin{list}{(\arabic{kirshb})}{\usecounter{kirshb}} }{\end{list}}
\newtheorem{theorem}{Theorem}[section]

\newtheorem{lemma}[theorem]{Lemma}
\newtheorem{corollary}[theorem]{Corollary}
\theoremstyle{definition}

\newtheorem{example}[theorem]{Example}
\newtheorem{definition}[theorem]{Definition}
\newtheorem{proposition}[theorem]{Proposition}
\def\R{\mathbb{R}}
\def\Q{\mathbb{Q}}

\def\A{{\mathfrak{A}}}
\def\At{{\sf At}}

\def\A{{\mathfrak{A}}}
\def\At{{\sf At}}
\def\B{{\mathfrak{B}}}
\def\C{{\mathfrak{C}}}
\def\Ca{{\mathfrak Ca}}
\def\CA{{\sf CA}}
\def\Cm{{\sf Cm}}

\def\de{Dedekind-MacNeille}
\def\ef{Ehren\-feucht--Fra\"\i ss\'e}
\def\F{{\mathfrak{F}}}
\def\Fm{{\mathfrak{Fm}}}
\def\Fr{{\mathfrak{Fr}}}
\def\g{{\sf g}}
\def\Id{{\sf Id}}
\def\K{{\sf K}}
\def\Lf{{\sf Lf}}
\def\Mo{{\sf M}}
\def\N{{\mathbb{N}}}
\def\nodes{{\sf nodes}}
\def\Nr{{\mathfrak{Nr}}}
\def\Nrr{{\mathfrak{Nr}}}
\def\pa{$\forall$}
\def\PA{{\bf PA}}
\def\pe{$\exists$}
\def\PEA{\sf PEA}
\def\QEA{{\sf QEA}}
\def\r{{\sf r}}
\def\R{\mathfrak{R}}
\def\Ra{{\sf Ra}}
\def\RCA{{\sf RCA}}
\def\Rd{{\mathfrak{Rd}}}

\def\rng{{\sf rng}}

\def\s{{\sf s}}
\def\Sc{{\sf Sc}}
\def\Sg{{\mathfrak Sg}}
\def\Tm{{\mathfrak{Tm}}}
\def\tp{{\sf tp}}
\def\VT{{\sf VT}}
\def\w{{\sf w}}
\def\ws{winning strategy}
\def\y{{\sf y}}
\def\Z{{\mathbb{Z}}}
\def\D{{\mathfrak{D}}}

\def\T{{\bf T}}

\def\T{{\bf T}}

\def\Ra{{\sf Ra}}
\def\set#1{ \{#1\}}
\def\Mo{{\sf Mo}}
\def\restr #1{{\restriction_{#1}}}
\def\ws{winning strategy}

 \def\CA{{\sf CA}}
\def\B{{\sf B}}
\def\G{{\sf G}}
\def\w{{\sf w}}
\def\y{{\sf y}}
\def\g{{\sf g}}

\def\r{{\sf r}}
\def\K{{\sf K}}

\def\pa{$\forall$}
\def\pe{$\exists$}

\def\ef{Ehren\-feucht--Fra\"\i ss\'e}

\def\tp{{\sf tp}}

\def\M{{\mathfrak{M}}}

\def\Ca{{\mathfrak{Ca}}}

\def\M{{\mathfrak{M}}}

\def\A{{\mathfrak{A}}}
\def\B{{\mathfrak{B}}}
\def\C{{\mathfrak{C}}}
\def\D{{\mathfrak{D}}}
\def\Ig{{\mathfrak{Ig}}}
\def\M{{\mathfrak{M}}}
\def\dom{{\sf dom}}
\def\rng{{\sf rng}}
\def\Rd{{\sf{Rd}}}
\def\At{{\sf At}}
\def\Ra{{\mathfrak{Ra}}}
\def\Tm{{\mathfrak{Tm}}}
\def\Cm{{\mathfrak{Cm}}}
\def\E{{\mathfrak{E}}}
\def\Bl{{\mathfrak{Bl}}}

\def\VT{{\sf VT}}
\def\ls { L\"owenheim--Skolem}
\def\Gp{{\sf Gp}}

\def\ef{Ehren\-feucht--Fra\"\i ss\'e}

\def\Id{{\sf Id}}

\def\F{{\mathfrak{F}}}
\def\Sc{{\mathfrak{Sc}}}
\def\QEA{{\sf QEA}}

\def\RCA{{\sf RCA}}
\def\QEA{{\bf PEA}}
\def\PA{{\bf PA}}
\def\R{{\sf R}}
\def\L{{\sf L}}

\def\Df{{\sf Df}}
\def\Uf{{\sf Uf}}
\def\Rd{{\mathfrak{Rd}}}
\def\PEA{{\sf QEA}}

\def\s{{\sf s}}
\def\Sc{{\sf Sc}}
\def\nodes{{\sf nodes}}

\def\c{{\sf c}}

\def\T{{\sf T}}
\def\QEA{{\sf PEA}}

\def\cyl#1{{\sf c}_{#1}}

\def\cyl#1{{\sf c}_{#1}}

\def\diag#1#2{{\sf d}_{#1#2}}

\def\V{{\sf V}}
\def\de{Dedekind-MacNeille}
\def\PA{{\sf PA}}
\def\QA{{\sf QA}}

\def\Z{{\mathbb{Z}}}
\def\Lf{{\sf Lf}}
\def\Cs{{\sf Cs}}
\def\Ra{{\sf Ra}}
\def\Mo{{\sf M}}
\def\QEA{{\sf QEA}}

\def\RA{{\sf RA}}

\def\G{{\sf G}}

\def\T{{\cal T}}

\def\Nrr{\mathfrak{{Nr}}}
\def\w{{\sf w}}
\def\g{{\sf g}}
\def\y{{\sf y}}
\def\r{{\sf r}}
\def\L{\mathfrak{L}}
\def\G{{\cal G}}
\def\PEA{{\sf PEA}}
\def\R{\mathfrak{R}}
\def\Nr{{\sf Nr}}

\def\Bb{{\sf Bb}}

\def\Bb{\\mathfrak{Bb}}

\def\Nrr{\mathfrak{{Nr}}}
\def\w{{\sf w}}
\def\g{{\sf g}}
\def\y{{\sf y}}
\def\r{{\sf r}}
\def\L{\mathfrak{L}}
\def\G{{\cal G}}
\def\PEA{{\sf PEA}}
\def\R{\mathfrak{R}}
\def\Nr{{\sf Nr}}

\def\Q{{\mathbb{Q}}}
\def\Bb{{\mathfrak Bb}}

\def\notVT{{\sf notVT}}
\def\M{{\sf M}}
\def\Mo{{\sf M}}
\def\Nrr{{\mathfrak{Nr}}}

\makeindex             

\title{A universal approach to Omitting types for various multimodal and quantifier logics} 
\author{Tarek Sayed Ahmed}
\date{}
\begin{document}
\maketitle
\begin{abstract}
We intend to investigate the metalogical property of 'omitting types' for a wide variety of quantifier logics (that can also be seen as multimodal logics 
upon identifying existential quantifiers with modalities syntactically and semantically) 
exhibiting the essence of its  abstract algebraic facet, namely, atom-canonicity, the last reflecting a well known persistence propery in modal logic.
In the spirit of 'universal logic ', with this algebraic abstraction at hand, the omitting types theorem $\sf OTT$ will be studied for various reducts extensions and 
variants (possibly allowing formulas of infinite length) of first order logic. 
In the course of our investigations, both negative and positive results will be presented. For example, we show that for any countable theory $L_n$ 
theory $T$ that has quantifier elimination $< 2^{\omega}$ many non-principal complete types can be omitted.
Furthermore, the maximality (completeness) condition, if eliminated, leads to an independent statement from $\sf ZFC$ implied by Martin's axiom.  
$\sf OTT$s are approached  for other algebraizable (in the classical Blok-Pigozzi sense) reformulations or/ and versions of $L_{\omega \omega}$; 
they are shown to hold for some and fail for others. 
\end{abstract}
\section{Introduction}

{\bf Persistence properties and omitting types:} The technical notion of a modal logic corresponds to the one of a variety of {\it Boolean algebras with operators} ($\sf BAO$s) 
which provides {\it algebraic semantics} for modal logic.
We assume familiarity with the very basics of the well developed duality theory between $\sf BAO$s
and
multimodal logic; the class of all $\sf BAO$s correspond to the minimal normal multimodal logic; this
correspondence is established by forming quotient Tarski-Lindenbaum algebras.
The starting point of this duality is that algebraic terms correspond to modal formulas. By that identification we get:
$\F\models \phi \Longleftrightarrow  \Cm\F\models \phi=1,$
where $\F$ is a relational structure (Kripke frame)
$(F, R_i)_{i\in I}$, $I$ a non-empty indexing set, and $\Cm\F$ (its complex algebra) is an algebra having signature
$(f_i:i\in I)$ where each $f_i$ is a unary modality, in other words
an operator. Prominent examples of $\sf BAO$s are relation, cylindric and polyadic algebras.  
Relation algebras ($\sf RA$) correspond to so--called arrow logic, while cylindric algebras of dimension $n$ ($\CA_n$) and the relativized versions of the representable $\CA_n$s,  
correspond to $L_n$
and its guarded and clique-guarded fragments \cite{HHbook} dealt with below The topic of this paper is typical of what  happens in algebraic logic and that is: Using algebraic machinery to obtain fine results in first order logic via so called 'bridge theorems'.  
The algebraic machinery we use, is in essence a disguised form of Andreka's splitting \cite{Andreka} wrapped up in, as indicated in the title, a blow up and blur construction. The last is an indicative term invented by Andre\'ka and N\'emeti in \cite{ANT} and has several reincarnations in the literatue though not under this name. One purpose of this paper is to gather such instances under the title (of this very paper): Blow up and blur constructions in algebraic logic. 
The metalogical result our investigations have impact on is the celebrated Henkin-Orey omitting types theorem.  
Let $\L$ be an extension or reduct or variant of first order logic,  like first logic itself, $L_n$ as defined in the abstract with $2<n<\omega$,  
$L_{\omega_1, \omega}$, $L_{\omega}$ as defined in \cite[\S 4.3]{HMT2}, $\ldots$,  etc. 
An omitting types theorem for $\L$, briefly an $\sf OTT$, is typically of the form  `A countable family of non--isolated types in a countable $\L$ theory 
$T$ can be omitted in a countable model of $T$.  From this it directly follows that if a  type is realizable in every model of a countable theory $T$, 
then there should be a formula consistent with $T$ that isolates this
type. A type is simply a set of formulas $\Gamma$ say.  The type $\Gamma$ is realizable in a model 
if there is an assignment that satisfies (uniformly) all formulas in $\Gamma$. 
Finally, $\phi$ isolates $\Gamma$ means that $T\vdash \phi\to \psi$ 
for all $\psi\in \Gamma$. What Orey and Henkin proved is that  the $\sf OTT$ 
holds for $L_{\omega, \omega}$ when such types are {\it finitary}, meaning that they all consist of $n$-variable 
formulas for some $n<\omega$.
$\sf OTT$ has an algebraic facet exhibited in the property of {\it atom-canonicity}; 
which in turn reflects an important {\it persistence} property in modal logic.
Algebraically,  so--called {\it persistence properties} refer to closure of a variety
$\V$ under passage from a given algebra $\A\in \V$ to some `larger' algebra $\A^*$.
Canonicity,  which is the most prominent persistence property in modal logic, the `large algebra' $\A^*$ is the canonical embedding algebra (or perfect) extension of $\A$, a complex algebra based on the 
{\it ultrafilter frame} of $\A$ whose underlying set is the set of all Boolean ultrafilters of $\A$. 
A completely additive variety $\sf V$ is {\it atom-canonical} if whenever $\A\in \sf V$ is atomic, then the complex algebra of its atom structure, in symbols $\Cm\At\A$, is also in $\sf V$. More concisely, $\sf V$ is such if  
$\Cm\At\sf V\subseteq \sf V$.
Atom-canonicity is concerned with closure under forming \de\ completions (sometimes occuring in the literature under the name of {\it the minimal completions}) 
of atomic algebras in the variety $\sf V$, because for an atomic $\A\in \sf V$, $\Cm\At\A$ is 
its \de\ completion. 
Though $\RCA_n$ is canonical, it is not atom-canonical for $2<n<\omega$ \cite{Hodkinson}. From non-atom-canonicity of $\RCA_n$, it follows from \cite{Venema} that $\sf RCA_n$ 
cannot be axiomatized by Sahlqvist equations. 
We shall see that (non-) atom-canonicity of subvarieties of $\RCA_n$ is closely related to (the failure) of some version of the $\sf OTT$ 
in modal fragments of $L_n$. While the classical Orey-Henkin $\sf OTT$ holds for $L_{\omega, \omega}$, it is known \cite{ANT} that the $\sf OTT$ fails for $L_n$ in the following  (strong) sense. 
For every $2<n\leq l<\omega$, there is a countable and complete $L_n$ atomic theory $T$, and a single
type, namely, the type consisting of co-atoms of $T$,  that is realizable in every model of $T$, but cannot be isolated by a formula 
$\phi$ using $l$ variables. Such $\phi$ will be referred to henceafter as a {\it witness}.
Here we prove stronger negative $\sf OTT$s for $L_n$ when types are required to be omitted  with respect to 
certain (much wider) generalized semantics, called {\it $m$-flat} and {\it $m$--square} with $2<n<m<\omega$. 
These locally relativized representations were invented by Hirsch and 
Hodkinson in the context of relation algebras \cite[Chapter 13]{HHbook} and were adapted to $\CA$s in \cite[\S5]{mlq} which is the  form we stick to  in this paper. Considering such 
{\it clique-guarded} semantics 
swiftly leads us to rich territory.

{\bf Complete representations and $\sf VT$:}  An important semantical notion in the theory of both relation and cylindric algebras closely related to a restricted version of $\sf OTT$ which we refer to  as Vaught's Theorem $\sf VT$,
is that of {\it complete representability}. 
 One of the old classical  basic theorems in young  model theory,  $\sf VT$ says that countable atomic theories have counbtable atomic (equivalently in this context prime) models. A prime model for a theory $T$ is a 'small model', in the sense that  it elementay embeds into any other model of $T$. Vaught proved his $\sf VT$ by a farily straightforward application of the Orey-Henkin $\sf OTT$ for $L_{\omega, \omega}$. 
These connections will be explicity revealed as the paper unfolds. 
A {\it (relativized) complete representation} of $\A\in \CA_n$ on $V\subseteq {}^nU$, is a (relativized) representation of $\A$ on $V$ via $f$, that preserves  arbitrary sums carrying
them to set--theoretic unions, that is the representation $f:\A\to \wp(V)$ is required to satisfy   
$f(\sum S)=\bigcup_{s\in S}f(s)$ for all $S\subseteq \A$ such that $\sum S$ exists. In this case, we 
say that $\A$ is {\it completely represented on $V$ via $f$}, or simply completely represented on $V$. 
If $\A\in \CA_n$, then a relativized representation of $\A$ on $V$ via $f$ is not necessarily a {\it complete} relativized  representation of $\A$ on $V$ via (the same) $f$. 
It is known that if $\A\in \CA_n$ has
a relativized complete representation on $V$ $\iff$
the Boolean reduct of $\A$ is atomic (below every non-zero element of $\A$ there is an atom, that is, 
a minimal non-zero element), and that $f$ is {\it atomic} in the sense that
$\bigcup_{x\in \At\A}f(x)=V$ \cite{HH}.  
It is also known that there are countable and atomic $\RCA_n$s,
that have no complete representations on generalized cartesian spaces of dimension $n$. So atomicity is necessary but not sufficient for complete 
representability of $\CA_n$s on generalized cartesian spaces of the same dimension $n$.
In fact, the class of completely representable $\CA_\alpha$s on cartesian spaces of dimension $\alpha$, briefly ${\sf CRCA}_{\alpha}$ 
when $\alpha>2$ (where generalized cartesian spaces for infinite dimensions is defined like the finite dimensional case), 
is not even elementary, cf.\cite{HH} and  \cite[Corollary 3.7.1]{HHbook2}.

The subtle semantical phenomena of (relativized) complete representability is closely
related to the algebraic notion of  atom--canonicity of (certain supervarieties of) $\RCA_n$ 
(like $\bold S\Nr_n\CA_m$ for $2<n<m<\omega$),
and to the metalogical property of  omitting types in $n$--variable fragments of first order logic
\cite[Theorems 3.1.1-2, p.211, Theorems 3.2.8, 9,10]{Sayed},
when non--principal types are omitted with respect to (relativizing) Tarskian `square' semantics. 
The typical question is: given a $\A\in \CA_n$ and a family $(X_i: i\in I)$ of subsets of $\A$ ($I$ a non--empty set),  such that $\prod X_i$ exists in $\A$ for each $i\in I$
is there a representation $f:\A\to \wp(V)$, on some $V\subseteq {}^nU$ ($U$ a non--empty set), 
that carries  this set of meets to set theoretic intersections, in the sense that  $f(\prod X_i)=\bigcap_{i\in I}f(x)$ for all $i\in I$? 
When the algebra $\A$ is countable, $|I|\leq \omega$ and $\prod X_i=0$ for all $\in I$, this is an algebraic version of an  omitting types theorem; 
the representation $f$ {\it omits the given set of meets (or non-principal types) on $V$.}  
When it is {\it only one meet} consisting of co-atoms, in an atomic algebra, such a representation $f$ will 
be a {\it complete representation on $V$}, and 
this is equivalent to that $f(\prod X)=\bigcap_{x\in X}f(x)$ for all $X\subseteq \A$ 
for which $\prod X$ exists in $\A$ \cite{HH}. The last condition, when we require that $V$ is a generalized cartesian space of dimension $n$, is an algebraic version 
of Vaught's theorem for first order logic.  In this case the complete representation $f$ provides an `atomic model' $V$ of the  
`atomic theory' $\A$ that omits all 
non--principal types, where a non--principal type is a set $X\subseteq \A$, such that $\prod^{\A} X=0$. 
It is well known that for first order logic Vaught's theorem is a consequence of the classical Orey--Henkin omitting types theorem
by omitting the non-principal type consisting of co-atoms in $\Nrr_n\A$ for all $n\in \omega$ 
where $\A$ is (isomorphic to) $\Fm_T$ and $T$ is a set of sentences  that is the {\it countable} atomic theory at hand; here $\Nrr_n\A$ is the neat $n$-reduct of $\A$, cf. \cite[Definition 2.6.28]{HMT2}. In this particular case 
$\sf \Nrr_n\A$ is iomorphic to the quotient (relative to $T$)  algebra of formulas that are equivalent modulo $T$ to formulas containing at most $n$ free- variables.

We show that the $\sf OTT$ for $L_n$ takes on an entire different subtle route; it fails dramatically. 
we shall violate $\sf VT$ for so-called $m$-clique guarded fragments of $L_n$ whose semantics are determined by $m$-square 
and $m$-flat models which are local `$m$  approximations' of ordinary models the last being $\omega$-square. 
Roughly an $m$ square-model $\sf M$ for an $L_n$  theory or an $m$-square representation of a $\CA_n$ (with base) $\sf M$ - upon identifying models of an $L_n$  theory $T$ say  with representations  of its Tarski-Lindenbaum algebra $\Fm_T$ - is a representation  $\sf M$ 
for which when one approaches $\Mo$  with  a 'movable window', then there will become a point determined by $m$ where one mistakes this locally relativized approximate representation 
with an ordinary one.  The discrepancy reappears as soon as one zooms out; the zooming out is inversely proportional to the`genuine  squareness' of the model. 
A semantical additional condition on commutativity of cylindrifications on $m$ squares gives $m$-flat models which are thus better approximations to ordinay models  on a local  level. 
\footnote{Worthy ot note is that for $n<m$, the variety consisting of algbras having $m$-flat models 
coincide with $\bold S\Nr_n\CA_m$ and that for $m\geq n+3$ the last 
is not finitely axiomatizable over that having $m$-square ones.
 Furthermore, $\RCA_n$ is not finitely axiomatizable over the last two varieties and for any positive $k\geq 1$,  $\bold S\Nr_n\CA_{n+k+1}$ is not finitelt axiomatizable over $\bold S\Nr_n\CA_{n+k}$.}
The proof methods resorted to so-called {\it blow up and blur} constructions to prove {\it non}-atom canonicity of several varieties of relation and cylindric algebras. 
We recall that a class $\sf K$ of Boolean algebras with operators ($\sf BAO$s) is {\it atom--canonical} if whenever $\A\in \sf K$ is atomic and completey additive, then its \de\ completion, 
namely, the complex algebra of its atom structure, namely,  $\Cm\At\A$ is also in $\sf K$.

{\bf Blow up and blur constructions and splitting atoms:} This subtle construction may be applied to any two classes  $\bold L\subseteq \bold K$ of completely additive $\sf BAO$s.
One takes  an atomic $\A\notin \bold K$ (usually but not always finite), blows it up, by splitting one or more of its atoms each to infinitely many subatoms,
obtaining an (infinite) {\it countable} atomic ${\Bb}(\A)\in \bold L$,  such that $\A$ {\it is blurred} in ${\Bb}(\A)$ meaning that  $\A$ {\it does not} embed in 
${\Bb}(\A)$, but $\A$ embeds in the \de\ completion of ${\Bb}(\A)$,
namely, $\Cm\At{\Bb}(\A)$. 
Then any class $\bold M$ say, between $\bold L$ and $\bold K$ that is closed under forming subalgebras  
will not be atom--canonical, for ${\Bb}(\A)\in \bold L(\subseteq \bold M)$, but $\Cm\At{\Bb}(\A)\notin \bold K(\supseteq \bold M)$ because $\A\notin \bold M$ and $\bold S\bold M=\bold M$. 
We say, in this case, that {\it $\bold L$ is not atom--canonical with respect to $\bold K$}.
This method is applied to $\bold K=\bold S\Ra\CA_l$, $l\geq 5$ and $\bold L=\sf RRA$ in \cite[\S 17.7]{HHbook}  and 
to $\bold K=\sf RRA$ and $\bold L={\sf RRA}\cap \Ra\CA_k$ for all $k\geq 3$ in \cite{ANT},
and will applied below to $\bold K=\bold S\Nr_n\CA_{n+k}$, $k\geq 3$
and $\bold L=\RCA_n$, where  $\Ra$ denote the operator of forming 
relation algebra reducts (applied to classes) of $\CA$s, respectively, \cite[Definition 5.2.7]{HMT2}.
The idea of splitting one or more atoms in an  algebra to get a (bigger) superalgebra tailored to a certain purpose   
seems to originate with Henkin \cite[p.378, footnote 1]{HMT2}.
Let $2<n<\omega$. Monk proved his seminal result that 
$\RCA_n$ 
is not finitely axiomatizable by constructing finite non--representable 
algebras (referred to together with variations thereof in the literature as Monk--like algebras), whose ultraproduct is representable.
The idea involved in the construction of a non--representable finite Monk (--like) $\A\in \CA_n$s is not so hard.
Such $\A$ is finite, hence atomic, more precisely its Boolean reduct is  atomic.
The algebra $\A$ is obtained by splitting some atoms in a finite $\CA_n$ each into one or more subatoms.
The new atoms are given colours, and
cylindrifications and diagonals are re-defined by stating that monochromatic triangles
are inconsistent. If the atoms resulting after splitting are `enough', that is, a Monk's algebra has many more atoms than colours,
it follows by using a fairly standard form of 
Ramsey's Theorem that any representation of $\A$ will contain a monochromatic triangle, 
so $\A$, 
by definition, cannot be representable. 
Maddux modified Monk's algebras using $\CA_n$s based on atomic 
relation algebras generated by a single element. This improvement is highly rewarding for it transfers incompleteness theorems 
for $n$--variable first order logic from languages having
countably many $n$--ary relation symbols to incompleteness theorems for languages in a signature having only one binary relatuon symbol.
Furthermore, the modification is far from being trivial for one generation 
makes the automorphism group of the algebra rigid. The results of Monk and Maddux will reproved below.
In the cylindric paradigm, And\'reka modified such splitting methods {\it re-inventing} (Andr\'eka's) splitting. 
In this new setting,   Andr\'eka proved a plethora of 
{\it relative non--finite axiomatizability results}.
In Andr\'eka's splitting one typically splits an atom in a given representable (set) algebra to finitely many subatoms $m$ say, obtaining a bigger non--representable superalgebra
$\A_m$ whose subalgebras generated by $<m$ elements are representable, and various proper reducts of $\A_m$ remain representable. Iterating such a splitting argument for $m$, with $m$ getting
abitrarily large,  excludes universal finite variable axiomatizations of the variety of representable algebras at hand, over such strict subreducts. 
Here, besides dealing with a variant of  Andr\'eka' splitting,  
we present other constructions such as the so--called blow up and blur constructions which are splitting arguments at heart involving splitting (atoms) in 
finite Monk--like algebras and  rainbow algebras. 

As another application to splitting methods for both cylindric and relation algebras, together with  a rainbow construction, we also prove the following new results. 
Let ${\sf CRCA}_n$ denote the class of completely representable $\CA_n$s, ${\sf CRRA}$ 
denote the class of completely representable $\RA$s, $\bold S_c$ denote the operation of forming complete subalgebras, and
$\bold S_d$ denote the operation of forming dense subalgebras.
We show that any class $\bold K$ such that $\Nr_n\CA_{\omega}\cap {\sf CRCA}_n\subseteq \bold K\subseteq \bold S_c\Nr_n\CA_{n+3}$,  $\bold K$ is not elementary,
and that for any class $\bold L$ such that $\bold S_d\Ra\CA_{\omega}\cap {\sf CRRA}\subseteq \bold L\subseteq \bold S_c\Ra\CA_{5}$,  
$\bold L$ is not elementary
Furthermore, all  constructions used to violate in this paper resorted to  a variation on the following single theme typical of Andr\'eka's splitting:
{\it Split some (possibly all) atoms in an  algebra (that need not be finite nor even atomic)  
each into one or more subatoms forming a bigger superalgebra that constitutes the starting point 
for a construction serving  the purpose at hand. For cylindric--like algebras of dimension $\alpha$, $\alpha$ an ordinal, if the atom $x\in \A$ is split to the 
subatoms $X=(x_i: i\in I)$ in $\B\supseteq \A$, then $X$ is a partition of $x$ in the sense that $x_l\cdot x_m=0$ for $l\neq m\in I$, 
and $\sum_{i\in I}^{\B}x_i=x$. Furthermore, for each $i\in I$, $x_i$ is {\it cylindrically equivalent} to $x$, in the sense 
that for all $j<\alpha$, $\c_j^{\B}x_i=\c_j^{\B}x$.}

This roughly means that cylindrifiers, the most prominent citizens in such algebras, cannot distinguish between an atom and its 
splitted subatoms. So splitting, leading from $\A$ to $\B$ 
does not ruin the algebraic structure dramatically, at least as far as cylindrifiers are concerned. But at the same time, we show that 
the subtle technique of splitting, undetected by cylindrifiers, can lead from a representable $\A$ 
to a non--representable 
$\B$ (Andr\'eka's splitting) and conversely from a non--representable $\A$ to a representable $\B$ (blow up and blur constructions). 
A third case encountered below is that splitting does not alter the representability status of $\A$ at all; 
$\A$ is representable $\iff \B$ is 
representable.

{\bf Omitting types for finite variable fragments:} We obtain negative results of the form:
{\it There exists a countable, complete and  atomic $L_n$ first order theory $T$ in a signature $L$ such that  the type $\Gamma$ consisting of co-atoms in the  cylindric Tarski-Lindenbaum quotient algebra $\Fm_T$ is realizable in every {\it $m$--square} model, 
but $\Gamma$ cannot be isolated using $\leq l$ variables, where $n\leq l<m\leq  \omega$.} A co-atom of $\Fm_T$ is the negation of an atom in $\Fm_T$, 
that is to say, is an element of the form $\Psi/\equiv_T$, where $\Psi/\equiv_T=(\neg \phi/{\equiv_T})=\sim (\phi/{\equiv_T})$ and $\phi/{\equiv_T}$ is an atom in $\Fm_T$ (for $L$-fomulas, $\phi$ and $\psi$).
Here the quotient algebra $\Fm_T$is formed relative to the congruence relation of semantical equivalence modulo $T$.
An  $m$-square model of $T$ is an $m$-square representation of $\Fm_T$.
The last statement  denoted by $\Psi(l, m)$, short for Vaught's Theorem ({\it $\sf VT$) fails at (the parameters) $l$ and $m$.}
Let ${\sf VT}(l, m)$ stand for {\it {\sf VT} holds at $l$ and $m$}, so that by definition $\Psi(l, m)\iff \neg {\sf VT}(l, m)$. We also include $l=\omega$ in the equation by 
defining ${\sf VT}(\omega, \omega)$ as {\sf VT} holds for $L_{\omega, \omega}$: Atomic countable first order theories have atomic countable models. 
We provide strong evidence that $\sf VT$ {\it fails everywhere} in the sense that for the permitted values $n\leq l, m\leq \omega$, namely, 
for $n\leq l<m\leq \omega$  and $l=m=\omega$, 
$\VT(l, m)\iff l=m=\omega.$  
 For example, from the non--atom canonicity of $\RCA_n$ with respect to the variety of $\CA_n$s having $n+3$--square representations ($\supseteq \bold S\Nr_n\CA_{n+3}$), 
we proved $\Psi(n, n+k)$ for $k\geq 3$ and from 
the non--atom canonicity of $\Nr_n\CA_{n+k}\cap \RCA_n$ with respect  to $\RCA_n$
for all $k\in \omega$, 
we proved $\Psi(l, \omega)$ for all finite $l\geq n$. 
Both results are obtained by blowing up and blurring finite algebras; a rainbow $\CA_n$ in the former case, and a finite 
$\sf RA$ (whose number of atoms depend on $k$) in the second case.
In this case, we said (and proved) that   $\sf VT$ fails {\it almost} everywhere.  
The non atom--canonicity of $\Nr_n\CA_{m-1}\cap \RCA_n$  with respect to the variety of $\CA_n$s having $m$--square representations ($\supseteq \bold S\Nr_n\CA_m$)
for all $2<n<m<\omega$, implies that $\Psi(l, m)$ holds for all $2<n\leq l<m\leq \omega$, in which case 
we can say that {\it $\sf VT$ fails everywhere.}
The latter was reduced to (finding then) blowing up  and blurring a finite relation algebra $\R$ 
that is `strongly blurred up to $m-1$' (a notion defined below)   
but $\R$ and no $m$-dimensional relational basis, for each $2<n<m<\omega$.  
Figuratively speaking, $\sf VT$ holds only at the limit 
when $l\to \infty$ and $m\to  \infty$. So we can express the situation (using elementary Calculas terminology) as follows: 
For $2<n\leq l<m<\omega$, ${\sf VT}(l, m)$ is false, but  as $l$ and $m$ gets larger, ${\sf VT}(l, m)$ gets closer to $\sf VT$, in symbols,
${\sf lim}_{l,m\to \infty}{\sf VT}(l, m)={\sf VT}({\sf lim}_{\l\to \infty}l, {\sf lim}_{m\to \infty}m)={\sf VT}(\omega, \omega)$.
The double limit $lim_{l,m\to \infty}{\sf VT}(l,m)={\sf VT}(\omega, \omega)$ 
can be given a precise algebraic expression using ultraproducts. The limit on the left was reincarnated as a sequence of non-representable  
$\CA_n$s, generated by a single $2$ dimensional element,  whose ultraproduct 
converges to an algebra satisfying the so-called Lyndon conditions, which is the reincarnation of the right hand side of the equation, namely, $\sf VT(\omega, \omega)$.
This reproves the classical results of Monk  and Maddux refined by Biro
on non-finite axiomatizability of $\sf RRA$ and $\RCA_n$ and their finite first order definable expansions \cite{maddux, Biro}. 
A first order definable expansion of an ${\sf RCA}_n$ is an expansion where the semantics of 
the finite extra operations in set algebras (corresponding to models) are stimulated using the semantics of first order formulas (in such models). 
Our construction also  proved that ${\sf LCA}_n$ is not finitely axiomatizable too, reproving a result of Hirsch and Hodkinson \cite{HHbook}. 
The class ${\sf LCA}_n={\bf El}{\sf CRCA}_n$ is the class of agebras that can be axiomatized by a (necessariy infinite) 
set of first order sentences known as the Lyndon conditions $\rho_m$ ($m\in \omega$), where $\rho_m$ codes a \ws\ for \pe\ in 
a deterministic two player $m$-rounded game between \pe\ and \pa\, played on so-called atomic networks of an an atomic $\CA_n$  cf. \cite{HHbook2}. 
Variations on such games were also 
used to characterize the semantical notions of $m$-flatness and $m$-squareness.
We use throughout the paper fairly standard or/and self-explanatory notation
following mainly  the notation of \cite{1,HMT2}.

{\bf On the notation:} We use throughout the paper fairly standard or/and self-explanatory notation
following mainly  the notation of \cite{1,HMT2}. Less usual notation will be introduced at its first occurence in the text.
In the meantime the following  would be helpful. 
We write $\subseteq$ for inclusion, and $\subsetneq$ for
proper inclusion.
Throughout the paper we make the following convention. We denote infinite ordinals by $\alpha, \beta\ldots$ 
and finite ordinals by $n, m\ldots$. Ordinals which are arbitrary meaning that they could be finite or infinite 
will be denoted by  $\alpha,\beta\ldots$. 
Also 
algebras will be denoted by Gothic letters, and when we write $\A$ for an algebra, then we shall be tacitly assuming that 
$A$ denotes its universe, that is 
$\A=\langle A, f_i^{\A}\rangle_{i\in I}$ where $I$ is a non--empty set and $f_i$ $(i\in I)$ are the operations in the signature of $\A$ 
interpreted via $f_i^{\A}$ in $\A$. For better readability, we omit the superscript $\A$ and we write simply 
$\A=\langle A, f_i\rangle_{i\in I}$.
However, in some occasions we will identify (notationally)
an algebra with its universe.
For operators on classes of algebras, $\bold S$ stands for the operation of forming subalgebras, $\bold P$
stands for that of forming products, $\bf H$ for forming homomorphic images, ${\bf Up}$ for forming ultraproducts.
If $I$ is a non--empty set and $U$ is an ultrafilter over $\wp(I)$ and if $\A_i$ is some structure (for $i\in I$) we write either
$\Pi_{i\in I}\A_i/U$ or $\Pi_{i/U}\A_i$ for the ultraproduct of the $\A_i$s over $U$. 
If $\A$ is an algebra and $X\subseteq \A$, we write $\Sg^{\A}X$, or simply $\Sg X$ if $\A$ is clear from context, for the subalgebra of $\A$ generated by $X$.
If $\A\in \V$ where $\V$ is a class of $\sf BAO$s, then ${\sf Uf}\A$ denotes its {\it ultrafilter atom structure} whose underlying set is the set of Boolean ultrafilters of $\A$. 
The {\it canonical extension} of $\A$ say,  denoted by $\A^+$, is the complete algebra $\Cm{\sf Uf}\A$. It is known that $\A$ embeds into $\A^+$ v
ia $a\mapsto \{F\in {\sf Uf}\A: a\in F\}$.
We write $\prod^{\A}X (\sum ^{\A}X )$ for the infimum (supremum) of $X$ in $\A$, if it exists. Often, however,  we omit the superscript $\A$ if it is clear from context.

\section{The algebras and some basic concepts}
For a set $V$, ${\cal B}(V)$ denotes the Boolean set algebra $\langle \wp(V), \cup, \cap, \sim, \emptyset, V\rangle$.
Let $U$ be a set and $\alpha$ an ordinal; $\alpha$ will be the dimension of the algebra.
For $s,t\in {}^{\alpha}U$ write $s\equiv_i t$ if $s(j)=t(j)$ for all $j\neq i$.
For $X\subseteq {}^{\alpha}U$ and $i,j<\alpha,$ let
$${\sf C}_iX=\{s\in {}^{\alpha}U: (\exists t\in X) (t\equiv_i s)\}$$
and
$${\sf D}_{ij}=\{s\in {}^{\alpha}U: s_i=s_j\}.$$
$\langle{\cal B}(^{\alpha}U), {\sf C}_i, {\sf D}_{ij}\rangle_{i,j<\alpha}$ is called {\it the full cylindric set algebra of dimension $\alpha$}
with unit (or greatest element) $^{\alpha}U$. Any subalgebra of the latter is called a {\it set algebra of dimension $\alpha$}.
Examples of subalgebras of such set algebras arise naturally from models of first order theories.
Indeed, if $\Mo$ is a first order structure in a first
order signature $L$ with $\alpha$ many variables, then one manufactures a cylindric set algebra based on $\Mo$ as follows, cf. \cite[\S4.3]{HMT2}.
Let
$$\phi^{\Mo}=\{ s\in {}^{\alpha}{\Mo}: \Mo\models \phi[s]\},$$
(here $\Mo\models \phi[s]$ means that $s$ satisfies $\phi$ in $\Mo$), then the set
$\{\phi^{\Mo}: \phi \in Fm^L\}$ is a cylindric set algebra of dimension $\alpha$, where $Fm^L$ denotes the set of first order formulas taken in 
the signature $L$.  To see why, we have:
\begin{align*}
\phi^{\Mo}\cap \psi^{\Mo}&=(\phi\land \psi)^{\Mo},\\
^{\alpha}{\Mo}\sim \phi^{\Mo}&=(\neg \phi)^{\Mo},\\
{\sf C}_i(\phi^{\Mo})&=(\exists v_i\phi)^{\Mo},\\
{\sf D}_{ij}&=(x_i=x_j)^{\Mo}.
\end{align*}
Following \cite{HMT2}, $\Cs_{\alpha}$ denotes the class of all subalgebras of full set algebras of dimension $\alpha$.
The (equationally defined) $\CA_{\alpha}$ class is obtained from cylindric set algebras by a process of abstraction and is defined by a {\it finite} schema
of equations given in \cite[Definition 1.1.1]{HMT2} that holds of course in the more concrete set algebras. 
\begin{definition} Let $\alpha$ be an ordinal. By \textit{a cylindric algebra of dimension} $\alpha$, briefly a
$\CA_{\alpha}$, we mean an
algebra
$$ {\A} = \langle A, +, \cdot,-, 0 , 1 , {\sf c}_i, {\sf d}_{ij}\rangle_{\kappa, \lambda < \alpha}$$ where $\langle A, +, \cdot, -, 0, 1\rangle$
is a Boolean algebra such that $0, 1$, and ${\sf d}_{i j}$ are
distinguished elements of $A$ (for all $j,i < \alpha$),
$-$ and ${\sf c}_i$ are unary operations on $A$ (for all
$i < \alpha$), $+$ and $.$ are binary operations on $A$, and
such that the following equations  are satisfied for any $x, y \in
A$ and any $i, j, \mu < \alpha$:
\begin{enumerate}
\item [$(C_1)$] $  {\sf c}_i 0 = 0$,
\item [$(C_2)$]$  x \leq {\sf c}_i x \,\ ( i.e., x + {\sf c}_i x = {\sf c}_i x)$,
\item [$(C_3)$]$  {\sf c}_i (x\cdot {\sf c}_i y )  = {\sf c}_i x\cdot  {\sf c}_i y $,
\item [$(C_4)$] $  {\sf c}_i {\sf c}_j x   = {\sf c}_j {\sf c}_i x $,
\item [$(C_5)$]$  {\sf d}_{i i} = 1 $,
\item [$(C_6)$]if $  i \neq j, \mu$, then
 ${\sf d}_{j \mu} = {\sf c}_i
 ( {\sf d}_{j i} \cdot  {\sf d}_{i \mu}  )  $,
\item [$(C_7)$] if $  i \neq j$, then
 ${\sf c}_i ( {\sf d}_{i j} \cdot  x) \cdot   {\sf c}_i
 ( {\sf d}_{i j} \cdot  - x) = 0
 $.
\end{enumerate}
\end{definition}
Our main results involve the central notion of neat reducts:
\begin{definition} 
Let  $\alpha<\beta$ be ordinals and $\B\in \CA_{\beta}$. Then the {\it $\alpha$--neat reduct} of $\B$, in symbols
$\Nrr_{\alpha}\B$, is the
algebra obtained from $\B$, by discarding
cylindrifiers and diagonal elements whose indices are in $\beta\sim \alpha$, and restricting the universe to
the set $Nr_{\alpha}B=\{x\in \B: \{i\in \beta: {\sf c}_ix\neq x\}\subseteq \alpha\}.$
\end{definition}
Let $\alpha$ be any ordinal.  If $\A\in \CA_\alpha$ and $\A\subseteq \Nrr_\alpha\B$, with $\B\in \CA_\beta$ ($\beta>\alpha$), 
then we say that $\A$ {\it neatly embeds} in $\B$, and 
that $\B$ is a {\it $\beta$--dilation of $\A$}, or simply a {\it dilation} of $\A$ if $\beta$ is clear 
from context.  
For $\sf K\subseteq {\sf CA}_\beta$, and $\alpha<\beta$, ${\sf Nr}_\alpha{\sf K}=\{\Nrr_\alpha\B: \B\in {\sf K}\}\subseteq \CA_\alpha$.

We shall have the occasion to deal with (in addition to $\CA$s), the following cylindric--like algebras \cite{1}: $\sf Df$ short for diagonal free cylindric algebras, $\sf Sc$ short for Pinter's substitution algebras,  
$\sf QA$($\QEA$) short for quasi--polyadic (equality) algebras, $\sf PA(\sf PEA)$ short for polyadic (equality) algebras. For $\K$ any of these classes and $\alpha$ any ordinal, 
we write $\K_{\alpha}$ for variety of $\alpha$--dimensional $\K$ algebras which can be axiomatized by a finite schema of equations, 
and $\sf RK_{\alpha}$ for the class of  representable $\K_{\alpha}$s, which happens to be a variety too (that cannot be axiomatized by a finite schema of equations for $\alpha>2$ unless $\sf K=\sf PA$ and $\alpha\geq \omega$).
The standard reference for all the classes of algebras mentioned previously  is  \cite{HMT2}. 
We recall the concrete verions of such algebras.
Let $\tau:\alpha\to \alpha$
and $X\subseteq {}^{\alpha}U,$  then 
$${\sf S}_{\tau}X=\{s\in {}^{\alpha}U: s\circ \tau\in X\}.$$
For $i,j\in \alpha$, $[i|j]$ is the replacement on $\alpha$ that sends $i$ to $j$ and is the identity map on $\alpha\sim \{i\}$ while $[i,j]$ is the transposition on $\alpha$ that interchanges $i$ and $j$. 
\begin{itemize}
\item A {\it diagonal free cylindric set algebra of dimension $\alpha$} is an algebra of the form 
$\langle \B(^{\alpha}U),  {\sf C}_i\rangle_{i,j<\alpha}.$ 
\item A {\it quasi-polyadic set algebra of dimension $\alpha$} is an algebra of the form\\ 
$\langle \B(^{\alpha}U),  {\sf C}_i,  {\sf S}_{[i|j]}, {\sf S}_{[i,j]}\rangle_{i,j<\alpha}.$ 
\item A {\it quasi-polyadic equality set algebra} is an algebra of the form\\ 
$\langle \B(^{\alpha}U),  {\sf C}_i,  {\sf S}_{[i|j]}, {\sf S}_{[i,j]}, {\sf D}_{ij}\rangle_{i,j<\alpha}$.
\item A {\it polyadic set algebra of dimension $\alpha$} is an algebra of the form\\ 
$\langle \B(^{\alpha}U),  {\sf C}_i,  {\sf S}_{\tau}\rangle_{\tau:\alpha\to \alpha}.$  
\item A {\it polyadic equality set algebra of dimension $\alpha$} is an algebra of the form\\ 
$\langle \B(^{\alpha}U),  {\sf C}_i,  {\sf S}_{\tau}\rangle_{\tau:\alpha\to \alpha, i, j<\alpha}$  
\end{itemize}

\begin{figure}
\[\begin{array}{l|l}
\mbox{class}&\mbox{extra non-Boolean operators}\\
\hline
\Df_{\alpha}& \cyl i: i<\alpha\\
\Sc_\alpha&\cyl i, \s_i^j :i, j<\alpha\\
\CA_\alpha&\cyl i, \diag i j: i, j<\alpha\\
\PA_\alpha&\cyl i, \s_\tau: i<n,\; \tau\in\;^\alpha\alpha\\
\PEA_\alpha&\cyl i, \diag i j,  \s_\tau: i, j<n,\;  \tau\in\;^\alpha\alpha\\
\QA_\alpha&  \cyl i, \s_i^j, \s_{[i, j]} :i, j<\alpha  \\
\QEA_\alpha&\cyl i, \diag i j, \s_i^j, \s_{[i, j]}: i, j<\alpha
\end{array}\]
\caption{Non-Boolean operators for the classes\label{fig:classes}}
\end{figure}

Let $\alpha$ be an ordinal. For any such abstract class of algebras $\sf K_{\alpha}$ in the above table, $\sf RK_{\alpha}$ is defined to be the subdirect product of set algebras of dimension $\alpha$. 
A {\it cartesian square of dimension $\alpha$} is a set of the form ${}^{\alpha}U$ ($U$ some non-empty set); these appear as top elements of ${\sf Cs}_\alpha$s. 
We let ${\sf Gs}_\alpha$ denote the class of {\it generalized set algebras of dimension $n$}; $\A\in {\sf Gs}_\alpha$ $\iff$  $\A$ has top element a disjoint union of cartesian squares of dimension $\alpha$ 
and the cylindric operations are  defined like in set algebras. 
It is known that $\RCA_\alpha=\bold I{\sf Gs}_\alpha$. For $\alpha<\omega$, $\sf PA_{\alpha}(\sf PEA_{\alpha}$) is definitionally equivalent to $\sf QA_{\alpha}(\sf QEA_{\alpha})$ which is is no longer the case for infinite $\alpha$ where the deviation is 
largely significant. 
For example a countable $\QA_{\omega}$ has a countable signature, while a countable $\sf PA_{\omega}$ has an uncountable signature having the same cardinality as (substitutions in) $^\omega\omega$.
The class of completely representable $\sf K_{\alpha}$s ($\sf K$ any of the above classes) is denoted by $\sf CRK_{\alpha}$. 
We recall the definition for $\CA$s of finite dimension. The rest of the cases are defined similarly.
\begin{definition}\label{omit}
Let $n<\omega$.
Then $\A\in \CA_n$ is {\it completely representable}, if there exists $\B\in {\sf Gs}_{n}$ and 
an  isomorphism $f:\A\to \B$
such for all $X\subseteq \A$, $f(\prod X)=\bigcap_{x\in X} f(x)$ whenever $\prod X$ exists.
\end{definition}
If $\A$ is an atomic $\CA_n$, then an isomorphism $f:\A\to \B$, where $\B\in {\sf Gs}_n$  has top element
$V$, is {\it atomic}, if $\bigcup_{a\in \At\A}f(a)=V$.
It can be easily shown that $f$ is a complete representation of $\A\iff$ $\A$ is atomic and $f$ is an atomic representation.
Considering polyadic algebras, we will encounter $\PEA_{\alpha}$ and $\PA_{\alpha}$, $\alpha$ an infinite ordinal  (having all substitutions and infinitary cylindrifications) only once in Theorem \ref{pa}.
We deal mostly with $\QA$s and $\QEA$s.
For a $\sf BAO$, $\A$ say, for any ordinal $\alpha$, $\Rd_{ca}\A$ denotes the cylindric reduct of $\A$ if it has one, $\Rd_{sc}\A$
denotes the $\Sc$ reduct of $\A$ if it has one, and
$\Rd_{df}\A$ denotes the reduct of $\A$ obtained by discarding all the operations except for cylindrifications.
If $\A$ is any of the above classes, it is always the case that $\Rd_{df}\A\in \sf Df_{\alpha}$. If $\A\in \CA_\alpha$, then $\Rd_{sc}\A\in \Sc_\alpha$, and if $\A\in \QEA_{\alpha}$ then $\Rd_{ca}\A\in \CA_{\alpha}$.
Roughly speaking for an ordinal $\alpha$, $\CA_{\alpha}$s are  not expansions of $\Sc_{\alpha}$s, but they are {\it definitionally equivalent} to expansions of $\Sc_{\alpha}$, 
because the $\s_i^j$s are term definable in $\CA_{\alpha}$s by 
$\s_i^j(x)=\c_i(x\cdot {-\sf d}_{ij})$ $(i,j<\alpha)$. This operation reflects algebraically the subsititution of the variable $v_j$ for $v_i$ in a formula such that the substitution is free; this can be always done by reindexing bounded variables.
In such  situation, we say that $\Sc$s are {\it generalized reducts} of $\CA$s. However, $\CA_{\alpha}$s and $\sf \QA_{\alpha}$ are (real )reducts of $\QEA$s, (in the universal algebraic sense) simply obtained by discarding 
the operations in their signature not in the signature of their common expansion 
$\QEA_{\alpha}$.
We give a finite  approximate equational axiomatization of the concrete algebras defined above, which are the prime source of inspiration for these axiomatizations introduced to capture representability.
However, like for $\CA$s, this works only for certain special cases like the locally finite algebras, but does not generalize much further, cf Proposition \ref{j}.

\begin{definition} Let $\alpha$ be an ordinal. We say that a variety $\sf V$ is a variety between $\Df_\alpha$ and $\QEA_\alpha$  
if the signature of $\sf V$ 
expands that of $\Df_\alpha$ and is contained in the signature of $\QEA_\alpha$. Furthermore, 
any  equation formulated in the signature of $\Df_\alpha$ that holds in $\sf V$ also holds in $\Sc_\alpha$ 
and all  equations that hold in $\sf V$ holds in $\sf QEA_\alpha$. 
\end{definition}
Proper examples include $\Sc$, $\CA_\alpha$ and ${\sf QA}_\alpha$ (meaning strictly between). 
Analogously we can define varieties between $\Sc_\alpha$ and $\CA_\alpha$ or $\QA_\alpha$ and $\QEA_\alpha$, and more generally between a class $\sf K$ of $\sf BAO$s and a generalized reduct of it.
Notions like neat reducts generalize verbatim to such 
algebras, namely, to $\Df$s and $\QEA$s, and in any variety in between. This stems from the observation that for any pair of ordinals $\alpha<\beta$, $\A\in \QEA_{\beta}$ and any non-Boolean exra operation in the signature of $\QEA_{\beta}$, $f$ say, 
if $x\in \A$ and $\Delta x\subseteq \alpha$, then $\Delta(f(x))\subseteq \alpha$. Here $\Delta x=\{i\in \beta: \c_ix\neq x\}$ is referred as {\it the dimension set} of $x$; it reflects algebraically the essentially free variables occuring in a formula $\phi$. A variable is essentially free in a formula $\Psi$ $\iff$ it is free in every formula equivalent to $\Psi$.\footnote{It can well happen that a variable is free in  formula that is equivalent to another formula in which this same variable is not free.}
Therefore given a variety $\V$ between $\Sc_{\beta}$ and $\QEA_{\beta}$, if $\B\in \V$ then the algebra $\Nrr_{\alpha}\B$ having universe $\{x\in \B: \Delta x\subseteq \alpha$\} is closed under all operations in the signature
of $\V$.   

\begin{definition} Let $2<n<\omega$. For a variety $\sf V$ between $\Df_n$ and $\QEA_n$, a {\it $\sf V$ set algebra} is a subalgebra of an algebra, having the same signature as $\V$, of the form $\langle \B({}^nU), f_i^U)$, say, 
where $f_i^U$ is identical to the interpretation 
of $f_i$ in the class of quasipolyadic equality set algebras.  Let $\A$ be an algebra having the same signature of $\V$; then $\A$ is {\it a representable $\sf V$ algebra}, or simply {\it representable} 
$\iff$ $\A$ is isomorphic to a subdirect product of $\sf V$ set algebras. We write $\sf RV$ for the class of representable $\V$ algebras
\end{definition}
It can be proved that the class $\sf RV$,  as defined above, is also closed under $\bold H$, so that it is a variety. This can be proved using the same argument to show that $\RCA_n$ is a variety, cf. Corollary \cite[3.1.77]{HMT2}. 
Take $\A\in \sf RV$, an ideal $J$ of $\A$, 
then show that $\A/J$ is in $\sf RV$. Ideals in $\sf BAO$s are defined as follows. We consider only $\sf BAO$s with extra unary non-Boolean operators to simplify notation. If $\A$ is a $\sf BAO$, 
then $J\subseteq \A$ is an ideal in $J$ if is a Boolean ideal and for any extra 
non-Boolean operator $f$, say, in the signature of $\sf BAO$, and $x\in \A$, $f(x)\in \A$; the quotient algebra $\A/J$ is defined the usual way since ideals defined in this way correspond to congruence relations defined on $\A$. 
\begin{proposition} \label{j} Let $2<n<\omega$. Let $\V$ be a variety between $\Df_n$ and $\QEA_n$. Then $\sf RV$ is not a finitely axiomatizable variety.
\end{proposition}
\begin{proof} In \cite{j} a sequence $\langle \A_i: i\in \omega\rangle$  of algebras is constructed such that $\A_i\in  \QEA_n$ and $\Rd_{df}\A_n\notin {\sf RDf}_n$, but $\Pi_{i\in \omega}\A_i/F\in {\sf RQEA}_n$ for any non principal ultrafilter on $\omega$.
An appilcation of Los' Theorem, taking the ultraproduct of  $\V$ reduct of the $\A_i$s,  finishes the proof. In more detail, let $\Rd_{\V}$ denote restricting the signature to that of $\V$. 
Then  $\Rd_{\V}\A_i\notin \sf RV$ and $\Rd_{\V}\Pi_{i\in I}(\A_i/F)\in \sf RV$.
\end{proof} 
The last result generalizes to infinite dimensions replacing finite axiomatization by axiomatized by a finite schema \cite{HMT2, t}. 
We consider relation algebras as algebras of the form ${\cal R}=\langle R, +, \cdot, -, 1', \smile, ; , \rangle$, where $\langle R, +, \cdot , -\rangle$ is a Boolean algebra $1'\in R$, $\smile$ is a unary operation and $;$ is a binary operation. 
A relation algebra is {\it representable}$\iff$ it is isomorphic to a subalgebra of the form $\langle \wp(X), \cup, \cap, \sim, \smile, \circ, Id\rangle$, where $X$ is an equivalence relation, $1'$ is interpreted as the identity relation, $\smile$ is the operation of forming converses, 
and$;$  is interpreted as composition of relations. 
Following standard notation, $(\sf R)RA$ denotes the class of (representable) relation algebras. The class $\sf RA$ is  a discriminator variety that is finitely axiomatizable, cf. \cite[Definition 3.8, Theorems 3.19]{HHbook}.  
We let $\sf CRRA$ and $\sf LRRA$, 
denote the classes of completely representable $\RA$s, and its elementary closure, namely, the class of $\RA$s 
satisfying the Lyndon conditions  as defined in \cite[\S 11.3.2]{HHbook}, respectively. Complete representability of $\sf RA$s is defined like the $\CA$ case.
All of the above classes of algebras are instances of $\sf BAO$s. 
The action of the non--Boolean operators in a completely additive (where operators distribute over arbitrary joins componentwise) 
atomic $\sf BAO$, is determined by their behavior over the atoms, and
this in turn is encoded by the {\it atom structure} of the algebra.

\begin{definition}\label{definition}(\textbf{Atom Structure})
Let $\A=\langle A, +, -, 0, 1, \Omega_{i}:i\in I\rangle$ be
an atomic $\sf BAO$ with non--Boolean operators $\Omega_{i}:i\in I$. Let
the rank of $\Omega_{i}$ be $\rho_{i}$. The \textit{atom structure}
$\At\A$ of $\A$ is a relational structure
$$\langle At\A, R_{\Omega_{i}}:i\in I\rangle$$
where $At\A$ is the set of atoms of $\A$ 
and $R_{\Omega_{i}}$ is a $(\rho(i)+1)$-ary relation over
$At\A$ defined by
$$R_{\Omega_{i}}(a_{0},
\cdots, a_{\rho(i)})\Longleftrightarrow\Omega_{i}(a_{1}, \cdots,
a_{\rho(i)})\geq a_{0}.$$
\end{definition}
\begin{definition}(\textbf{Complex algebra})\label{definition2}
Conversely, if we are given an arbitrary first order structure
$\mathcal{S}=\langle S, r_{i}:i\in I\rangle$ where $r_{i}$ is a
$(\rho(i)+1)$-ary relation over $S$, called an {\it atom structure}, we can define its
\textit{complex
algebra}
$$\Cm(\mathcal{S})=\langle \wp(S),
\cup, \setminus, \phi, S, \Omega_{i}\rangle_{i\in
I},$$
where $\wp(S)$ is the power set of $S$, and
$\Omega_{i}$ is the $\rho(i)$-ary operator defined
by$$\Omega_{i}(X_{1}, \cdots, X_{\rho(i)})=\{s\in
S:\exists s_{1}\in X_{1}\cdots\exists s_{\rho(i)}\in X_{\rho(i)},
r_{i}(s, s_{1}, \cdots, s_{\rho(i)})\},$$ for each
$X_{1}, \cdots, X_{\rho(i)}\in\wp(S)$.
\end{definition}
It is easy to check that, up to isomorphism,
$\At(\Cm(\mathcal{S}))\cong\mathcal{S}$ alway. If $\A$ is
finite then of course
$\A\cong\Cm(\At\A)$. 
An atom structure will be denoted by $\bf At$.  An atom structure $\bf At$ has the signature of a class $\sf K$ of $\sf BAO$s,  
if  $\Cm\bf At\in \sf K$.

We define the notion of {\it clique guarded semantics}.
 \begin{definition} Let $2<n\leq m<\omega$. Let $\M$ be the base of a relativized representation of $\A\in \CA_n$ witnessed by an injective
homomorphism $f:\A\to \wp(V)$, where $V\subseteq {}^n\M$ and $\bigcup_{s\in V} \rng(s)=\M$. 
We write $\M\models a(s)$ for $s\in f(a)$. Let  $\L(\A)^m$ be the first order signature using $m$ variables
and one $n$--ary relation symbol for each element in $A$. 
Let $\L(\A)^m_{\infty, \omega}$ be the infinitary extension of $\L(\A)^m$ allowing infinite conjunctions. 
Then {\it an $n$--clique} is a set $C\subseteq \M$ such
$(a_1,\ldots, a_{n})\in V=1^{\M}$
for distinct $a_1, \ldots, a_{n}\in C.$

Let
${\sf C}^m(\M)=\{s\in {}^m\M :\rng(s) \text { is an $n$--clique}\}.$
${\sf C}^m(\M)$ is called the {\it $n$--Gaifman hypergraph of $\M$}, with the $n$--hyperedge relation $1^{\M}$.\\
The {\it clique guarded semantics $\models_c$} are defined inductively. For atomic formulas and Boolean connectives they are defined
like the classical case and for existential quantifiers
(cylindrifiers) they are defined as follows:
for $\bar{s}\in {}^m\M$, $i<m$, $\M, \bar{s}\models_c \exists x_i\phi$ $\iff$ there is a $\bar{t}\in {\sf C}^m(\M)$, $\bar{t}\equiv_i \bar{s}$ such that
$\M, \bar{t}\models \phi$.

(1) We say that $\M$ is  {\it an $m$--square representation} of $\A$,
if  for all $\bar{s}\in {\sf C}^m(\M), a\in \A$, $i<n$,
and   injective map $l:n\to m$, whenever $\M\models {\sf c}_ia(s_{l(0)},\ldots, s_{l(n-1)})$, then there is a $\bar{t}\in {\sf C}^m(\M)$ with $\bar{t}\equiv _i \bar{s}$,
and $\M\models a(t_{l(0)}, \ldots, t_{l(n-1)})$.
$\M$ is {\it a complete $m$--square representation of $\A$ via $f$}, or simply a complete representation of $\A$ if 
$f(\sum X)=\bigcup_{x\in X} f(x)$, for all 
$X\subseteq \A$ for which $\sum X$ exists. (Like in the classical case this is equivalent to that $\A$ is atomic and that $f$ is atomic in the sense that $\bigcup_{x\in \At\A} f(x)=1^{\M}$).

(2) We say that $\M$ is an {\it (infinitary) $m$--flat representation} of $\A$ if  it is $m$--square and
for all $\phi\in (\L(\A)_{\infty, \omega}^m) \L(\A)^m$, 
for all $\bar{s}\in {\sf C}^m(\M)$, for all distinct $i,j<m$,
$\M\models_c [\exists x_i\exists x_j\phi\longleftrightarrow \exists x_j\exists x_i\phi](\bar{s})$.
Complete representability is defined like for squareness.
\end{definition}

For sequences $f, g$ having the same domain an ordinal $\alpha$ say, and $i\in \dom f$, we write $f\equiv_i g$$\iff$ {\it $f$ and $g$ agree off of $i$}, that is to say $f(x)=g(x)$ for all 
$x\in \dom(f)\sim \{i\}$.

\begin{definition}\label{network}  
An {\it $n$--dimensional atomic network} on an atomic algebra $\A\in \QEA_n$  is a map $N: {}^n\Delta\to  \At\A$, where
$\Delta$ is a non--empty finite set of {\it nodes}, denoted by $\nodes(N)$, satisfying the following consistency conditions for all $i<j<n$: 
\begin{enumroman}
\item If $\bar{x}\in {}^n\nodes(N)$  then $N(\bar{x})\leq {\sf d}_{ij}\iff\bar{x}_i=\bar{x}_j$,
\item If $\bar{x}, \bar{y}\in {}^n\nodes(N)$, $i<n$ and $\bar{x}\equiv_i \bar{y}$, then  $N(\bar{x})\leq {\sf c}_iN(\bar{y})$,
\item (Symmetry): if $\bar{x}\in {}^n\nodes(N)$, then  $\s_{[i, j]}N(\bar{x})=N(\bar{x}\circ [i, j]).$
\end{enumroman}
If $\A\in \CA_n$, then an $\A$ network is a map defined like above satisfying only (i) and (ii). If $\A\in \QA_n$, then an $\A$ network satisfies  (ii) and (iii) together with 
the condition that if $\bar{x}\in {}^n\nodes(N)$, then  $\s_{[i|j]}N(\bar{x})=N(\bar{x}\circ [i|j])$ (instead of (i)). Finally, if $\A\in \Sc_n$ than an $\A$ network satisfies the last condition 
together with (ii).
\end{definition}
\begin{definition} 
\begin{enumerate}
\item   Assume that $m, k\leq \omega$. 
The {\it atomic game $G^m_k(\At\A)$, or simply $G^m_k$}, is the game played on atomic networks
of $\A$ using $m$ nodes, each node only once, so that any node being used  is not alllowed to be reused;  and having $k$ rounds \cite[Definition 3.3.2]{HHbook2}, where
\pa\ is offered only one move, namely, {\it a cylindrifier move}: 
Suppose that we are at round $t>0$. Then \pa\ picks a previously played network $N_t$ $(\nodes(N_t)\subseteq m$), 
$i<n,$ $a\in \At\A$, $x\in {}^n\nodes(N_t)$, such that $N_t(\bar{x})\leq {\sf c}_ia$. For her response, \pe\ has to deliver a network $M$
such that $\nodes(M)\subseteq m$,  $M\equiv _i N$, and there is $\bar{y}\in {}^n\nodes(M)$
that satisfies $\bar{y}\equiv _i \bar{x}$ and $M(\bar{y})=a$, cf. \cite[Definition 12.5(2)]{HHbook} for the notation $M\equiv_iN$.   

\item  We write $G_k(\At\A)$, or simply $G_k$, for $G_k^m(\At\A)$ if $m\geq \omega$.

\item  The $\omega$--rounded game $\bold G^m(\At\A)$ or simply $\bold G^m$ is like the game $G_{\omega}^m(\At\A)$ 
except that \pa\ has the option 
to reuse the $m$ nodes in play.
\end{enumerate}
\end{definition}

For $\sf BAO$s,  $\A$ and $\B$ say, having the same signature,  
we say that $\A$ is {\it dense} in $\B$ if $\A\subseteq \B$ and for all non--zero $b\in \B$, there is a non--zero 
$a\in A$ such that $a\leq b$.
An atom structure will be denoted by $\bf At$.  An atom structure $\bf At$ has the signature of $\CA_\alpha$, $\alpha$ an ordinal, 
if  $\Cm\bf At$ has the signature of $\CA_\alpha$. 

\begin{definition}\label{canonical} 
Let $\V$ be a completely additive variety of $\sf BAO$s. Then $\V$ is {\it atom--canonical} if whenever $\A\in \V$ and $\A$ is atomic, then $\mathfrak{Cm}\At\A\in \V$.
The {\it  \de\ completion} of  $\A\in \V$, is the unique (up to isomorphisms that fix $\A$ pointwise)  complete  
$\B$ such that $\A\subseteq \B$ and $\A$ is {\it dense} in $\B$. 
\end{definition}

We also need the notion of $m$--dimensional hyperbasis for cylindric algebras. Hypernetyworks and hyperbasis are defined for relation algebras by Hirsch and Hodkinson \cite[Definitions, 12.1, 12.11]{HHbook}. 
This hyperbasis is made up of $m$--dimensional hypernetworks.
An $m$--dimensional hypernetwork on the atomic algebra $\A$ is an $n$--dimensional  network $N$ (which is a basic matrix), with $\nodes(N)\subseteq m$, endowed with a set of labels $\Lambda$ for
hyperedges of length $\leq m$,
not equal to $n$ (the dimension), such that $\Lambda\cap \At\A=\emptyset$. We call a label in $\Lambda$ a non-atomic label.
Like in networks, $n$--hyperedges are labelled by atoms. In addition to the consistency properties for networks,
an $m$--dimensional hypernetwork should satisfy the following additional consistency rule involving non--atomic labels:
If $\bar{x}, \bar{y}\in {}^{\leq m}m$, $|\bar{x}|=|\bar{y}|\neq n$ and $\exists \bar{z}$, such that $\forall i<|\bar{x}|$,
$N(x_i,y_i,\bar{z})\leq {\sf d}_{01}$,
then $N(\bar{x})=N(\bar{y})\in \Lambda$. 

\begin{definition} Let $2<n<m<\omega$ and $\A\in \CA_n$ be atomic.

(1) An $m$--dimensional basis $B$ for $\A$ consists of a set of $n$--dimensional networks (basic matrices) whose nodes $\subseteq m$, satisfying 
the following properties: 

\begin{itemize}

\item For all  $a\in \At\A$, there is an $N\in B$ such that $N(0,1,\ldots, n-1)=a,$

\item The {\it cylindrifier property}: For all $N\in B$, all $i<n$,  all $\bar{x}\in {}^n\nodes(N)(\subseteq {}^nm)$, all $a\in\At\A$, such that
$N(\bar{x})\leq {\sf c}_ia$,  there exists $M\in B$, $M\equiv_i N$, $\bar{y}\in {}^n\nodes(M)$ such 
that $\bar{y}\equiv_i\bar{x}$ and $M(\bar{y})=a.$   

\end{itemize}

(2)  An $m$--dimensional  hyperbasis $H$ consists of $m$--dimensional hypernetworks, satisfying the above two conditions reformulated 
the obvious way for hypernetworks, in addition,  $H$ has an amalgamation property for overlapping hypernertworks; 
this property corresponds to commutativity of cylindrifiers: 

For all $M,N\in H$ and $x,y<m$, with $M\equiv_{xy}N$, there is $L\in H$ such that
$M\equiv_xL\equiv_yN$. Here $M\equiv_SN$, means 
that $M$ and $N$ agree off of $S$ \cite[Definition 12.11]{HHbook}.
\end{definition}

\begin{theorem} Let $2<n<m<\omega$. Assume that $A\in \CA_n$. Then $\A$ has a complete  
$m$-square representation $\iff$ \pe\ has a \ws\ in the $\omega$ rounde atomic game $G_{\omega}^m(\At\A)$ using $m$ nodes. 
\end{theorem}
\begin{proof} Assume that $\A$ is an atomic $\CA_n$ having a complete $m$--square representation.
We will show that \pe\ has a \ws\  in $G_{\omega}^m$.
As before, let $\M$ be a complete $m$--square representation of $\A$.
One constructs the $m$--dimensional dilation $\D$ using $L_{\infty, \omega}^n$ formulas 
from a complete $m$--square representation exactly like in the proof of lemma \ref{flat}.
The neat embedding map  $\theta:\A\to \D$ is the same, defined  via $r\mapsto r(\bar{x})^{\M}$.
Here the $m$--neat reduct of $\D$ is defined like the $\CA$ case,  even though the dilation $\D$ {\it may not be a $\CA_m$} for we
do not necessarily have commutativity of cylindrifiers, because there is no guarantee that $\M$ is
$m$--flat. As before $\theta$ is an injective homomorphism into $\Nr_n\D$, and $\D$ is atomic.
For each $\bar{a}\in 1^{\D},$ define \cite[Definition 13.22] {HHbook} a labelled
hypergraph $N_{\bar{a}}$ with nodes $m$, and
$N_{\bar{a}}(\bar{x})$ when $|\bar{x}|=n$, is the unique atom of $\A$
containing the tuple of length $m>n$,\\
$(a_{x_0},\ldots, a_{x_{1}},\ldots, a_{x_{n-1}}, a_{x_0}\ldots,\ldots a_{x_0}).$
It is clear that if $s\in 1^{\D}$ and $i, j<m$,
then $s\circ [i|j]\in 1^{\D}$.

Hence the above definition is sound. Indeed, if $\Psi: m\to m$ is defined by 
$\Psi(1)=x_1,\dots$,  $\Psi(n-2)=x_{n-2},$
and $\Psi(i)=x_0$ for $i\in m\sim \{1,\ldots, n-2\}$, then $\Psi$ is not injective, hence it is a composition of
replacements, so $\bar{a}\circ \Psi=(a_{x_0},\ldots, a_{x_{1}},\ldots, a_{x_{n-1}}, a_{x_0}\ldots\ldots a_{x_0})\in 1^{\D}.$
It is also easy to see, since $\A\subseteq \Nr_n\D$, that if $\bar{a}=(a_0, \ldots, a_{m-1})\in 1^{\D}$,
$i_0, \ldots, i_{n-1}<m$
and  $\bar{b}\in 1^{\D}$ is such that $\bar{b}\upharpoonright n\subseteq \bar{a}$, then for
all atoms $r\in \A$,  $\bar{b}\in r\iff  r=N_{\bar{a}}(i_0,\ldots, i_{n-1}).$
Furthermore, \cite[Lemma 13.24]{HHbook}  $N_{\bar{a}}$ is a network.
Let $H$ be the symmetric closure
of $\{N_a: \bar{a}\in 1^M\}$, that is $\{N\theta: \theta:m\to m, N\in H\}$.
Then $H$ is an $m$--dimensional basis \cite[Lemma 13.26]{HHbook} as defined in the proof of lemma \ref{Thm:n}.
Recall that $H$ `eliminates cylindrifiers' in the following sense: For all $a\in \At\A$, $i<n$ and $N\in H$, for all $\bar{x}\in {}^n\nodes(N)$, whenever 
$N(\bar{x})\leq {\sf c}_ia$, then there is an $M\in H$,  with $M\equiv _i N$, and 
$\bar{y}\in {}^n\nodes(M)$, such that  $\bar{y}\equiv_i \bar{x}$ 
and $M(\bar{y})=a.$
Now \pe\ can win $G_{\omega}^m$ by always
playing a subnetwork of a network in the constructed $H$.
In round $0$, when \pa\ plays
the atom $a\in \A$, \pe\ chooses $N\in H$ with $N(0,1,\ldots, n-1)=a$ and plays $N\upharpoonright n$.
In round $t>0$, inductively if the current network is $N_{t-1}\subseteq M\in H$, then no matter how \pa\ defines $N$, we have
$N\subseteq M$ and $|N|<m$, so there is $z<m$, with $z\notin \nodes(N)$.
Assume that  \pa\ picks $x_0,\ldots, x_{n-1}\in \nodes(N)$, $a\in \At\A$ and $i<n$ such that
$N(x_0,\ldots, x_{n-1})\leq {\sf c}_ia$, so $M(x_0, \ldots  x_{n-1})\leq {\sf c}_ia$,
and hence (by the properties of $H$), there is $M'\in H$ with
$M'\equiv _i M$ and $M'(x_0, \ldots, z, \ldots,  x_{n-1})=a$, with $z$ in the $i$th place.
Now \pe\ responds with the restriction of $M'$
to $\nodes(N)\cup \{z\}$.
\end{proof}

\begin{definition}\label{sub} Let $m$ be a finite ordinal $>0$. An $\sf s$ word is a finite string of substitutions $({\sf s}_i^j)$ $(i, j<m)$,
a $\sf c$ word is a finite string of cylindrifications $({\sf c}_i), i<m$;
an $\sf sc$ word $w$, is a finite string of both, namely, of substitutions and cylindrifications.
An $\sf sc$ word
induces a partial map $\hat{w}:m\to m$:
\begin{itemize}

\item $\hat{\epsilon}=Id,$

\item $\widehat{w_j^i}=\hat{w}\circ [i|j],$

\item $\widehat{w{\sf c}_i}= \hat{w}\upharpoonright(m\smallsetminus \{i\}).$

\end{itemize}
If $\bar a\in {}^{<m-1}m$, we write ${\sf s}_{\bar a}$, or
${\sf s}_{a_0\ldots a_{k-1}}$, where $k=|\bar a|$,
for an  arbitrary chosen $\sf sc$ word $w$
such that $\hat{w}=\bar a.$
Such a $w$  exists by \cite[Definition~5.23 ~Lemma 13.29]{HHbook}.
\end{definition}
The proof of the  following lemma can be distilled
from its $\sf RA$ analogue \cite[Theorem 13.20]{HHbook},  by reformulating deep concepts
originally introduced by Hirsch and Hodkinson for $\sf RA$s in the $\CA$ context, involving the notions of 
hypernetworks and hyperbasis. This can (and will) be done.
In the coming proof, we highlight
the main ideas needed to perform such a transfer from $\sf RA$s to $\CA$s
\cite[Definitions 12.1, 12.9, 12.10, 12.25, Propositions 12.25, 12.27]{HHbook}. 
In all cases, the $m$--dimensional dilation stipulated in the statement of the theorem, will have
top element ${\sf C}^m(\Mo)$, where $\Mo$ is the $m$--relativized representation of the given algebra, and the operations of the dilation
are induced by the $n$-clique--guarded semantics. For a class $\sf K$ of $\sf BAO$s, $\sf K\cap \bf At$ denotes the class of atomic algebras in $\sf K$.

A set $V$ ($\subseteq {}^nU$)  is {\it diagonizable} if  
$s\in V\implies s\circ [i|j]\in V$. We write $\bold S_c$ for the operation of forming complete subalgebras. 
That is to say, for a class $\sf K$ of $\sf BAO$s, $\B\in \bold S_c\K\iff$ there is an  $\A\in {\sf K}$ such that $\A\subseteq \B$ and for all $X\subseteq \A$, if $\sum^{\A}X=1$, then $\sum^{\B}X=1$.
For two $\sf BAO$s,  $\A$ and $\B$ having the same signature,  we write $\A\subseteq_c\B$, if $\A$ is a complete subalgebra of $\B$.
\begin{lemma}\label{flat}\cite[Theorems 13.45, 13.36]{HHbook}.
Assume that $2<n<m<\omega$ and let $\A$ be a $\sf BAO$ having the same signature as $\CA_n$ and satisfying all the $\CA_n$ axioms except possibly for comutativity of cylindrifications. 

1. Then $\A\in \bold S\Nr_n\CA_m\iff \A$ has an  infinitary $m$--flat representation
$\iff \A$ has an $m$--flat representation. Furthermore, 
if $\A$ is atomic, then $\A$ has a complete infinitary $m$--flat representation $\iff$ $\A\in \bold S_c\Nr_n(\CA_m\cap \bf At)$.

2. We  can replace infinitary $m$-flat and $\CA_m$ by $m$-square and ${\sf D}_m$, respectively, where $\sf D_m$ are set algebras having a diagonizable top element $V$ with operations defined like ${\sf Cs}_m$ restricted to $V$.
\end{lemma}

\begin{proof} We give a sketchy outline.
We start from representations to  dilations.
Let $\Mo$ be an $m$--flat representation of $\A$. 
For $\phi\in \L(\A)^m$, 
let $\phi^{\Mo}=\{\bar{a}\in {\sf C}^m(\Mo):\Mo\models_c \phi(\bar{a})\}$, where ${\sf C}^m(\Mo)$ is the $n$--Gaifman hypergraph.
Let $\D$ be the algebra with universe $\{\phi^{M}: \phi\in \L(\A)^m\}$ and with  cylindric
operations induced by the $n$-clique--guarded (flat) semantics. 
For $r\in \A$, and $\bar{x}\in {\sf C}^m(\Mo)$, we identify $r$ with the formula it defines in $\L(\A)^m$, and 
we write $r(\bar{x})^{\Mo}\iff \Mo, \bar{x}\models_c r$.
Then $\D$ is a set algebra 
with domain $\wp({\sf C}^m(\Mo))$ and with unit $1^{\D}={\sf C}^m(\Mo)$.
Since $\Mo$ is $m$--flat, then cylindrifiers in $\D$ commute, and so $\D\in \CA_m$.
Now define $\theta:\A\to \D$, via $r\mapsto r(\bar{x})^{\Mo}$. Then exactly like in the proof of \cite[Theorem 13.20]{HHbook},
$\theta$ is an injective neat embedding, that is to say, $\theta(\A)\subseteq \mathfrak{Nr}_n\D$.
 The relativized model $\Mo$ itself might not be  infinitary $m$--flat, but one can build an infinitary $m$--flat representation of $\A$, whose base $\Mo$ is an $\omega$--saturated model
of the consistent first order theory, stipulating the existence of an $m$--flat representation, cf. \cite[Proposition 13.17, Theorem 13.46 items (6) and (7)]{HHbook}.

The inverse implication from dilations to representations harder. One constructs from the given 
$m$--dilation, an $m$--dimensional 
hyperbasis (that can be defined similarly to the $\RA$ case, cf. \cite[Definition 12.11]{HHbook}) 
from which
the required $m$-relativized representation is built.  
This can be done in a step--by step manner treating the hyperbasis 
as a `saturated set of mosaics', cf. \cite[Proposition 13.37]{HHbook}..
We show how an $m$--dimensional hyperbasis for the  canonical extension of $\A\in \CA_n$
is obtained from an $m$--dilation of $\A$ \cite[Definition 13.22, lemmata 13.33-34-35, Proposition 36]{HHbook}.
Suppose that $\A\subseteq \Nrr_n\D$ for some $\D\in \CA_m$.
Then $\A^+\subseteq_c \Nrr_m\D^+$ and $\D^+$ is atomic. We show that $\D^+$ has an $m$--dimensional hyperbasis.
First, it is not hard to see that for every $n\leq l\leq m$, $\Nrr_l\D^+$ is atomic.
The set of non--atomic labels $\Lambda$ is the set $\bigcup_{k<m-1}\At\Nrr_k\D^+$.
For each atom $a$ of $\D^+$, define a labelled  hypergraph $N_a$ as follows.
Let $\bar{b}\in {}^{\leq m}m$. Then if $|\bar{b}|=n$,  so that $\bar{b}$  has to get a label that is an atom of $\D^+$, one sets  $N_a(\bar{b})$ to be 
the unique $r\in \At\D^+$ such that $a\leq {\sf s}_{\bar{b}}r$; notation here
is given in definition \ref{sub}.
If $n\neq |\bar{b}| <m-1$, $N_a(\bar{b})$ is the unique atom $r\in \Nrr_{|b|}\D^+$ such that $a\leq {\sf s}_{\bar{b}}r.$ Since
$\Nrr_{|b|}\D^+$ is atomic, this is well defined. Note that this label may be a non--atomic one; it 
might not be an atom of $\D^+$. But by definition it is a permitted label.
Now fix $\lambda\in \Lambda$. The rest of the labelling is defined by $N_a(\bar{b})=\lambda$.
Then $N_a$ as an $m$--dimensional
hypernetwork, for each 
such chosen $a$,  and $\{N_a: a\in \At\D^+\}$ is the required $m$--dimensional hyperbasis.
The rest of the proof consists of a fairly straightforward adaptation of the proof \cite[Proposition 13.37]{HHbook},
replacing edges by $n$--hyperedges.

For results on {\it complete} $m$--flat representations, one works in $L_{\infty, \omega}^m$ instead of first order logic. 
With $\D$ formed like above from (the complete $m$--flat representation) $\Mo$, using $\L(\A)_{\infty,\omega}^m$ instead of $L_n$, 
let $\phi^{\Mo}$ be a non--zero element in $\D$.
Choose $\bar{a}\in \phi^{\Mo}$, and let $\tau=\bigwedge \{\psi\in \L(\A)_{\infty,\omega}^m: \Mo\models_c \psi(\bar{a})\}.$
Then $\tau\in \L(\A)_{\infty,\omega}^m$, and $\tau^{\Mo}$ is an atom below $\phi^{\Mo}$. 
The rest is entirely analogous, cf. \cite[p.411]{HHbook}. 

\end{proof}


\section{Atom-canonicity}

Here we review and elaborate on the construction in \cite{ANT} as an instance of  
first instance of a blow up and blur construction. 
We will construct cylindric agebras from atomic relation algebras that posses  {\it cylindric basis}. A cylindric basis is a `saturated' set of matrices. 

\begin{definition} Let $\R$ be an atomic  relation algebra.  An {$n$--dimensional basic matrix}, or simply a matrix  
on $\R$, is a map $f: {}^2n\to \At\R$ satsfying the 
following two consistency 
conditions $f(x, x)\leq \Id$ and $f(x, y)\leq f(x, z); f(z, y)$ for all $x, y, z<n$. For any $f, g$ basic matrices
and $x, y<m$ we write $f\equiv_{xy}g$ if for all $w, z\in m\setminus\set {x, y}$ we have $f(w, z)=g(w, z)$.
We may write $f\equiv_x g$ instead of $f\equiv_{xx}g$.  
\end{definition}
\begin{definition}\label{b}
An {\it $n$--dimensional cylindric basis} for an atomic relation algebra 
$\R$ is a set $\cal M$ of $n$--dimensional matrices on $\R$ with the following properties:
\begin{itemize}
\item If $a, b, c\in \At\R$ and $a\leq b;c$, then there is an $f\in {\cal M}$ with $f(0, 1)=a, f(0, 2)=b$ and $f(2, 1)=c$
\item For all $f,g\in {\cal M}$ and $x,y<n$, with $f\equiv_{xy}g$, there is $h\in {\cal M}$ such that
$f\equiv_xh\equiv_yg$. 
\end{itemize}
\end{definition}
For the next lemma, we refer the reader to \cite[Definition 12.11]{HHbook} 
for the definition of of hyperbasis for relation algebras as well as to 
\cite[Chapter 13, Definitions 13.4, 13.6]{HHbook} for the notions 
of $n$--flat and $n$--square representations for  relation algebras ($n>2$) to be generalized below to cylindric algebras.
For a Boolean algebra with operators $\A$, say,   $\A^+$ denotes its canonical extension.

\begin{lemma}\label{i} Let $\R$ be  a relation algebra and $3<n<\omega$.  Then the following hold:
\begin{enumerate}
\item $\R^+$ has an $n$--dimensional infinite basis $\iff\ \R$ has an infinite $n$--square representation.

\item $\R^+$ has an $n$--dimensional infinite hyperbasis $\iff\ \R$ has an infinite $n$--flat representation.
\end{enumerate}
\end{lemma}
\begin{proof} \cite[Theorem 13.46, the equivalence $(1)\iff (5)$ for basis, and the equivalence $(7)\iff (11)$ for hyperbasis]{HHbook}.
\end{proof} 

One can construct a $\CA_n$ in a natural way from an atomic relation algebra possessing an $n$--dimensional cylindric basis which can be viewed as an atom structure of a $\CA_n$
(like in \cite[Definition 12.17]{HHbook} addressing hyperbasis).
For an atomic  relation algebra $\R$ and $l>3$, we denote by ${\sf Mat}_n(\At\R)$ the set of all $n$--dimensional basic matrices on $\R$.
${\sf Mat}_n(\At\R)$ is not always an $n$--dimensional cylindric basis, but sometimes it is,
as will be the case described next. 
The following definition to be used in the sequel is taken from \cite{ANT}:
\begin{definition}\label{strongblur}\cite[Definition 3.1]{ANT}
Let $\R$ be a relation algebra, with non--identity atoms $I$ and $2<n<\omega$. Assume that  
$J\subseteq \wp(I)$ and $E\subseteq {}^3\omega$.
\begin{enumerate}
\item We say that $(J, E)$  is an {\it $n$--blur} for $\R$, if $J$ is a {\it complex $n$--blur} defined as follows:   
\begin{enumarab}
\item Each element of $J$ is non--empty,
\item $\bigcup J=I,$
\item $(\forall P\in I)(\forall W\in J)(I\subseteq P;W),$
\item $(\forall V_1,\ldots V_n, W_2,\ldots W_n\in J)(\exists T\in J)(\forall 2\leq i\leq n)
{\sf safe}(V_i,W_i,T)$, that is there is for $v\in V_i$, $w\in W_i$ and $t\in T$,
we have
$v;w\leq t,$ 
\item $(\forall P_2,\ldots P_n, Q_2,\ldots Q_n\in I)(\forall W\in J)W\cap P_2;Q_n\cap \ldots P_n;Q_n\neq \emptyset$.
\end{enumarab}
and the tenary relation $E$ is an {\it index blur} defined  as 
in item (ii) of \cite[Definition 3.1]{ANT}.

\item We say that $(J, E)$ is a {\it strong $n$--blur}, if it $(J, E)$ is an $n$--blur,  such that the complex 
$n$--blur  satisfies:
$$(\forall V_1,\ldots V_n, W_2,\ldots W_n\in J)(\forall T\in J)(\forall 2\leq i\leq n)
{\sf safe}(V_i,W_i,T).$$ 
\end{enumerate}
\end{definition}

\begin{definition} A $\CA_n$ atom structure $\bf At$  is {\it weakly representable} if there is an atomic $\A\in {\sf RCA}_n$ such that $\bf At=\At\A$; it is  {\it strongly representable} if $\Cm {\bf At}\in {\sf RCA}_n$.
The same notions apply to $\sf RA$s and any class between ${\sf Df}_n$ and $\QEA_n$.
\end{definition}
These two notions are distinct, cf. \cite{Hodkinson}, \cite{ANT}, the following Example \ref{exa} and Theorem \ref{can}.

\begin{example}\label{exa}

We give an example taken from \cite{ANT} of a blowing up and blurring a finite relation algebra $\R$ such getting $\cal R$ such that that $\At\cal R$ is  weakly
but not strongly representable.
Furthermore $\cal R$ has an $n$ dimensional cylindric basis, and ${\sf Mat}_n(\At\cal \R)$ is a weakly but 
not strongly representable $\CA_n$ atom structure. This example is taken from \cite{ANT}. 
Our exposition of the construction in \cite{ANT} will be addressing an (abstract) finite relation algebra $\R$ having an $l$--blur in the sense of definition \cite[Definition 3.1]{ANT}, 
with $3\leq l\leq k<\omega$ and $k$ depending on $l$.  Occasionally we use the concrete 
Maddux algebra $\mathfrak{E}_k(2, 3)$ to make certain concepts more tangible.
Here $k$ is the number of non-identity atoms is concrete example of $\R$. 
In this algebra a triple $(a, b, c)$ of non--identity atoms is consistent $\iff$ $|\{a, b, c\}|\neq 1$, i.e 
only monochromatic triangles are forbidden. 

We use the notation in \cite{ANT}. Let $2<n\leq l<\omega$. One starts with a finite relation algebra $\R$ that has only representations, if any, on finite sets (bases), 
having an $l$--blur $(J, E)$ as in \cite[Definition 3.1]{ANT} recalled in definition \ref{strongblur}. 
After {\it blowing up and bluring $\R$}, by splitting each of its atoms into infinitely many, one gets 
an infinite atomic representable relation algebra ${\mathfrak Bb}(\R, J, E)$ \cite[p.73]{ANT}, whose atom structure $\bf At$ is weakly but not
strongly representable. The atom structure $\bf  At$ is not strongly representable, because $\R$ is {\it not blurred} in ${\sf Cm}\bf At$. 
The finite relation algebra $\R$ embeds into $\Cm\bf At$, so that a representation 
of $\Cm\bf At$, necessarily on
an infinite base, induces one of $\R$ on the same base, which is impossible.
The representability of ${\mathfrak Bb}(\R, J, E)$ depend on the properties of the $l$--blur,  which {\it blurs $\R$ in ${\mathfrak Bb}(\R, J, E)$}.
The set of blurs here, namely, $J$ is finite. In the case of $\mathfrak{E}_k(2, 3)$ used in \cite{ANT},  the set of blurs 
is the set of all subsets of non--identity atoms having the same size $l<\omega$, where $k=f(l)\geq l$ 
for some recursive function $f$ from $\omega\to \omega$, so that $k$ depends recursively on $l$. 

One (but not the only) way to define the {\it index blur} $E\subseteq {}^3\omega$ is as follows \cite[Theorem 3.1.1]{Sayed}:
$E(i,j,k)\iff (\exists p,q,r)(\{p,q,r\}=\{i,j,k\} \text { and } r-q=q-p.$
This is a concrete instance of an index blur as defined in \cite[Definition 3.1(iii)]{ANT} (recalled in definition \ref{strongblur} above), 
but defined uniformly, it does not depends on the blurs.
The underlying set of $\bf At$, the atom structure of ${\mathfrak Bb}(\R, J, E)$ is the following set consisting of triplets:
$At=\{(i, P, W): i\in \omega, P\in \At\R\sim \{\Id\}, W\in J\}\cup \{\Id\}$.
When $\R=\mathfrak{E}_k(2, 3)$ (some finite $k>0)$, composition  is defined by singling out the following (together with their Peircian transforms), 
as the consistent triples:
$(a, b, c)$ is consistent $\iff$ one of $a, b, c$ is $\sf Id$ and the other two are equal, or 
if $a=(i, P, S), b=(j, Q, Z), c= (k, R, W)$ 
$$S\cap Z\cap W\neq \emptyset \implies E(i,j,k)\&|\{P,Q,R\}|\neq 1.$$ 
(We are avoiding mononchromatic triangles).
That is if for $W\in J$,  $E^W=\{(i, P, W): i\in \omega, P\in W\},$
then $$(i, P, S); (j, Q, Z)=\bigcup\{E^W: S\cap Z\cap W=\emptyset\}$$
$$\bigcup \{(k, R, W): E(i,j,k), |\{P,Q,R\}|\neq 1\}.$$

More generally, for the $\R$ as postulated in the hypothesis, composition in $\bf At$ is defined as follow. First 
the index blur $E$ can be taken to be like above. 
Now 
the triple  $((i, P, S), (j, Q, Z), (k, R, W))$ in which no two entries are equal, is consistent
if either $S, Z, W$ are ${\it safe}$, briefly ${\sf safe}(S, Z, W)$, 
witness item  (4) in definition \ref{strongblur} (which vacuously hold if $S\cap Z\cap W=\emptyset$), 
or $E(i, j, k)$ and $P; Q\leq  R$ in $\R$. 
This generalizes the above definition of composition, 
because in $\mathfrak{E}_k(2, 3)$, the triple of non--identity atoms 
$(P, Q, R)$ is consistent $\iff$ they do not have the same colour $\iff$ $|\{P, Q, R\}|\neq 1.$
Having specified its atom structure,  its timely to 
specfiy the relation algebra ${\mathfrak Bb}(\R, J, E)\subseteq \Cm{\bf At}$.
The relation algebra ${\mathfrak Bb}(\R, J, E)$ is $\Tm\bf At$ (the term algebra).
Its  universe is the set $\{X\subseteq H\cup \{\Id\}: X\cap E^W\in {\sf Cof}(E^W), \text{ for all } W\in J\}$, where 
${\sf Cof}(E^{W})$ denotes the set of co--finite subsets of $E^{W}$, 
that is subsets of $E^W$ whose complement is infinite,  with $E^W$ as defined above. The relation algebra 
operations lifted from $\bf At$ the usual way.
The algebra  ${\mathfrak Bb}(\R, J, E)$ is proved to be representable \cite{ANT} as shown next..
For brevity, denote ${\mathfrak Bb}(\R, J, E)$ by $\cal R$, and its domain by $R$.
For $a\in \bf At$, and $W\in J,$  set
$U^a=\{X\in R: a\in X\}\text { and } U^{W}=\{X\in R: |X\cap E^W|\geq \omega\}.$
Then the principal ultrafilters of $\cal R$ are exactly $U^a$, $a\in H$ and $U^W$
are non-principal ultrafilters for $W\in J$ when $E^W$ is infinite.
Let  $J'=\{W\in J: |E^W|\geq \omega\},$
and let ${\sf Uf}=\{U^a: a\in F\}\cup \{U^W: W\in J'\}.$
${\sf Uf}$ is the set of ultrafilters of $\cal R$ which is used as colours
to represent $\cal R$, cf. \cite[pp. 75-77]{ANT}. The representation is 
built from coloured graphs whose edges are labelled 
by elements in ${\sf Uf}$   in a fairly standard step--by--step construction.

Now we show  why the \de\ completion $\Cm \bf At$ is {\it not} representable. 
For $P\in I$, let $H^P=\{(i, P,W): i\in \omega, W\in J, P\in W\}$.
Let  $P_1=\{H^P: P\in I\}$ and $P_2=\{E^W: W\in J\}$.   These are two partitions of $At$. 
The partition $P_2$  was used to {\it represent},
${\mathfrak Bb}(\R, J, E)$, in the sense that the tenary relation corresponding to composition 
was defined on $\bf At$, in a such a way so that the singletons generate the partition
$(E^W: W\in J)$ up to ``finite deviations." 
The  partition $P_1$ will now be used to show that $\Cm({\mathfrak Bb}(\R, J, E))=\Cm (\bf At)$ 
is {\it not }  representable.   This follows by observing that 
omposition restricted to $P_1$ satisfies: $ H^P;H^Q=\bigcup \{H^Z: Z;P\leq Q \text { in } \R\}$
which means that $\R$ embeds into the complex algebra 
$\Cm \bf At$ prohibiting its representability, 
because $\R$ allows only representations having 
a finite base.

So far we have been dealing with relation algebras. The constructions lifts to higher dimensions expressed in $\CA_n$s, $2<n<\omega$. 
Because $(J, E)$ is an $l$--blur, then by \cite[Theorem 3.2 9(iii)]{ANT}, ${\bf At}_{ca}={\sf Mat}_l(\At {\mathfrak Bb}(\R, J, E))$, the set of $l$ by $l$ basic matrices on $\bf At$ 
is an $l$--dimensional cylindric basis, giving  an algebra $\B_l={\mathfrak  Bb}_l(\R, J, E)\in \RCA_l$. Again 
${\bf At}_{ca}$ is not strongly representable, for had it been then a representation of $\Cm{\bf At}_{ca}$, induces a representation
of $\R$ on an infinite base, because $\Ra\Cm{\bf At}_{ca}\supseteq \Cm\bf At\supseteq \R$, 
and the representability  of $\Cm{\bf At}_{ca}$ induces one of $\Ra\Cm{\bf At}_{ca}$, necessarily having an infinite base. 
 For $2<n\leq l<\omega$, denote by $\C_l$ the non-representable 
\de\ completion of the algebra ${\mathfrak Bb }_l(\R, J, E)\in \RCA_l$, that is 
$\C_l=\Cm\At({\mathfrak Bb}_l(\R, J, E))=\Cm {\sf Mat}_l(\bf At)$.
If the $l$--blur happens to be {\it strong}, in the sense of definition \ref{strongblur} and $n\leq m\leq l,$
then we get by \cite[item (3) pp.80]{ANT},  that ${\mathfrak Bb}_m(\R, J, E)\cong \Nr_m{\mathfrak Bb}_l(\R, J, E)$.
This is proved by defining  an embedding $h:\Rd_m\C_l\to \C_m$ via 
$x\mapsto \{M\upharpoonright m: M\in x\}$ and showing that  
$h\upharpoonright  \mathfrak{Nr}_m\C_{l}$ 
is an isomorphism onto $\C_m$ \cite[p.80]{ANT}. Surjectiveness uses the condition $(J5)_l$ formulated in the second item  of  definition \ref{strongblur} of strong $l$-blurness.
Without this condition, that is if the $l$--blur $(J, E)$ is not strong, then still $\C_m$ and $\C_l$ can be defined because by definition $(J,E)$ is an $t$--blur
for all $m\leq t\leq l$, so $\sf Mat_t(\bf At)$ is a cylindric basis and for $t<l$ 
$\C_t$ embeds into $\Nr_m\C_l$ using the same above map, 
but this embedding might not be  surjective.

\end{example}

The following Theorem summarizes
the essence of construction in \cite{ANT} given above but neatly and methodically arranged and says some more new signifnificant facts.
We denote the relation algebra ${\mathfrak Bb}(\R, J, E)$ with atom structure $\bf At$ obtained by blowing up and blurring $\R$ 
(with underlying set is denoted by $At$ on \cite[p.73]{ANT}) by ${\sf split}(\R, J, E)$).  
By the same token,  we denote the algebra ${{\mathfrak Bb}}_l(\R, J, E)$ as defined in \cite[Top of p. 78]{ANT} by 
${\sf split}_l(\R, J, E)$. This switch of notation is motivated by the fact that we wish to emphasize the role of {\it splitting some (possibly all) atoms into infinitely subatoms 
during blowing up and blurring a finite algebra.}  We adopt  the same convention for all blow up and blur constructions encountered in what follows.

\begin{theorem}\label{ANT} Let $2<n\leq l<m\leq \omega$.
\begin{enumerate}
\item  Let $\R$ be a finite relation algebra with an $l$--blur $(J, E)$ where $J$ is the $l$--complex blur and $E$ is the index blur.  

(a) Let $\bf At$ be the relation algebra atom structure obtained by blowing up and blurring $\R$ as specified above.
Then the set of $l$ by $l$--dimensional matrices 
${\bf At}_{ca}={\sf Mat}_l({\bf At})$ is an $l$--dimensional cylindric basis, that is a weakly representable atom structure \cite[Theorem 3.2]{ANT}. 
The algebra  ${{\sf split}}_l(\R, J, E)$ with 
atom structure ${\bf At}_{ra}$  is in $\RCA_l$. Furthermore, 
$\R$ embeds into $\Cm{\bf At}$ which embeds into $\Ra\Cm({\bf At}_{ca}).$ 

(b) If $(J, E)$ is a strong $m$--blur for $\R$,
then $(J, E)$ is a strong $l$--blur for $\R$.  Furthermore, ${{{{\sf split}}}}_l(\R, J, E)\cong {\mathfrak{Nr}}_l{{{\sf split}}}_m(\R, J, E)$ and 
for any $l\leq j\leq m$, ${{\sf split}}(\R, J, E)$ having atom structure $\bf At$, is isomorphic to $\Ra({{{\sf split}}}_j(\R, J, E))$.

\item For every $n<l$, 
there is an $\R$ having a strong $l$--blur $(J, E)$
but no infinite representations (representations on an infinite base). 
Hence the atom structures defined in (a) of the previous item  (denoted by $\bf At$ and ${\bf At}_{ca}$)
for this specific $\R$ are not strongly representable.  

\item    Let $m<\omega$. If $\R$ is  a finite relation algebra having  
a strong $l$--blur, and no $m$--dimensional hyperbasis, 
then $l<m$.

\item If $n=l<m<\omega$ and $\R$ is a finite relation algebra with an $n$ blur $(J, E)$ (not necessarily strong) 
and no infinite $m$--dimensional hyperbasis, then the algebras $\Cm\At({{\sf split}}(\R, J, E))$ and $\Cm\At({{\sf split}}_l(\R, J, E))$ are outside
$\bold S\Ra\CA_m$ and  $\bold S{\sf Nr}_n\CA_m$, respectively, 
and the latter two varieties are not atom--canonical. 
\end{enumerate}
\end{theorem}
\begin{proof} \cite[Lemmata 3.2, 4.2, 4.3]{ANT}. We start by an outline of (a) of item (1).  Let $\R$ be as in the hypothesis. 
Let $3<n\leq l$. We blow up and blur $\R$. $\R$ is blown up by splitting all of the atoms each to infinitely many
defining an (infinite atoms) structure $\bf At$.
$\R$ is blurred by using a finite set of blurs (or colours) $J$. 
The term algebra denoted in \cite{ANT} by ${{{\sf split}}}(\R, J, E)$) over $\bf At$, 
 is representable using the finite number of blurs. Such blurs are basically non--principal ultrafilters; they are used as colours together 
with the principal ultrafilters (the atoms) to represent comletely the canonical extension of ${{\sf split}}(\R, J, E)$. 
Because $(J, E)$ is a complex set of $l$--blurs, this atom structure has an $l$--dimensional cylindric basis, 
namely, ${\bf At}_{ca}={\sf Mat}_l(\bf At)$. The resulting $l$--dimensional cylindric term algebra $\Tm{\sf Mat}_l(\bf At)$, 
and an algebra $\C$ having atom structure ${\bf At}_{ca}$ (denoted in \cite{ANT} by 
${\sf split}_l(\R, J, E)$) such that $\Tm{\sf Mat}_l({\bf At})\subseteq \C\ \subseteq \Cm{\sf Mat}_l(\bf At)$ 
is shown to be  representable.  
We prove (b) of item (1): Assume that the $m$--blur $(J, E)$ is strong, then by definition $(J, E)$ is a strong  $j$ blur for all $n\leq j\leq m$.
Furthermore,  by \cite[item (3) pp. 80]{ANT},  
${{\sf split}}(\R,J, E)=\Ra({{{\sf split}}}_j(\R, J, E))$ 
and ${{\sf split}}_j(\R, J, E)\cong \mathfrak{Nr}_j{{\sf split}}_m(\R, J, E)$. 

 (2):   Like in \cite[Lemma 5.1]{ANT},  one takes $l\geq 2n-1$, $k\geq (2n-1)l$, $k\in \omega$.
The Maddux integral relation algebra ${\mathfrak E}_k(2, 3)$ 
where $k$ is the number of non-identity atoms is the required $\R$. 

(3): Let $(J, E)$ be the strong $l$--blur of $\R$. Assume  for contradiction that $m\leq l$. Then we get by \cite[item (3), p.80]{ANT},  
that  $\A={{{\sf split}}}_n(\R, J, E)\cong \mathfrak{Nr}_n{{\sf split}}_l(\R, J, E)$.  But the cylindric $l$--dimensional algebra ${{{\sf split}}}_l(\R, J, E)$ is atomic,  having atom structure  
${\sf Mat}_l \At({{\sf split}}(\R, J, E))$, so $\A$ has an atomic $l$--dilation.
So $\A=\mathfrak{Nr}_n\D$ where $\D\in \CA_l$ is atomic.
But $\R\subseteq_c \mathfrak{Ra}\mathfrak{Nr}_n\D\subseteq_c \mathfrak{Ra}\D$. 
By \cite[Theorem 13.45 $(6)\iff (9)$]{HHbook},  $\R$ has a complete $l$--flat representation, 
thus it has a complete $m$--flat representation, because $m<l$ and $l\in \omega$. 
This is a contradiction. 

(4):   Let $\B={{{\sf split}}}_n(\R, J, E)$. Then, since $(J, E)$ is an $n$ blur, $\B\in \RCA_n$. But  
$\C=\Cm\At\B\notin \bold S{\sf Nr}_n\CA_{m}$, because $\R\notin \bold S\Ra\CA_m$, 
$\R$ embeds into ${{\sf split}}(\R, J, E)$ which, in turn, embeds into 
$\Ra\Cm\At\B$. Similarly, ${{\sf split}}(\R, J, E)\in \sf RRA$ and 
$\Cm(\At{{\sf split}}(\R, J, E))\notin \bold S\Ra\CA_m$.   
Hence the alledged varieties are not atom--canonical. 
\end{proof}

\section{Non-atom canonicity of any $\sf V$ between $\bold S\Nr_\alpha\CA_{\alpha+3}$ and $\RCA_\alpha$, for any ordinal $\alpha\geq 3$} 

In \cite{ANT} a single blow up and blur construction was used to prove non-atom--canonicity of $\sf RRA$ and ${\sf RCA}_n$ for $2<n<\omega$.
To obtain finer results, we use {\it two blow up and blur constructions} applied to rainbow algebras.
To put things into a unified perspective, we formulate a definition: 

\begin{definition}\label{blow} Let $\bold M$ be a variety 
of completely additive $\sf BAO$s.

(1) Let $\A\in \bold M$ be a finite algebra. We say that {\it $\D\in \bold M$ is obtained by blowing up and blurring $\A$} if $\D$ is atomic,  
$\A$ does not embed in $\D$, but $\A$ embeds into $\Cm\At\D$. 

(2) Assume that $\sf K\subseteq L\subseteq \bold M$, such that $\bold S\sf L=\sf L$.

(a) We say that {\it $\sf K$ is \underline{not} atom-canonical with respect to $\sf L$} if there exists an atomic $\D\in \sf K$ 
such that $\Cm\At\D\notin \sf L$.  In particular, $\sf K$ is not atom--canonical $\iff$ $\sf K$  not atom-canoincal with respect to itself.

(b)  We say that a finite algebra $\A\in \bold M$ {\it detects} that 
$\bold K$ is not atom--canonical with respect to $\sf L$, if $\A\notin \sf L$, 
and there is a(n atomic)  $\D\in \sf K$ 
obtained by blowing up and blurring $\A$.

\end{definition}

The most general exposition of $\CA$ rainbow constructions is given
in \cite[Section 6.2, Definition 3.6.9]{HHbook2} in the context of constructing atom structures from classes of models.
Our models are just coloured graphs \cite{HH}.
Let $\sf G$, $\sf R$ be two relational structures. Let $2<n<\omega$.
Then the colours used are:
\begin{itemize}

\item greens: $\g_i$ ($1\leq i\leq n-2)$, $\g_0^i$, $i\in \sf G$,

\item whites : $\w_i: i\leq n-2,$

\item reds:  $\r_{ij}$ $(i,j\in \sf R)$,

\item shades of yellow : $\y_S: S\text { a finite subset of } B$ or $S=B$.

\end{itemize}
A {\it coloured graph} is a graph such that each of its edges is labelled by the colours in the above first three items,
greens, whites or reds, and some $n-1$ hyperedges are also
labelled by the shades of yellow.
Certain coloured graphs will deserve special attention.
\begin{definition}
Let $i\in \sf G$, and let $M$ be a coloured graph  consisting of $n$ nodes
$x_0,\ldots,  x_{n-2}, z$. We call $M$ {\it an $i$ - cone} if $M(x_0, z)=\g_0^i$
and for every $1\leq j\leq n-2$, $M(x_j, z)=\g_j$,
and no other edge of $M$
is coloured green.
$(x_0,\ldots, x_{n-2})$
is called  the {\it base of the cone}, $z$ the {\it apex of the cone}
and $i$ the {\it tint of the cone}.
\end{definition}
The rainbow algebra depending on $\sf G$ and $\sf R$ from the class $\bold K$ consisting
of all coloured graphs $M$ such that:
\begin{enumerate}
\item $M$ is a complete graph
and  $M$ contains no triangles (called forbidden triples)
of the following types:
\begin{eqnarray}
&&\nonumber\\
(\g, \g^{'}, \g^{*}), (\g_i, \g_{i}, \w_i)
&&\mbox{any }1\leq i\leq  n-2,  \\
(\g^j_0, \g^k_0, \w_0)&&\mbox{ any } j, k\in \sf G,\\
\label{forb:match}(\r_{ij}, \r_{j'k'}, \r_{i^*k^*})&&\mbox{unless }i=i^*,\; j=j'\mbox{ and }k'=k^*
\end{eqnarray}
and no other triple of atoms is forbidden.

\item If $a_0,\ldots,   a_{n-2}\in M$ are distinct, and no edge $(a_i, a_j)$ $i<j<n$
is coloured green, then the sequence $(a_0, \ldots, a_{n-2})$
is coloured a unique shade of yellow.
No other $(n-1)$ tuples are coloured shades of yellow. Finally, if $D=\set{d_0,\ldots,  d_{n-2}, \delta}\subseteq M$ and
$M\upharpoonright D$ is an $i$ cone with apex $\delta$, inducing the order
$d_0,\ldots,  d_{n-2}$ on its base, and the tuple
$(d_0,\ldots, d_{n-2})$ is coloured by a unique shade
$\y_S$ then $i\in S.$
\end{enumerate}
Let $\sf G$ and $\sf R$ be relational structures as above. Take the set $\sf J$ consisting of all surjective maps $a:n\to \Delta$, where $\Delta\in \bold K$
and define an equivalence relation $\sim$  on this set  relating two such maps iff they essentially define the same graph \cite{HH};
the nodes are possibly different but the graph structure is the same.
Let $\At$ be the atom structure with underlying set $J\sim$. We denote the equivalence class of $a$ by $[a]$. Then define, for $i<j<n$,
the accessibility
relations corresponding to $ij$th--diagonal element, $i$th--cylindrifier, and substitution operator
corresponding to the transposition $[i,j]$, as follows:

(1) \ \  $[a]\in E_{ij} \text { iff } a(i)=a(j),$

(2) \ \ $[a]T_i[b] \text { iff }a\upharpoonright n\smallsetminus \{i\}=b\upharpoonright n\smallsetminus \{i\},$

 (3) \ \ $[a]S_{[i,j]}[b] \text { iff } a\circ [i,j]=b.$

This, as easily checked, defines a $\QEA_n$
atom structure. The complex $\QEA_n$ over this atom structure will be denoted by
$\A_{\sf G, \sf R}$. The dimension of $\A_{\sf G, \sf R}$, always finite and
$>2$, will be clear from context.
For rainbow atom structures, there is a one to one correspondence between atomic networks and coloured graphs \cite[Lemma 30]{HH}, 
so for $2<n<m\leq \omega$, we use the graph versions of the games $G^m_k$, $k\leq \omega$, played on rainbow atom 
structures of dimension $m$ \cite[pp.841--842]{HH}. 
We start by  translating the atomic $k$ rounded game game $G^m_k$  where the number of nodes are limited to $n$ 
to games on coloured graphs \cite[lemma 30]{HH}.

Let $\C$ be a rainbow algebra. Let $N$ be an atomic $\C$ network.
Let $x,y$ be two distinct nodes occurring in the
$n$ tuple $\bar z$. $N(\bar z)$ is an atom of $\C$
which defines an edge colour of
$x,y$. Using the fact that the dimension is at least $3$,
the edge colour depends only on $x$ and $y$
not on the other elements of
$\bar z$ or the positions of $x$ and $y$ in $\bar z$.
Similarly $N$ defines shades of white for certain $(n-1)$ tuples.  In this way $N$
translates
into a coloured graph.
This translation has an inverse. More precisely, letting $\sf CRG$ be the class of coloured graphs in a rainbow signature, we have:

Let $M\in \sf CRG$ be arbitrary. Define $N_{M}$
whose nodes are those of $M$
as follows. For each $a_0,\ldots, a_{n-1}\in M$, define
$N_{M}(a_0,\ldots,  a_{n-1})=[\alpha]$
where $\alpha: n\to M\upharpoonright \set{a_0,\ldots, a_{n-1}}$ is given by
$\alpha(i)=a_i$ for all $i<n$.
Then, as easily checked,  $N_{M}$ is an atomic $\C$ network.
Conversely, let $N$ be any non empty atomic $\C$ networkDefine a complete coloured graph $M_N$
whose nodes are the nodes of $N$ as follows:

\begin{itemize}
\item For all distinct $x,y\in M_N$ and edge colours $\eta$, $M_N(x,y)=\eta$
if and only if  for some $\bar z\in {}^nN$, $i,j<n$, and atom $[\alpha]$, we have
$N(\bar z)=[\alpha]$, $z_i=x$, $z_j=y$ and the edge $(\alpha(i), \alpha(j))$
is coloured $\eta$ in the graph $\alpha$.

\item For all $x_0,\ldots, x_{n-2}\in {}^{n-1}M_N$ and all yellows $\y_S$,
$M_N(x_0,\ldots,  x_{n-2})= \y_S$ $\iff$
for some $\bar z$ in $^nN$, $i_0,\ldots,  i_{n-2}<n$
and some atom $[\alpha]$, we have
$N(\bar z)=[\alpha]$, $z_{i_j}=x_j$ for each $j<n-1$ and the $n-1$ tuple
$\langle \alpha(i_0),\ldots, \alpha(i_{n-2})\rangle$ is coloured
$\y_S.$ Then $M_N$ is well defined and is in $\sf CRG$.
\end{itemize}

The following is then, though tedious and long,  easy to  check:
For any $M\in \sf CGR$, we have  $M_{N_{M}}=M$,
and for any $\C$ network
$N$, $N_{{M}_N}=N.$
This translation makes the following equivalent formulation of the
games $G^m_mk(\At\C)$ originally defined on networks.
The new  graph version of
the game \cite[p.27--29]{HH} builds a nested sequence $M_0\subseteq M_1\subseteq M_i\ldots, i<k$ $(k$ the number of rounds $\leq \omega$)
of coloured graphs such that $\nodes(M_i)\subseteq m$.\\
\pa\ picks a graph $M_0\in \sf CRG$ with $M_0\subseteq m$ and
$\exists$ makes no response
to this move. In a subsequent round, let the last graph built be $M_i$ $(i<k)$.
$\forall$ picks
\begin{itemize}
\item a graph $\Phi\in \G$ with $|\Phi|=n,$
\item a single node $t\in \Phi,$
\item a coloured graph embedding $\theta:\Phi\smallsetminus \{t\}\to M_i.$
Let $F=\phi\smallsetminus \{t\}$. Then \pe\ must respond by amalgamating
$M_i$ and $\Phi$ with the embedding $\theta$. In other words, she has to define a
graph $M_{i+1}\in \sf CRG$ and embeddings $\lambda:M_i\to M_{i+1}$
$\mu:\phi \to M_{i+1}$, such that $\lambda\circ \theta=\mu\upharpoonright F.$
\end{itemize}
Summarizing we have:
\begin{theorem}
Let $2<n<\omega$. Let $k, m\leq \omega$, and $\C$ be a rainbow $\CA_n$. Then \pe\ has a \ws\ in $G_k^m(\At\C)\iff$ \pe\ has a \ws\ in the above $k$--rounded graph game played 
on $\C$  where the size of graphs during the play is limited to $m$ nodes.
\end{theorem}
The typical \ws\ for \pa in the graph version of this game is bombarding \pe\ with cones having 
a common base and {\it green} tints untill she runs out of (suitable) {\it reds}, that is to say, reds whose indicies do not match \cite[4.3]{HH}.
\begin{definition}\label{sub} \begin{enumerate}
\item Let $m$ be a finite ordinal $>0$. An $\sf s$ word is a finite string of substitutions $({\sf s}_i^j)$ $(i, j<m)$,
a $\sf c$ word is a finite string of cylindrifications $({\sf c}_i), i<m$;
an $\sf sc$ word $w$, is a finite string of both, namely, of substitutions and cylindrifications.
An $\sf sc$ word
induces a partial map $\hat{w}:m\to m$:
\begin{itemize}

\item $\hat{\epsilon}=Id,$

\item $\widehat{w_j^i}=\hat{w}\circ [i|j],$

\item $\widehat{w{\sf c}_i}= \hat{w}\upharpoonright(m\smallsetminus \{i\}).$

\end{itemize}
If $\bar a\in {}^{<m-1}m$, we write ${\sf s}_{\bar a}$, or
${\sf s}_{a_0\ldots a_{k-1}}$, where $k=|\bar a|$,
for an  arbitrary chosen $\sf sc$ word $w$
such that $\hat{w}=\bar a.$
Such a $w$  exists by \cite[Definition~5.23 ~Lemma 13.29]{HHbook}.

\item Fix $2<n<m$. Assume that $\C\in\Sc_m$, $\A\subseteq_c{\mathfrak Nr}_n\C$ is an
atomic $\Sc_n$ and $N$ is an $\A$--network with $\nodes(N)\subseteq m$. Define
$N^+\in\C$ by
\[N^+ =
 \prod_{i_0,\ldots, i_{n-1}\in\nodes(N)}{\sf s}_{i_0, \ldots, i_{n-1}}{}N(i_0,\ldots, i_{n-1}).\]
\end{enumerate}
\end{definition}
\begin{lemma}\label{n}
Let $2<n<m$. If either:
\begin{itemize}
\item $\A\in \Sc_n$ and $\A\in S_c\Nr_n\Sc_m^{\sf ad}$ or, 
\item $\A\in \QA_n$ and $\A\subseteq_c \mathfrak{Nr}_n\C$, $\C\in \QA_m$ and $\s_0^1$ is completely additive in $\C$ or,
\item $\K$ is any class having signature between $\CA$ and $\QEA$, 
$\A\in \K_n$ and $\A\in \bold S_c\Nr_n\K_m$,
\end{itemize} 
then \pe\ has a \ws\ in $\bold G^m(\At\A).$ 
\end{lemma}

\begin{proof}
We assume that the ${\sf s}_i^j$s for  $i< j<m$ are completely additive in $\C$. (This condition is superfluous for any $\K$ between $\CA$ and $\QEA$.)
Then the following hold:
\begin{enumerate}
\item for all $x\in\C\setminus\set0$ and all $i_0, \ldots, i_{n-1} < m$, there is $a\in\At\A$, such that
${\sf s}_{i_0,\ldots, i_{n-1}}a\;.\; x\neq 0$,

\item for any $x\in\C\setminus\set0$ and any
finite set $I\subseteq m$, there is a network $N$ such that
$\nodes(N)=I$ and $x\cdot N^+\neq 0$. Furthermore, for any networks $M, N$ if
$M^+\cdot N^+\neq 0$, then
$M\restr {\nodes(M)\cap\nodes(N)}=N\restr {\nodes(M)\cap\nodes(N)},$

\item if $\theta$ is any partial, finite map $m\to m$
and if $\nodes(N)$ is a proper subset of $m$,
then $N^+\neq 0\rightarrow {(N\theta)^+}\neq 0$. If $i\not\in\nodes(N),$ then ${\sf c}_iN^+=N^+$.

\end{enumerate}
Using the above facts,  we are now ready to show that \pe\  has a \ws\ in $F^m$. She can always
play a network $N$ with $\nodes(N)\subseteq m,$ such that
$N^+\neq 0$.
In the initial round, let \pa\ play $a\in \At\A$.
\pe\ plays a network $N$ with $N(0, \ldots, n-1)=a$. Then $N^+=a\neq 0$.
Recall that here \pa\ is offered only one (cylindrifier) move.
At a later stage, suppose \pa\ plays the cylindrifier move, which we denote by
$(N, \langle f_0, \ldots, f_{n-2}\rangle, k, b, l).$
He picks a previously played network $N$,  $f_i\in \nodes(N), \;l<n,  k\notin \{f_i: i<n-2\}$,
such that $b\leq {\sf c}_l N(f_0,\ldots,  f_{i-1}, x, f_{i+1}, \ldots, f_{n-2})$ and $N^+\neq 0$.
Let $\bar a=\langle f_0\ldots f_{i-1}, k, f_{i+1}, \ldots f_{n-2}\rangle.$
Then by  second part of  (3)  we have that ${\sf c}_lN^+\cdot {\sf s}_{\bar a}b\neq 0$
and so  by first part of (2), there is a network  $M$ such that
$M^+\cdot{\sf c}_{l}N^+\cdot {\sf s}_{\bar a}b\neq 0$.
Hence $M(f_0,\dots, f_{i-1}, k, f_{i-2}, \ldots$ $, f_{n-2})=b$,
$\nodes(M)=\nodes(N)\cup\set k$, and $M^+\neq 0$, so this property is maintained.
Assume that $\Rd_{qa}\C\subseteq_c\Nr_n\D$ for some $\D\in \QA_{n+3}$ where only 
${\sf s}_0^1{}^{D}$ is completely additive.
Then every substitution operation corresponding to a replacement in $\D$, can be obtained from
a composition of finitely many substitution operations involving only one replacement 
${\sf s}_0^1$  and all the rest are substitution operations that correspond to transpositions.
To prove this, we can assume without loss that $i\neq 0, 1$. 
Computing we get: 
$${\sf s}_{[1,i]}{\sf s}_0^1x
={\sf s}_0^i{\sf s}_{[1, i]}x\text { so }{\sf s}_{[1,i]}{\sf s}_0^1{\sf s}_{[1, i]}x={\sf s}_0^i{\sf s}_{[1, i]}{\sf s}_{[1,i]}x={\sf s}_0^ix,$$
$${\sf s}_i^0x={\sf s}_{[0,i]}{\sf s}_0^1x
={\sf s}_{[i,0]}{\sf s}_{[1, i]}{\sf s}_0^1{\sf s}_{[i,1]}x  
\text { and }{\sf s}_{[0, j]}{\sf s}_i^0x 
={\sf s}_i^j {\sf s}_{[0,j]}x.$$
Continuing the computation:
\begin{align*}
{\sf s}_i^jx&={\sf s}_{[0,j]}{\sf s}_{[0,j]}{\sf s}_i^jx\\
&={\sf s}_{[0,j]}{\sf s}_i^0{\sf s}_{[0,j]}x\\
&={\sf s}_{[i,0]} {\sf s}_{[1,i]}  {\sf s}_0^1 {\sf s}_{[1,i]} {\sf s}_{[0,j]}x.
\end{align*}
We have shown that:
$${\sf s}_i^j={\sf s}_{[0, j]}\circ {\sf s}_{[i,0]}\circ {\sf s}_{[1,i]}\circ  {\sf s}_0^1\circ {\sf s}_{[1,i]}\circ {\sf s}_{[0,j]}.$$

All such substitution operations are completely additive,
the ones involving transpositions are in fact self--conjugate,  
hence we get that $\D$ is completely additive. 
By complete additivity for $\CA$s and $\QEA$s, we get the second and 
third required and we are done. 
\end{proof}

With these preliminaries out of the way we are now ready  to start digging deeper:
\begin{theorem}\label{can2} Let $2<n<\omega$.
Let $\K$ be any variety between $\Sc$ and $\QEA$. Let $m\geq n+3$. Then ${\sf RK}_n$ is not-atom canonical 
with respect to $\bold S{\sf Nr}_n\K_m$. In fact, there is a countable atomic simple $\A\in {\sf RQEA}_n$ 
such that $\Rd_{sc}\Cm\At\A$ 
does not have an $n+3$-square,{\it a fortiori} $n+3$- flat,  representation.
\end{theorem}
\begin{proof} 

The proof is divided into four parts:

1: {\bf Blowing up and blurring  $\A_{n+1, n}$ forming a weakly representable atom structure $\bf At$}:
Take the finite rainbow ${\sf QEA}_n$,  $\A_{n+1, n}$
where the reds $\sf R$ is the complete irreflexive graph $n$, and the greens
are  $\{\g_i:1\leq i<n-1\}
\cup \{\g_0^{i}: 1\leq i\leq n+1\}$, endowed with the quasi--polyadic equality operations.
We will show that for any variety ${\sf K}_n$ between ${\sf Sc}_n$ and $\QEA_n$, the $\K_n$ reduct of $\A_{n+1, n}$ detects that ${\sf RK}_n$ is not atom-canonical with respect to $\bold S\Nr_n\K_{n+3}$. 
Denote the finite atom structure of $\A_{n+1, n}$ by ${\bf At}_f$; 
so that ${\bf At}_f=\At(\A_{n+1, n})$.
One  then replaces the red colours 
of the finite rainbow algebra of $\A_{n+1, n}$ each by  infinitely many reds (getting their superscripts from $\omega$), obtaining this way a weakly representable atom structure $\bf At$.
The cylindric reduct of the resulting atom structure after `splitting the reds', namely, $\bf At$,  is 
like the weakly (but not strongly) representable 
atom structure of the atomic, countable and simple algebra $\A$ as defined in \cite[Definition 4.1]{Hodkinson}; the sole difference is that we have $n+1$ greens
and not $\omega$--many as is the case in \cite{Hodkinson}; also we count it the polyadic operations of subtitutions. 
We denote the resulting term $\QEA_n$,  $\Tm\bf At$ by ${{\mathfrak{Bb}}}(\A_{n+1, n}, \r, \omega)$ short hand for  blowing 
up and blurring $\A_{n+1,n}$ by splitting each {\it red graph (atom)} into $\omega$ many.  
It can be shown exactly like in \cite{Hodkinson} that \pe\ can win the rainbow $\omega$--rounded game
and build an $n$--homogeneous model $\Mo$ by using a shade of red $\rho$ {\it outside} the rainbow signature, when
she is forced a red;  \cite[Proposition 2.6, Lemma 2.7]{Hodkinson}. The $n$-homogeniuty  entails that any subgraph (substructure)
of $\Mo$ of size $\leq n$, is independent of its location in $\Mo$;
it is uniquely determined by its isomorphism type.
One proves like in {\it op.cit} 
that  ${{\mathfrak{Bb}}}(\A_{n+1, n}, \r, \omega)$ is representable as a set algebra having top element $^n\Mo$.

We give more details. In the present context, after the splitting `the finitely many red colours' replacing each such red colour $\r_{kl}$, $k<l<n$  by $\omega$ many 
$\r_{kl}^i$, $i\in \omega$, the rainbow signature for the resulting rainbow theory as defined in \cite[Definition 3.6.9]{HHbook} call this theory $T_{ra}$,
 consists of $\g_i: 1\leq i<n-1$, $\g_0^i: 1\leq i\leq n+1$,
$\w_i: i<n-1$,  $\r_{kl}^t: k<l< n$, $t\in \omega$,
binary relations, and $n-1$ ary relations $\y_S$, $S\subseteq_{\omega} n+k-2$ or $S=n+1$. 
The set algebra ${\mathfrak{Bb}}(\A_{n+1, n}, \r, \omega)$ of dimension $n$ has
base an $n$--homogeneous  model $\Mo$ of another theory $T$ whose signature expands that of 
$T_{ra}$ by an additional binary relation (a shade of red) $\rho$.  
In this new signature $T$ is obtained from $T_{ra}$ by 
some axioms  (consistency conditions) extending $T_{ra}$. Such axioms (consistency conditions) 
specify consistent triples involving $\rho$. We call the models of $T$ {\it extended} coloured graphs. 
In particular, $\Mo$ is an extended coloured graph.
To build $\Mo$, the class of coloured graphs is considered in
the signature $L\cup \{\rho\}$ like in uual rainbow constructions as given above with the two additional forbidden triples
$(\r, \rho, \rho)$ and $(\r, \r^*, \rho)$, where $\r, \r^*$ are any reds. 
This model $\Mo$ is constructed as a countable limit of finite models of $T$ 
using a game played between \pe\ and \pa.   Here, unlike the extended $L_{\omega_1, \omega}$ theory 
dealt with in \cite{Hodkinson},  $T$ is a {\it first order one}
because the number of greens used are finite.
In the rainbow game \cite{HH, HHbook} 
\pa\ challenges \pe\  with  {\it cones} having  green {\it tints $(\g_0^i)$}, 
and \pe\ wins if she can respond to such moves. This is the only way that \pa\ can force a win.  \pe\ 
has to respond by labelling {\it appexes} of two succesive cones, having the {\it same base} played by \pa.
By the rules of the game, she has to use a red label. She resorts to  $\rho$ whenever
she is forced a red while using the rainbow reds will lead to an inconsistent triangle of reds;  \cite[Proposition 2.6, Lemma 2.7]{Hodkinson}. 

2. {\bf Representing $\Tm\At\A$ (and its completion) as (generalized) set algebras:} From now on, forget about $\rho$; having done its task as a colour  to (weakly) represent $\bf At$, it will play no further role.
Having $\Mo$ at hand, one constructs  two atomic $n$--dimensional set algebras based on $\Mo$, sharing the same atom structure and having 
the same top element.  
The atoms of each will be the set of coloured graphs, seeing as how, quoting Hodkinson \cite{Hodkinson} such coloured graphs are `literally indivisible'. 
Now $L_n$ and $L_{\infty, \omega}^n$ are taken in the rainbow signature (without $\rho$). Continuing like in {\it op.cit}, deleting the one available red shade, set
$W = \{ \bar{a} \in {}^n\Mo : \Mo \models ( \bigwedge_{i < j <n} \neg \rho(x_i, x_j))(\bar{a}) \},$
and for $\phi\in L_{\infty, \omega}^n$, let
$\phi^W=\{s\in W: \Mo\models \phi[s]\}.$
Here $W$ is the set of all $n$--ary assignments in
$^n\Mo$, that have no edge labelled by $\rho$.
Let $\A$  be the relativized set algebra with domain
$\{\varphi^W : \varphi \,\ \textrm {a first-order} \;\ L_n-
\textrm{formula} \}$  and unit $W$, endowed with the
usual concrete quasipolyadic operations read off the connectives.
Classical semantics for {\it $L_n$ rainbow formulas} and their
semantics by relativizing to $W$ coincide \cite[Proposition 3.13]{Hodkinson} {\it but not with respect to 
$L_{\infty,\omega}^n$ rainbow formulas}.
This depends essentially on \cite[Lemma 3.10]{Hodkinson}, which is the heart and soul of the proof in \cite{Hodkinson}, and for what matters this proof.
The referred to lemma says that any permutation $\chi$ of $\omega\cup \{\rho\}$,
$\Theta^{\chi}$ as defined in  \cite[Definitions 3.9, 3.10]{Hodkinson} is an $n$ back--and--forth system
induced by any permutation of $\omega\cup \{\rho\}$. 
Let $\chi$ be such a permutation. The system $\Theta^{\chi}$ consists of isomorphisms between coloured graphs such that 
superscripts of reds are `re-shuffled along $\chi$', in such a way that rainbow red labels are permuted,  
$\rho$ is replaced by a red rainbow label, and all other colours are preserved.
One uses such $n$-back-and-forth systems mapping a 
tuple $\bar{b} \in {}^n \Mo \backslash W$ to a tuple
$\bar{c} \in W$ preserving any formula in $L_{ra}$ containing the non-red symbols that are
`moved' by the system, so if $\bar{b}\in {}^n\Mo$ refutes the $L_n$ rainbow formula  $\phi$, then there is a $\bar{c}$ in $W$ 
refuting $\phi$. 
Hence the set algebra $\A$ is isomorphic to a quasi-polyadic equality  set algebra of dimension $n$ 
having top element $^n\Mo$, so $\A$
is simple, in fact its $\Df$ reduct is simple.
Let $\E=\{\phi^W: \phi\in L_{\infty, \omega}^n\}$
\cite[Definition 4.1]{Hodkinson}
with the operations defined like on $\A$ the usual way. $\Cm\bf At$ is a complete $\QEA_n$ and, so like in \cite[Lemma 5.3]{Hodkinson}
we have an isomorphism from $\Cm\bf At$  to $\E$ defined
via $X\mapsto \bigcup X$.
Since $\At\A=\At\Tm(\At\A)=\bf At$  
and $\Tm\At\A\subseteq \A$, hence $\Tm\At\A$ is representable.
The atoms of $\A$, $\Tm\At\A$ and $\Cm\At\A=\Cm \bf At$ are the coloured graphs whose edges are {\it not labelled} by $\rho$.
These atoms are uniquely determined by the interpretion in $\Mo$ of so-called $\sf MCA$ formulas in the rainbow signature of $\bf At$  as in
\cite[Definition 4.3]{Hodkinson}.

3. {\bf Embedding $\A_{n+1, n}$ into $\Cm(\At({{\mathfrak{Bb}}}(\A_{n+1, n}, \r, \omega)))$:} Let ${\sf CRG}_f$ be the class of coloured graphs on 
${\bf At}_f$ and $\sf CRG$ be the class of coloured graph on $\bf At$. We 
can (and will) assume that  ${\sf CRG}_f\subseteq \sf CRG$.
Write $M_a$ for the atom that is the (equivalence class of the) surjection $a:n\to M$, $M\in \sf CGR$.
Here we identify $a$ with $[a]$; no harm will ensue.
We define the (equivalence) relation $\sim$ on $\bf At$ by
$M_b\sim N_a$, $(M, N\in {\sf CGR}):$
\begin{itemize}
\item $a(i)=a(j)\Longleftrightarrow b(i)=b(j),$

\item $M_a(a(i), a(j))=\r^l\iff N_b(b(i), b(j))=\r^k,  \text { for some $l,k$}\in \omega,$

\item $M_a(a(i), a(j))=N_b(b(i), b(j))$, if they are not red,

\item $M_a(a(k_0),\dots, a(k_{n-2}))=N_b(b(k_0),\ldots, b(k_{n-2}))$, whenever
defined.
\end{itemize}
We say that $M_a$ is a {\it copy of $N_b$} if $M_a\sim N_b$ (by symmetry $N_b$ is a copy of $M_a$.) 
Indeed, the relation `copy of' is an equivalence relation on $\bf At$.  An atom $M_a$ is called a {\it red atom}, if $M_a$ has at least one red edge. 
Any red atom has $\omega$ many copies, that are {\it cylindrically equivalent}, in the sense that, if $N_a\sim M_b$ with one (equivalently both) red,
with $a:n\to N$ and  $b:n\to M$, then we can assume that $\nodes(N) =\nodes(M)$ 
and that for all $i<n$, $a\upharpoonright n\sim\{i\}=b\upharpoonright n\sim \{i\}$.
Any red atom has $\omega$ many copies that are {\it cylindrically equivalent}, in the sense that, if $N_a\sim M_b$ with one (equivalently both) red,
with $a:n\to N$ and  $b:n\to M$, then we can assume that $\nodes(N) =\nodes(M)$ 
and that for all $i<n$, $a\upharpoonright n\sim\{i\}=b\upharpoonright n\sim \{i\}$.
In $\Cm\bf At$, we write $M_a$ for $\{M_a\}$ 
and we denote suprema taken in $\Cm\bf At$, possibly finite, by $\sum$.
Define the map $\Theta$ from $\A_{n+1, n}=\Cm{\bf At}_f$ to $\Cm\bf At$,
by specifing first its values on ${\bf At}_f$,
via $M_a\mapsto \sum_jM_a^{(j)}$ where $M_a^{(j)}$ is a copy of $M_a$. 
So each atom maps to the suprema of its  copies.  
This map is well-defined because $\Cm\bf At$ is complete. 
We check that $\Theta$ is an injective homomorphim. Injectivity follows from $M_a\leq \Theta(M_a)$, hence $\Theta(x)\neq 0$ 
for every atom $x\in \At(\A_{n+1, n})$.
We check preservation of all the $\QEA_n$ operations.  
The Boolean join is obvious.
\begin{itemize}
\item For complementation: It suffices to check preservation of  complementation `at atoms' of ${\bf At}_f$. 
So let $M_a\in {\bf At}_f$ with $a:n\to M$, $M\in {\sf CGR}_f\subseteq \sf CGR$. Then: 
$$\Theta(\sim M_a)=\Theta(\bigcup_{[b]\neq [a]} M_b)
=\bigcup_{[b]\neq [a]} \Theta(M_b)
=\bigcup_{[b]\neq [a]}\sum_j M_b^{(j)}$$
$$=\bigcup_{[b]\neq [a]}\sim \sum_j[\sim (M_a)^{(j)}]
=\bigcup_{[b]\neq [a]}\sim \sum_j[(\sim M_b)^j]
=\bigcup_{[b]\neq [a]}\bigwedge_j M_b^{(j)}$$
$$=\bigwedge_j\bigcup_{[b]\neq [a]}M_b^{(j)}
=\bigwedge_j(\sim M_a)^{j}
=\sim (\sum M_a^j)
=\sim \Theta(a)$$

\item Diagonal elements. Let $l<k<n$. Then:
\begin{align*}
M_x\leq \Theta({\sf d}_{lk}^{\Cm{\bf At}_f})&\iff\ M_x\leq \sum_j\bigcup_{a_l=a_k}M_a^{(j)}\\
&\iff M_x\leq \bigcup_{a_l=a_k}\sum_j M_a^{(j)}\\
&\iff  M_x=M_a^{(j)}  \text { for some $a: n\to M$ such that $a(l)=a(k)$}\\
&\iff M_x\in {\sf d}_{lk}^{\Cm\bf At}.
\end{align*}

\item Cylindrifiers. Let $i<n$. By additivity of cylindrifiers, we restrict our attention to atoms 
$M_a\in {\bf At}_f$ with $a:n\to M$, and $M\in {\sf CRG}_f\subseteq \sf CRG$. Then: 

$$\Theta({\sf c}_i^{\Cm{\bf At}_f}M_a)=f (\bigcup_{[c]\equiv_i[a]} M_c)
=\bigcup_{[c]\equiv_i [a]}\Theta(M_c)$$
$$=\bigcup_{[c]\equiv_i [a]}\sum_j M_c^{(j)}=\sum_j \bigcup_{[c]\equiv_i [a]}M_c^{(j)}
=\sum _j{\sf c}_i^{\Cm\bf At}M_a^{(j)}$$
$$={\sf c}_i^{\Cm\bf At}(\sum_j M_a^{(j)})
={\sf c}_i^{\Cm\bf At}\Theta(M_a).$$

\item Substitutions: Let $i, k<n$. By additivity of the ${\sf s}_{[i,k]}$s, we again restrict ourselves to atoms of the form $M_a$ as specified in the previous items.
Now computing we get:
$\Theta({\sf s}_{[i,k]}^{\Cm{\bf At}_f} M_a)= \Theta(M_{a\circ [i,k]})=  \sum_j^{\Cm\bf At}(M_{a\circ [i,k]}^{(j)})=\sum_j {\sf s}_{[i,k]}^{\Cm\bf At}{} M_a^{(j)}={\sf s}_{[i,k]}^{\Cm \bf At}{}(\sum_{j}M_a^{(j)})=
{\sf s}_{[i,k]}^{\Cm\bf At}{}\Theta(M_a).$
\end{itemize}

4.: {\bf \pa\ has  a  \ws\ in $G^{n+3}\At(\A_{n+1, n})$; and the required result:} It is straightforward to show that 
\pa\ has \ws\ first in the  \ef\ forth  private 
game played between \pe\ and \pa\ on the complete
irreflexive graphs $n+1$ and $n$ in 
$n+1$ rounds
${\sf EF}_{n+1}^{n+1}(n+1, n)$ \cite [Definition 16.2]{HHbook2}
since $n+1$ is `longer' than $n$. 
Here $r$ is the number of rounds and $p$ is the number of pairs of pebbles
on board. Using (any) $p>n$ many pairs of pebbles avalable on the board \pa\ can win this game in $n+1$ many rounds.
In each round $0,1\ldots n$, \pe\ places a new pebble  on  a new element of $n+1$.
The edge relation in $n$ is irreflexive so to avoid losing
\pe\ must respond by placing the other  pebble of the pair on an unused element of $n$.
After $n$ rounds there will be no such element, so she loses in the next round.
 \pa\  lifts his \ws\ from the private \ef\ forth game ${\sf EF}_{n+1}^{n+1}(n+1, n)$ to the graph game on ${\bf At}_f=\At(\A_{n+1,n})$ 
\cite[pp. 841]{HH} forcing a
win using $n+3$ nodes. 
He bombards \pe\ with cones
having  common
base and distinct green  tints until \pe\ is forced to play an inconsistent red triangle (where indicies of reds do not match).
Let $\Rd_{sc}$ denote the `$\Sc$ reduct'. 
For brevity let $\D=\Rd_{sc}\A_{n+1, n}(\in \Sc_n)$. Then by Lemma \ref{n}, since $\D$ is finite, then $\D$ does not have an $n+3$-square representation. 
But $\D$ embds into $\Rd_{sc}\Cm\bf At$ it follows that the last dos not have an $n+3$ square represenation either. 
\end{proof}


\begin{theorem} \label{can} Let $2<n<\omega$.
\begin{enumerate}

\item For any ordinal $\alpha\geq 3$ (possibly infinite), there exists $\A\in \RCA_\alpha$ such that $\Cm\At\A\notin \bold S\Nr_\alpha\CA_{\alpha+3},$

\item There exists $\A\in \Nr_n\CA_{l}\cap \RCA_n$ such that $\Cm\At\A\notin \RCA_n,$

\item There exists $\B\in{\sf  Cs}_n$, $\B\notin {\bf El}\Nr_n\CA_{n+1}$, but $\At\B\in \Nr_n\CA_{\omega}$ and $\Cm\At\B\in \Nr_n\CA_{\omega}$
 \end{enumerate}

\end{theorem}

\begin{proof}

Now we lift the last result to the transfnite. We consider for simplicity of notation only the infinite ordinal $\omega$.
For each finite $k\geq 3$, let $\A(k)$
be an atomic countable simple representable
$\CA_k$ such that  $\B(k)=\Cm\At\A(k)\notin \bold S\Nr_k\CA_{k+3}.$ We know that such algebras exist by the above.
{\it It can be that : (*)    $\B_m$ embeds into $\Rd_m\B_{t}$, whenever $3\leq m<t<\omega$.}
Throughout this part of the proof $F$ denote a non--principal ultrafilter on $\omega\setminus 3$.
For each finite $k\geq 3$, let $\A(k)$ and $\B(k)$ be the algebras construsted in the finite dimensional case (of dimension $k$). 
Let $\A_k$ be an (atomic) algebra having the signature of $\CA_{\omega}$
such that $\Rd_k\A_k=\A(k)$.
Analogously, let $\B_k$ be an algebra having the signature
of $\CA_{\omega}$ such that $\Rd_k\B_k=\B(k)$, and we require in addition that $\B_k=\Cm(\At\A_k)$.
We use a lifting argument using ultraproducts. Let $\B=\Pi_{i\in \omega\setminus 3}\B_i/F$.
It is easy to show that 
$\A=\Pi_{i\in \omega\setminus 3}\A_i/F\in \RCA_{\omega}$.
Furthermore, a direct computation gives:
$\Cm\At\A=\Cm(\At[\Pi_{i\in \omega\setminus 3 }\A_i/F])
=\Cm[\Pi_{i\in \omega\setminus 3}(\At\A_i)/F)]
=\Pi_{i\in \omega\setminus 3}(\Cm(\At\A_i)/F)
=\Pi_{i\in \omega\setminus 3}\B_i/F
=\B.$
By the same token, $\B\in \CA_{\omega}$. 
Assume for contradiction that $\B\in \bold S\Nr_{\omega}\CA_{\omega+3}$.
Then $\B\subseteq \mathfrak{Nr}_{\omega}\C$ for some $\C\in \CA_{\omega+3}$.
Let $3\leq m<\omega$ and  let $\lambda:m+3\rightarrow \omega+3$ be the function defined by $\lambda(i)=i$ for $i<m$
and $\lambda(m+i)=\omega+i$ for $i<3$.
Then we get (**): $\Rd^\lambda\C\in \CA_{m+3}$ and $\Rd_m\B\subseteq \mathfrak{Nr}_m\Rd^\lambda\C$. 
By assumption 
let $I_t: \B_m\to \Rd_m\B_t$ be an injective homomorphism for $3\leq m<t<\omega$.
Let $\iota( b)=(I_{t}b: t\geq m )/F$ for  $b\in \B_m$.
Then $\iota$ is an injective homomorphism that embeds $\B_m$ into
$\Rd_m\B$.  By (**)  we know that $\Rd_{m}\B\in {\bf S}\Nr_m\CA_{m+3}$, hence  $\B_m\in \bold S\Nr_{m}\CA_{m+3}$, too.
This is a contradiction, and we are done.

2.  We use the construction in \cite{ANT}. 
Let $\R$ be a relation algebra, with non--identity atoms $I$ and $2<n<\omega$. Assume that  
$J\subseteq \wp(I)$ and $E\subseteq {}^3\omega$.
$(J, E)$  is an {\it $n$--blur} for $\R$, if $J$ is a {\it complex $n$--blur}
and the tenary relation $E$ is an {\it index blur} defined  as 
in item (ii) of \cite[Definition 3.1]{ANT}.
We say that $(J, E)$ is a {\it strong $n$--blur}, if it $(J, E)$ is an $n$--blur,  such that the complex 
$n$--blur  satisfies:
$(\forall V_1,\ldots V_n, W_2,\ldots W_n\in J)(\forall T\in J)(\forall 2\leq i\leq n)
{\sf safe}(V_i,W_i,T)$ (with notation as in \cite{ANT}).  
Now Let $l\geq 2n-1$, $k\geq (2n-1)l$, $k\in \omega$. 
One takes the finite
integral relation algebra $\R_l={\mathfrak E}_k(2, 3)$ 
where $k$ is the number of non-identity atoms in
$\R_l$. Then $\R_l$ has a strong $l$--blur, $(J, E)$ and it can only be represented 
on a finite basis \cite{ANT}.  Then $\Bb_n(\R_l, J, E)=\Nr_n\Bl_l(\R_l, J, E)$ has no complete representation, so
$\Cm\At\Bb_n(\R_l, J, E)$ is not representable.

3.  Let $V={}^n\Q$ and let ${\A}\in {\sf Cs}_n$ has universe $\wp(V)$.
Then clearly $\A\in {\sf Nr}_{n}\CA_{\omega}$. To see why,  let 
$W={}^{\omega}\Q$ and let $\D\in {\sf Cs}_{\omega}$ have universe $\wp(W)$.
Then the map $\theta: \A\to \wp(\D)$ defined via $a\mapsto \{s\in W: (s\upharpoonright \alpha)\in a\}$, 
is an injective homomorphism from $\A$ into $\mathfrak{Rd}_{n}\D$ that is onto 
$\mathfrak{Nr}_{n}\D$.
Let $y$ denote the following $n$--ary relation:
$y=\{s\in V: s_0+1=\sum_{i>0} s_i\}.$ Let $y_s$ be the singleton containing $s$, i.e. $y_s=\{s\}$
and ${\B}=\Sg^{\A}\{y,y_s:s\in y\}.$ It is shown in \cite{SL} that $\{s\}\in \B$, for all $s\in V$. 
Now $\B$ and $\A$ having same top element $V$, share the same atom structure, namely, the singletons, so $\B\subseteq_ d \A$ 
and $\Cm\At\B=\A$. Furthermore, plainly $\A, \B\in {\sf CRCA}_n$; the identity maps establishes a complete representation for both, 
since $\bigcup_{s\in V}\{s\}=V$. 
Since $\B\subseteq_d \A$, then $\B\subseteq_c \A$, so
$\B\in \bold S_c\Nr_{n}\CA_{\omega}\cap \bf At$ 
because $\A\in  \Nr_{n}\CA_{\omega}$ is atomic.  
As proved in \cite{SL},  
$\B\notin {\bf  El}\Nr_{n}{\sf CA}_{n+1}(\supseteq \Nr_{n}\CA_{\omega}\cap \bf At)).$ 

\end{proof}

\begin{corollary}\label{strongly}There is a $\CA_n$ atom structure, namely $\bf At$ in Theorem \ref{can} that is weakly but not strongly representable
\end{corollary}
It is known that the claess of weakly reprsentable atom structures is elementary \cite{Venema}. In a moment we will show that the class of strongly represntable atom structure is not elementary;
reproving a result of Hirsch and Hodkinson \cite{HHbook2}.
\begin{corollary} There are infinitely many subvarieties of $\CA_n$ containing $\RCA_n$ that are not atom-canonical.
\end{corollary}
\begin{proof} It is known that for any pair of ordinals $\alpha<\beta$, $\bold S\Nr_{\alpha}\CA_{\beta}$ is a variety, and that 
for $k\geq 1$ and $2<n<\omega$, $\RCA_n\subsetneq \bold S\Nr_n\CA_{n+k+1}\subsetneq \bold S\Nr_n\CA_{n+k}\subseteq \CA_n$,cf. \cite[Chapter 15]{HHbook}.
In other words, the sequence $\langle \bold S\Nr_n\CA_{n+k+1}: k\geq 1\rangle$ is strictly decreasing.
\end{proof}
\begin{corollary} There exists a countable  atomic $\A\in {\sf Cs}_n$ that is dense in an algebra $\B$ such that $\B\notin bold S\Nr_n\CA_{n+3}$.
\end{corollary}
\begin{proof} Take $\A$ to be the algebra $\Tm\bf At$ and $\B$ its \de\ completion, namely, $\Cm\bf At$, where $\bf At$ is (weakly but not strongly representable) 
atom structure constructed in the proof of Theorem \ref{can}.
\end{proof}
Theorem \ref{can} prompts the following definition which suggests that not all algebras are representable in the 'same degree': Some algebras are `more representable' than others.
\begin{definition} Let $2<n<m\leq \omega$. We say that $\A\in \RCA_n$ {\it is reprsentable up to $m$} $\iff$ its \de\ completion is in $\bold S\Nr_n\CA_m$.
\end{definition} 
Let $\RCA_n^m$ be the class of algebras representable up to $m$. 
In \cite{t} finite algebras are constructed in $\bold S\Nr_n\CA_{n+k}\sim \bold S\Nr_n\CA_{n+k+1}$ for all $k\geq 1$. Since for a finite algebra $\A$, $\Cm\At\A=\A$, then 
for $2<n<m<l\leq \omega$, $\RCA_n^l\subsetneq \RCA_n^m$.

Using the previous algebraic result on non atom canonicity established in Theorem \ref{can}, we adress algebraically 
a version of the 
omitting types theorems  in the framework of 
the {\it clique guarded} $n$--variable fragments
of first order logic.  
\begin{lemma} Let $2<n<m<\omega$, and $\A\in \CA_n$ be an atomic algebra. Then $\A$ has a complete  $m$-square representation 
$\iff$ \pe\ has a \ws\ in $G_{\omega}^m(\At\A)$.
\end{lemma}
\begin{proof}\cite[Lemma 5.8]{mlq}.
\end{proof}
\begin{corollary} There exists $\A\in \Cs_n$ such that $\Cm\At\A$ does not have an $n+3$-square representation.
\end{corollary}
\begin{proof} This follows from the previous Lemma, together with the proof of (c) in Theorem \ref{can} by observing that \pa\ has a \ws\ in $G_{\omega}^{n+3}\CA_{n+1, n}$ 
(in finitely many rounds of course) without the need to reuse nodes. The game $\bold G^m$ is stronger than what is really needed.
\end{proof}
\begin{lemma}\label{cr} if $\A\in \CA_n$ has a complete $m$--flat representation, then $\A$ is atomic and $\Cm\At\A$ has an $m$-flat representation.
An entirely analogous result holds by replacing $m$-flat by $m$-square.
\end{lemma}
\begin{proof} Atomicity is like the classical case \cite{HH}. Now let $f:\A\to \wp(V)$ be a complete $m$--flat 
representation $\A$ with $V\subseteq {}^n\M$ where $\M$ is the base of the representation, so that 
$\M=\bigcup_{s\in V}\rng(s)$. 
For $a\in \Cm\At\A$,  let $a\downarrow=\{x\in \At\A: x\leq a\}$. Define $g:\Cm\At\A\to \wp(V)$ by
$g(a)= \bigcup_{x\in \downarrow a}f(x)$. Then $g$ is a complete $m$-flat representation of $\Cm\At\A$ with 
base $\M$.
\end{proof}
For an $L_n$ theory $T$,  $\Fm_T$, denotes the Tarski--Lindenbaum quotient ${\sf RCA}_n$ corresponding to $T$
where the quoitent modulo $T$ is defined semantically.
Given an $L_n$ theory $T$ and $m>n$, by an {\it $m$--flat model of $T$}, 
we understand an $m$-- flat representation of $\Fm_T$ when $m<\omega$,
and an ordinary representation of $\Fm_T$ if $m$ is infinite.
An {\it atomic $L_n$} theory $T$ is one for which $\Fm_T$ is atomic.
A {\it co-atom of $T$} is a formula $\phi$ such that $(\neg \phi)_T$ 
is an atom in $\Fm_T$. 
\begin{corollary}\label{s} There is a countable, atomic and complete $L_n$ theory $T$ such that the non--principal type 
consisting of co--atoms cannot be omitted in an $n+3$-square, {\it a fortiori} $n+3$-flat 
model.
\end{corollary}
\begin{proof} Let $\A\in {\sf Cs}_n$ be countable (and simple)  such that its \de\ completion does not have an $n+3$-square representation. 
This $\A$ exists by Theorem \ref{can}.  
By \cite[\S 4.3]{HMT2}, we can (and will) assume that $\A= \Fm_T$ for a countable, atomic theory $L_n$ theory $T$.  
Let $\Gamma$ be the $n$--type consisting of co--atoms of $T$. Then $\Gamma$ is a non principal type that cannot be omitted in any
$n+3$--square model, for if $\M$ is an $n+3$--square model omitting $\Gamma$, then $\M$ would be the base of a complete $n+3$-square  
representation of $\A$, giving, by Lemma \ref{cr},
representation  of $\Cm\At\A$, which is impossible.
\end{proof}

\begin{corollary}\label{embed} There exists an atomic ${\cal T}\in \sf RRA$ and an atomic 
$\A\in \RCA_n$ such that their \de\ completions do not  embed into their canonical 
extensions.
\end{corollary}
\begin{proof} 
We prove the $\CA$ case only. The $\sf RA$ case is entirely analagous. Since $\RCA_n$ is canonical \cite{HMT2} and $\A\in \RCA_n$, then its canonical extension $\A^+\in \RCA_n$. But  
$\Cm \At\A\notin \RCA_n$, so it does not embed into $\A^+$, because $\RCA_n$ is a variety, {\it a fortiori} closed under $\bold S$.
\end{proof}
Sahlqvist formulas are a certain kind of modal formula with remarkable properties. 
The Sahlqvist correspondence theorem states that every Sahlqvist formula corresponds to a first order definable class of Kripke frames.
Sahlqvist's definition characterizes a decidable set of modal formulas with first-order correspondents. Since it is undecidable, by Chagrova's theorem, 
whether an arbitrary modal formula has a first-order correspondent \cite[Theorem 3.56]{modal}, there are formulas with first-order frame conditions 
that are not Sahlqvist.  But this is not the end of the story, for it might be the case that  every modal formula with a first order correspondant is {\it equivalent} to 
a Sahlqvist one, which is 
not the case \cite[Example 3.57]{modal}. The reader is referred to \cite{modal} and \cite[2.7]{HHbook} for more on aspects of duality for $\sf BAO$s and in particular for 
Sahlqvist axiomatizability in general. 
By the dualiity theory between $\sf BAO$s and multimodal logic, Sahlqvist formulas  transform recursively to Sahlqvist equations in 
the corresponding modal algebras,  cf. \cite[Section 2.7.6]{HHbook}. A variety $\V$ of $\sf BAO$s is Sahlqvist if it can be axiomatized by Sahlqvist equations.

\begin{theorem} For any $2<n<m\leq \omega$ the variety $\bold S\Nr_n\CA_m$ is not Sahlqvist. 
Conversely, for any pair of infinite ordinals $\alpha<\beta$, the varieties $\bold S\Nr_{\alpha}\PA_{\beta}$ and $\bold S\Nr_{\alpha}\PEA_{\beta}$ are Sahlqvist, and is closed under 
\de\ completions.
\end{theorem}
\begin{proof} Let $\alpha<\beta$ be infinite ordinals. Then $\bold S\Nr_{\alpha}\PA_{\beta}=\Nr_{\alpha}\PA_{\beta}=\PA_{\alpha}$, cf. the remark before \cite[Theorem 5.4.17]{HMT2}. 
The last is axiomatized by positive equations \cite[Definition 5.4.1]{HMT2} which are Sahlqvist. Applying \cite{Venema} we are done. 
The $\PEA$ case is entirely analogous using the axiomatization in the aforementioned definition.
\end{proof}

Let $2<n<\omega$. We approach the modal version of $L_n$ without equality, namely, ${\bf S5}^n$. The corresponding class of modal algebras
is the variety ${\sf RDf}_n$ of {\it diagonal free $\RCA_n$s} \cite{HMT2}. Let $\Rd_{df}$ denote 'diagonal free reduct'. 
 \begin{lemma}\label{dfb} Let $2<n<\omega$. If $\A\in \CA_n$ is such that $\Rd_{df}\A\in {\sf RDf}_n$,
and $\A$ is generated by $\{x\in \A: \Delta x\neq n\}$ (with other $\CA$ operations) using infinite intersections, then
$\A\in \RCA_n$.
\end{lemma}
\begin{proof}  Easily follows from \cite[Lemma 5.1.50, Theorem 5.1.51]{HMT2}.  Assume that $\A\in \CA_n$, $\Rd_{df}\A$ is a set algebra (of dimension $n$) 
with base $U$,  and $R\subseteq U\times U$ are as in the hypothesis of \cite[Theorem 5.1.49]{HMT2}. 
Let $E=\{x\in A:  (\forall x, y\in {}^nU)(\forall i <n)(x_iR y_i\implies (x\in X\iff y\in X))\}$.
Then $\{x\in \A: \Delta x\neq n\}\subseteq E$ and $E\in \CA_n$ 
is closed under infinite intersections. The required follows.
\end{proof}
\begin{theorem} \label{df} For $2<n<\omega$, ${\sf RDf}_n$ is not atom--canonical, hence not Sahlqvist. 
\end{theorem}
\begin{proof} It is enough to show that $\Cm\At\A$, where $\A$ is constructed in Theorem \ref{can}  
is generated by elements whose dimension sets have cardinality $<n$ using infinite unions, for in this case $\Rd_{df}A$ will be atomic, 
countable and representable, but  
having no complete representation.  
Indeed,  by Lemma \ref{dfb} and Theorem \ref{can}, $\Rd_{df}\Cm\At\A=\Cm\At\Rd_{df}\A$ will not be representable.
We show that for any rainbow atom $[a]$, $a:n\to \Gamma$, $\Gamma$ a coloured graph, that
$[a]=\prod_{i<n} {\sf c}_i[a]$. 
Clearly $\leq $ holds. Assume that $b:n\to \Delta$, $\Delta$ a coloured graph, and $[a]\neq [b]$. We show that $[b]\notin \prod_{i<n} {\sf c}_i[a]$  by which we 
will be done. Because $a$ is not equivalent to $b$, we have one of two possibilities;
either $(\exists i, j<n) (\Delta(b(i), b(j)\neq \Gamma(a(i), a(j))$ or
$(\exists i_1, \ldots, i_{n-1}<n)(\Delta(b_{i_1},\ldots, b_{i_{n-1}})\neq \Gamma(a_{i_1},\ldots, a_{i_{n-1}}))$. 
Assume the first possibility: Choose  $k\notin \{i, j\}$. This is possible because $n>2$.  Assume for contradiction that $[b]\in {\sf c}_k[a]$. 
Then $(\forall i, j\in n\setminus \{k\})(\Delta(b(i), b(j))=\Gamma(a(i) a(j)))$. By assumption  and the choice of
$k$,  $(\exists i, j\in n\setminus k)(\Delta(b(i), b(j))\neq \Gamma(a(i), a(j)))$, contradiction.
For the second possibility, one chooses $k\notin \{i_1, \ldots i_{n-1}\}$ and proceeds like the first case deriving an analogous contradiction.
\end{proof} 
${\bf K}^n$ is the logic of $n$-ary product frames, of the form $(W_i, R_i)_{i<n}$ where for each $i<n$, $R_i$ is any any relation on $W_i$.
On the other hand, 
${\bf S5}^n$ can be regarded as the logic of $n$--ary product frames of the form 
$(W_i, R_i)_{i<n}$ such that for each $i<n$, $R_i$ is an equivalence relation.

It is known that modal languages can come to grips with 
a strong fragment of second order logic. Modal 
formulas translate to second order formulas, {\it their correspondants} 
on frames.  Some of these formulas can be {\it genuinely second order}; 
they are not equivalent to first order formulas. An example is the {\it McKinsey formula}: 
$\Box \Diamond p\to \Diamond \Box p$. This can be proved by showing that its correspondant violates 
the downward L\"owenheim- Skolem Theorem.
The next proposition  bears on the last two issues. 
For a class $\bold L$ of frames, let $\L(\bold L)$ be the class of modal formulas valid in $\bold L$. 
It is difficult to find explicity (necessarily) infinite axiomatizations for ${\bf S5}^n$ as well as shown in the next Theorem. 

But first we recall some basic notions about graphs. 
A \textit{(directed) graph} is a set $G$ (of \textit{nodes} or
\textit{vertices}) endowed with a binary relation $E$, the edge
relation. A pair $(x, y)$ of elements of $G$ is said to be an
\textit{edge} if $xEy$ holds. A directed graph is said to be
{\it complete} if $(x, y)$ is an edge for all nodes $x, y$. A graph is
said to be {\it undirected} if $E$ is symmetric and irreflexive. An
undirected graph is {\it complete} if $(x, y)$ is an edge for all {\it distinct}
nodes $x, y$.  Finite ordinals were viewed as complete irreflexive graphs the obvious way, cf. Theorem \ref{can}.

A \textit{clique} in an undirected graph with set of
nodes $G$ is a set $C \subseteq G$ such that each pair of distinct
nodes of $C$ is an edge.

\begin{definition}\label{graph} Let $\G=(G, E)$ be an undirected graph ($G$ is the set of vertices and $E$ is an ireflexive symmetric 
binary relation on $E$), and
$C$ be a non-empty set of `colours'.
\begin{enumerate}

\item A subset $X\subseteq G$ is said to be an {\it independent set} if $(x,y)\in E$
for all $x,y\in X$.

\item A function $f:G\to C$ is called a {\it $C$ colouring} of  $\G$ if
$(v,w)\in E$ implies that $f(v)\neq f(w)$. 

\item The {\it chromatic number} of $\G$, denoted by $\chi(\G),$ is
the size of the smallest finite set $C$ such that there exists
a $C$ colouring of $\G$, if
such a $C$ exists, otherwise $\chi(\G)=\infty.$

\item A {\it cycle} in $\G$ is a finite sequence $\mu=(v_0,\ldots v_{k-1})$ (some $k\in \omega$) of
distinct nodes, such that $(v_0, v_1), \ldots (v_{k-2}, v_{k-1}), (v_{k-1}, v_0) \in E$. The {\it length} 
of such a cycle is $k.$

\item The {\it girth} of $\G$, denoted by $g(\G)$, is the
length of the shortest cycle
in $\G$ if $\G$ contains cycles, 
and $g(\G)=\infty$ othewise.
\end{enumerate}
\end{definition}

\begin{theorem}\label{complexity} Let $2<n<\omega$. 
There is no axiomatization of ${\bf S5}^n$ with formulas having 
first order correspondence. 
For any canonical logic $\L$ between $\bold K^n$ and ${\bf S5}^n$, 
it is undecidable to tell whether a finite frame is a frame for $\L$,  $\L$ cannot be finitely axiomatized in $k$th order logic (for any finite $k$),
and $\L$ cannot be axiomatized by canonical formulas, {\it a fortiori} Sahlqvist formulas.
\end{theorem}
\begin{proof}  Let $\bold L$ be the class of square frames for ${\bf S5}^n$.
Then $\L(\bold L)={\bf S5}^n$ \cite[p.192]{k}. But the class of frames $\F$ valid in $\L(\bold L)$ coincides 
with the class of  {\it strongly representable ${\sf Df}_n$ atom structures} which  
is {\it not elementary} as proved in \cite{b}. This gives the first required result for ${\bf S5}^n$. With lemma \ref{dfb} at our disposal, 
a slightly different proof can be easily distilled from the construction adressing $\CA$s in \cite{HHbook2} or \cite{k2}. 
We adopt the construction in the 
former reference, using the Monk--like $\CA_n$s  
${\mathfrak M}(\Gamma)$, $\Gamma$ a graph, as defined in
\cite[Top of p.78]{HHbook2}. 
For a graph $\G$, let $\chi(\G)$ denote it chromatic number. 
Then it is proved in {\it op.cit} that 
for any graph $\Gamma$, ${\mathfrak M}(\Gamma)\in \RCA_n$ 
$\iff$ $\chi(\Gamma)=\infty$.
By lemma \ref{dfb},
$\Rd_{df}{\mathfrak M}(\Gamma)\in {\sf RDf}_n\iff \chi(\Gamma)=\infty$,
because $\mathfrak{M}(\Gamma)$ 
is generated by the set $\{x\in {\mathfrak M}(\Gamma): \Delta x\neq n\}$ using infinite unions.
Now we adopt the argument in \cite{HHbook2}. Using Erdos' probabalistic graphs \cite{Erdos}, 
for each finite
$\kappa$, there is a finite graph $G_{\kappa}$ with
$\chi(G_{\kappa})>\kappa$ and with no cycles of length $<\kappa$. 
Let $\Gamma_{\kappa}$ be the disjoint union of the $G_{l}$ for
$l>\kappa$.  Then $\chi(\Gamma_{\kappa})=\infty$, and so
$\Rd_{df}\mathfrak{M}(\Gamma_{\kappa})$ is representable.
Now let $\Gamma$ be a non-principal ultraproduct
$\Pi_{D}\Gamma_{\kappa}$ for the $\Gamma_{\kappa}$s. For $\kappa<\omega$, let $\sigma_{\kappa}$ be a
first-order sentence of the signature of the graphs stating that
there are no cycles of length less than $\kappa$. Then
$\Gamma_{l}\models\sigma_{\kappa}$ for all $l\geq\kappa$. By
Lo\'{s}'s Theorem, $\Gamma\models\sigma_{\kappa}$ for all
$\kappa$. So $\Gamma$ has no cycles, and hence by $\chi(\Gamma)\leq 2$.
Thus $\mathfrak{Rd}_{df}\mathfrak{M}(\Gamma)$
is not representable.  
(Observe that the 
the term algebra $\Tm\At(\mathfrak{M}(\Gamma))$
is representable (as a $\CA_n$), 
because the class of weakly representable atom structures is elementary \cite[Theorem 2.84]{HHbook}.)
Since Sahlqvist formulas have first order correspondants, then ${\bf S5}^n$ is not Sahlqvist.  
In \cite{k2}, it is proved that it is undecidable to tell whether a finite frame
is a frame for $\L$,  and this gives
the non--finite axiomatizability result required as indicated in {\it op.cit}, 
and obviously implies undecidability.
The rest follows by transferring the required results holding for ${\bf S5}^n$ \cite{b, k2}
to $\L$ since ${\bf S5}^n$ is finitely axiomatizable 
over $\L$, and any axiomatization of ${\sf RDf}_n$ must contain infinitely many 
non-canonical equations. 
\end{proof}
\section{Complete representations and $\sf OTT$}

Suppose that  $\A=\Nrr_n\D$ for some {\it atomic} $\D\in \CA_{\omega}$, does this imply that $\A$ is completely representable? 
We show that this might not be the case if $\D$ is atomless. 
\begin{theorem}\label{bsl} Let $\kappa$ be an infinite cardinal. Then there exists an atomless $\C\in \CA_{\omega}$ such that  for all 
$2<n<\omega$, $\Nrr_n\C$ is atomic, with $|\At(\mathfrak{Nr}_n\C)|=2^{\kappa}$, $\mathfrak{Nr}_n\C\in {\sf LCA}_n$, 
but $\mathfrak{Nr}_n\C$ is not completely representable. 
\end{theorem}
\begin{proof}
We use the following uncountable version of Ramsey's theorem due to
Erdos and Rado:
If $r\geq 2$ is finite, $k$  an infinite cardinal, then
$exp_r(k)^+\to (k^+)_k^{r+1}$
where $exp_0(k)=k$ and inductively $exp_{r+1}(k)=2^{exp_r(k)}$.
The above partition symbol describes the following statement. If $f$ is a coloring of the $r+1$
element subsets of a set of cardinality $exp_r(k)^+$
in $k$ many colors, then there is a homogeneous set of cardinality $k^+$
(a set, all whose $r+1$ element subsets get the same $f$-value).
Let $\kappa$ be the given cardinal. We use a variation on the construction in \cite{bsl} which is a simplified more basic version of a rainbow construction where only 
the two predominent  colours, namely, the reds and blues are available. 
The algebra $\C$ will be constructed from a relation algebra possesing an $\omega$-dimensional cylindric basis.
To define the relation algebra we specify its atoms and the forbidden triples of atoms. The atoms are $\Id, \; \g_0^i:i<2^{\kappa}$ and $\r_j:1\leq j<
\kappa$, all symmetric.  The forbidden triples of atoms are all
permutations of $({\sf Id}, x, y)$ for $x \neq y$, \/$(\r_j, \r_j, \r_j)$ for
$1\leq j<\kappa$ and $(\g_0^i, \g_0^{i'}, \g_0^{i^*})$ for $i, i',
i^*<2^{\kappa}.$ 
Write $\g_0$ for $\set{\g_0^i:i<2^{\kappa}}$ and $\r_+$ for
$\set{\r_j:1\leq j<\kappa}$. Call this atom
structure $\alpha$.  
Consider the term algebra $\A$ defined to be the subalgebra of the complex algebra of this atom structure generated by the atoms.
We claim that $\A$, as a relation algebra,  has no complete representation, hence any algebra sharing this 
atom structure is not completely representable, too. Indeed, it is easy to show that if $\A$ and $\B$ 
are atomic relation algebras sharing the same atom structure, so that $\At\A=\At\B$, then $\A$ is completely representable $\iff$ $\B$ is completely representable.

Assume for contradiction that $\A$ has a complete representation $\Mo$.  Let $x, y$ be points in the
representation with $\Mo \models \r_1(x, y)$.  For each $i< 2^{\kappa}$, there is a
point $z_i \in \Mo$ such that $\Mo \models \g_0^i(x, z_i) \wedge \r_1(z_i, y)$.
Let $Z = \set{z_i:i<2^{\kappa}}$.  Within $Z$, each edge is labelled by one of the $\kappa$ atoms in
$\r_+$.  The Erdos-Rado theorem forces the existence of three points
$z^1, z^2, z^3 \in Z$ such that $\Mo \models \r_j(z^1, z^2) \wedge \r_j(z^2, z^3)
\wedge \r_j(z^3, z_1)$, for some single $j<\kappa$.  This contradicts the
definition of composition in $\A$ (since we avoided monochromatic triangles).
Let $S$ be the set of all atomic $\A$-networks $N$ with nodes
$\omega$ such that $\{\r_i: 1\leq i<\kappa: \r_i \text{ is the label
of an edge in $N$}\}$ is finite.
Then it is straightforward to show $S$ is an amalgamation class, that is for all $M, N
\in S$ if $M \equiv_{ij} N$ then there is $L \in S$ with
$M \equiv_i L \equiv_j N$, witness \cite[Definition 12.8]{HHbook} for notation.
Now let $X$ be the set of finite $\A$-networks $N$ with nodes
$\subseteq\kappa$ such that:

\begin{enumerate}
\item each edge of $N$ is either (a) an atom of
$\A$ or (b) a cofinite subset of $\r_+=\set{\r_j:1\leq j<\kappa}$ or (c)
a cofinite subset of $\g_0=\set{\g_0^i:i<2^{\kappa}}$ and

\item  $N$ is `triangle-closed', i.e. for all $l, m, n \in \nodes(N)$ we
have $N(l, n) \leq N(l,m);N(m,n)$.  That means if an edge $(l,m)$ is
labelled by $\sf Id$ then $N(l,n)= N(m,n)$ and if $N(l,m), N(m,n) \leq
\g_0$ then $N(l,n)\cdot \g_0 = 0$ and if $N(l,m)=N(m,n) =
\r_j$ (some $1\leq j<\omega$) then $N(l,n)\cdot \r_j = 0$.
\end{enumerate}
For $N\in X$ let $\widehat{N}\in\Ca(S)$ be defined by
$$\set{L\in S: L(m,n)\leq
N(m,n) \mbox{ for } m,n\in \nodes(N)}.$$
For $i\in \omega$, let $N\restr{-i}$ be the subgraph of $N$ obtained by deleting the node $i$.
Then if $N\in X, \; i<\omega$ then $\widehat{\cyl i N} =
\widehat{N\restr{-i}}$.
The inclusion $\widehat{\cyl i N} \subseteq (\widehat{N\restr{-i})}$ is clear.
Conversely, let $L \in \widehat{(N\restr{-i})}$.  We seek $M \equiv_i L$ with
$M\in \widehat{N}$.  This will prove that $L \in \widehat{\cyl i N}$, as required.
Since $L\in S$ the set $T = \set{\r_i \notin L}$ is infinite.  Let $T$
be the disjoint union of two infinite sets $Y \cup Y'$, say.  To
define the $\omega$-network $M$ we must define the labels of all edges
involving the node $i$ (other labels are given by $M\equiv_i L$).  We
define these labels by enumerating the edges and labeling them one at
a time.  So let $j \neq i < \kappa$.  Suppose $j\in \nodes(N)$.  We
must choose $M(i,j) \leq N(i,j)$.  If $N(i,j)$ is an atom then of
course $M(i,j)=N(i,j)$.  Since $N$ is finite, this defines only
finitely many labels of $M$.  If $N(i,j)$ is a cofinite subset of
$\g_0$ then we let $M(i,j)$ be an arbitrary atom in $N(i,j)$.  And if
$N(i,j)$ is a cofinite subset of $\r_+$ then let $M(i,j)$ be an element
of $N(i,j)\cap Y$ which has not been used as the label of any edge of
$M$ which has already been chosen (possible, since at each stage only
finitely many have been chosen so far).  If $j\notin \nodes(N)$ then we
can let $M(i,j)= \r_k \in Y$ some $1\leq k < \kappa$ such that no edge of $M$
has already been labelled by $\r_k$.  It is not hard to check that each
triangle of $M$ is consistent (we have avoided all monochromatic
triangles) and clearly $M\in \widehat{N}$ and $M\equiv_i L$.  The labeling avoided all
but finitely many elements of $Y'$, so $M\in S$. So
$\widehat{(N\restr{-i})} \subseteq \widehat{\cyl i N}$.

Now let $\widehat{X} = \set{\widehat{N}:N\in X} \subseteq \Ca(S)$.
Then we claim that the subalgebra of $\Ca(S)$ generated by $\widehat{X}$ is simply obtained from
$\widehat{X}$ by closing under finite unions.
Clearly all these finite unions are generated by $\widehat{X}$.  We must show
that the set of finite unions of $\widehat{X}$ is closed under all cylindric
operations.  Closure under unions is given.  For $\widehat{N}\in X$ we have
$-\widehat{N} = \bigcup_{m,n\in \nodes(N)}\widehat{N_{mn}}$ where $N_{mn}$ is a network
with nodes $\set{m,n}$ and labeling $N_{mn}(m,n) = -N(m,n)$. $N_{mn}$
may not belong to $X$ but it is equivalent to a union of at most finitely many
members of $\widehat{X}$.  The diagonal $\diag ij \in\Ca(S)$ is equal to $\widehat{N}$
where $N$ is a network with nodes $\set{i,j}$ and labeling
$N(i,j)=\sf Id$.  Closure under cylindrification is given.
Let $\C$ be the subalgebra of $\Ca(S)$ generated by $\widehat{X}$.
Then $\A = \mathfrak{Ra}(\C)$.
To see why, each element of $\A$ is a union of a finite number of atoms,
possibly a co--finite subset of $\g_0$ and possibly a co--finite subset
of $\r_+$.  Clearly $\A\subseteq\mathfrak{Ra}(\C)$.  Conversely, each element
$z \in \mathfrak{Ra}(\C)$ is a finite union $\bigcup_{N\in F}\widehat{N}$, for some
finite subset $F$ of $X$, satisfying $\cyl i z = z$, for $i > 1$. Let $i_0,
\ldots, i_k$ be an enumeration of all the nodes, other than $0$ and
$1$, that occur as nodes of networks in $F$.  Then, $\cyl
{i_0} \ldots
\cyl {i_k}z = \bigcup_{N\in F} \cyl {i_0} \ldots
\cyl {i_k}\widehat{N} = \bigcup_{N\in F} \widehat{(N\restr{\set{0,1}})} \in \A$.  So $\mathfrak{Ra}(\C)
\subseteq \A$.
$\A$ is relation algebra reduct of $\C\in\CA_\omega$ but has no complete representation.
Let $n>2$. Let $\B=\Nrr_n \C$. Then
$\B\in {\sf Nr}_n\CA_{\omega}$, is atomic, but has no complete representation for plainly a complete representation of $\B$ induces one of $\A$. 
In fact, because $\B$  is generated by its two dimensional elements,
and its dimension is at least three, its
$\Df$ reduct is not completely representable.

It remains to show that the $\omega$--dilation $\C$ is atomless. 
For any $N\in X$, we can add an extra node 
extending
$N$ to $M$ such that $\emptyset\subsetneq M'\subsetneq N'$, so that $N'$ cannot be an atom in $\C$.

\end{proof}

However, in Theorem \ref{complete4} to be proved next,  we give a positive answer if we replace $\CA_{\omega}$ by $\PA_{\omega}$ and $\PEA_{\omega}$, i.e
when the $\omega$--dilations are {\it atomic polyadic algebras} with and without equality.
This allows us to give in the next theorem a huge class of completely representable
$\PA_n$ and $\PEA_n$s for any finite dimension $n$, possibly having uncountably many atoms.
We need to recall from \cite[definition~5.4.16]{HMT2},  the notion of neat reducts of polyadic algebras, because we shall be dealing with infinite dimensional such
algebras. In this case neat reducts are defined differently. We do it only for $\PA$s.  The $\PEA$ case is defined analogously counting in diagonal elements the obvious way.
 \begin{definition} Let $J\subseteq \beta$ and
$\A=\langle A,+,\cdot ,-, 0, 1,{\sf c}_{(\Gamma)}, {\sf s}_{\tau}\rangle_{\Gamma\subseteq \beta ,\tau\in {}^{\beta}\beta}$
be a $\PA_{\beta}$.
Let $Nr_J\B=\{a\in A: {\sf c}_{(\beta\sim J)}a=a\}$. Then
$${\bf Nr}_J\B=\langle Nr_{J}\B, +, \cdot, -, {\sf c}_{(\Gamma)}, {\sf s}'_{\tau}\rangle_{\Gamma\subseteq J, \tau\in {}^{\alpha}\alpha}$$
where ${\sf s}'_{\tau}={\sf s}_{\bar{\tau}}.$ Here $\bar{\tau}=\tau\cup Id_{\beta\sim \alpha}$.
The structure ${\bf Nr}_J\B$ is an algebra, called the {\it $J$--compression} of $\B$.
When $J=\alpha$, $\alpha$ an ordinal, then ${\bf Nr}_{\alpha}\B\in \PA_{\alpha}$ and it is
called the {\it strong neat $\alpha$ reduct} of $\B$, and its elements are called
$\alpha$--dimensional. For $\beta\geq \omega$, $\alpha<\beta\cap \omega$, $\bold K\in \{\PA, \PEA\}$ and $\bold L\subseteq \bold K$,
${\bf  Nr}_{\alpha}\bold L_{\beta}$ denotes the class $\{{\bf Nr}_{\alpha}\A: \A\in \bold L_{\beta}\}$
\end{definition}

If $n<\omega$, $\alpha\geq \omega$,
and $\B\in \PA_{\alpha}$, then ${\bf Nr}_n\B$ is contained in $\mathfrak{Nr}_n\Rd_{qa}\B$ Here $\Rd_{qa}$ denotes the quasi--polyadic reduct of $\B$, obtained by discarding infinitary
substitutions and  the definition of neat reducts is like the $\CA$ not involving infinitary cylindrfiers.
Indeed, if ${\sf c}_{(\beta\sim n)}x=x$, then for any $i\in \alpha\sim n$, in $\B$ we have,
${\sf c}_ix\leq {\sf c}_{(\alpha\sim n)}x=x\leq {\sf c}_ix$,
hence ${\sf c}_ix=x$.   However, the converse might not be true. If ${\sf c}_ix=x$ for all $i\in \alpha\sim n$, this does not imply that
${\sf c}_{(\alpha\sim n)}x=x;$ it can happen that ${\sf c}_{(\alpha\sim n)}x>x={\sf c}_ix$ (for all $i\in \alpha\sim n$).
Same observation carries over to to $\PEA$s.
For an ordinal $\alpha$, let $\PA_{\alpha}^{\sf atc}$ denote the class of completely additive atomic ${\sf PA}_{\alpha}$s.
\begin{theorem} \label{complete4} Let $n$ be a finite ordinal.
Then   $\bold S_c{\bf  Nr}_n\PA_{\omega}^{\sf atc}\subseteq \bold S_c{\Nr}_n(\Rd_{qa}\PA_{\omega}^{\sf atc})\subseteq {\sf CRPA}_n$.
The same result holds for $\PEA$s without the condition of complete additivity imposed on the dilation.
\end{theorem}
\begin{proof}
We have already delat with the first inclusion. Let $\A\in \PA_n$ and $\A\subseteq_c {\bf Nr}_n\D$, where $\D\in \PA_{\omega}$ is atomic and completely additive.
From \cite{au}, we have $\D$ is completely representable.
We show that ${\bf Nr}_n\D\subseteq_c \D$ and $\mathfrak{Nr}_n\Rd_{qa}\D\subseteq_c \D.$ 
Assume that $S\subseteq {\bf Nr}_n\D$ and $\sum ^{{\bf Nr}_n\D}S=1$, and for contradiction, that there exists $d\in \D$ such that
$s\leq d< 1$ for all $s\in S$. Let  $J=\Delta d\sim n$ and take  $t=-{\sf c}_{(J)}(-d)\in {\bf Nr}_n\D\subseteq \mathfrak{Nr}_n\Rd_{qa}\D$.
Then
\begin{align*}
{\sf c}_{(\alpha\sim n)}t&={\sf c}_{(\alpha\sim n)}(-{\sf c}_{(J)} (-d))\\
&=  {\sf c}_{(\alpha\sim n)}-{\sf c}_{(J)} (-d)\\
&=  {\sf c}_{(\alpha\sim n)} -{\sf c}_{(\alpha\sim n)}{\sf c}_{(J)}( -d)\\
&= -{\sf c}_{(\alpha\sim n)}{\sf c}_{(J)}( -d)\\
&=-{\sf c}_{(J)}( -d)\\
&=t
\end{align*}
We have proved that $t\in {\bf Nr}_n\D\subseteq \mathfrak{Nr}_n\Rd_{qa}\D$.
We now show that $s\leq t<1$ for all $s\in S$, which contradicts $\sum^{{\bf  Nr}_n\D}S\geq \sum{}^{\mathfrak{Nr}_n\Rd_{qa}\D}S=1$.
If $s\in S$, we show that $s\leq t$.
By $s\leq d$, we have  $s\cdot -d=0$.
Hence $0={\sf c}_{(J)}(s\cdot -d)=s\cdot {\sf c}_{(J)}(-d)$, so
$s\leq -{\sf c}_{(J)}(-d)$, hence  $s\leq t$ as required. Assume for contradiction that $t=1$, then
$1=-{\sf c}_{(J)}(-d)$ and so $ {\sf c}_{(J)}(-d)=0$.
But $-d\leq {\sf c}_{(J)}(-d)$,  so $1\cdot -d\leq  {\sf c}_{(J)}(-d)=0.$
Hence $1\cdot -d =0$, and this contradicts that $d<1$. We have proved that $\sum^{\D}S=1$, so ${\bf Nr}_n\D\subseteq_c \mathfrak{Nr}_n\Rd_{qa}D\subseteq_c \D$ 
and thus ${\bf Nr}_n\D\subseteq_c \D.$
Now we only work with the operator ${\bf Nr}_n$. Let $f:\D\to \C$ be a complete representation of $\D$, where $\C$ has top element $1^{\C}$.  As mentioned earlier such an $f$ exists \cite{au}.
Here $\C$ is a generalized set algebra; the top element  $1^{\C}$ is of the form $\bigcup_{i\in I}{}^{\omega}U_i$,
where $I$ is a non--empty set,  for $i\neq j\in I$, $U_i\cap U_j=\emptyset$, and $\C$ is closed under the polyadic concrete operations including of course all
infinitary substitutions and cylindrifiers. Such operations are defined for $\tau\in {}^{\omega}\omega$ and $X\subseteq 1^{\C}$ as follows:
${\sf s}_{\tau}X=\{s\in 1^{\C}: s\circ {\tau}\in X\},$
${\sf c}_{(\Gamma)}X=\{s\in 1^{\C}: (\exists t\in X): t(j)=s(j), \forall j\notin \Gamma\}.$
We can assume without loss that $f$ is an isomorphism.
For brevity, let $\A={\sf Nr}_n\D (\in \PA_n)$. For $a\in \A$,   let $g(a)=\{s\in \bigcup_{i\in I}{}^{n}U_i: s\cup Id\in f(a)\}$. Then it is easy to show that $g$ is injective (because $f$ is).

We show that it preserves cylindrifiers. For a function $t$ and $i\in \dom t$, $t_i^u$ denotes the function that is like $t$ except that
$t(i)=u$. Let $x\in A$ and $i<n$.
Then $s\in {\sf c}_ig(x)\iff {\sf s}^i_u\in g(x)\iff {\sf s}_u^i\cup Id\in f(x)\iff (s\cup Id)^i_u\in f(x)\iff s\cup Id\in {\sf c}_if(x)=f({\sf c}_ix)\iff s\in g({\sf c}_ix).$
Like with cylindrifiers it is  straightforward to check preservation of the other polyadic operations.
Then $g$ is an injective homomorphism from $\A$
into a set algebra with top element $\bigcup_{i\in I}{}^{n}U_i$, say.
We show that $g$ is an atomic, hence complete representation. If not, then there exists $s\in {}^nU$, such that
$s\notin g(x)$ for all $x\in \At\A$.
But $f$ is a complete, hence an atomic representation, too, so there exists an atom $\beta$ of $\D$, such that $s\cup Id\in f(\beta)$.
Let $F=\{x\in \A: x\geq \beta\}\subseteq \A\subseteq \D$. Then $F$ is an ultrafilter in $\A$; it is clearly a filter and if $a\in \A$ and $a\ngeq \beta$,
then $-a\cdot \beta\neq 0$, so $\beta\leq -a$, because
$\beta$ is an atom, from which it follows that $-a\in F$. Also $F$ is non--principal, for if it is, then there would be an atom $\alpha\in \A$
 such that $\alpha\geq \beta$, so $s\in g(\alpha)$ which is impossible.
Therefore $\prod^{\A}F=0$ in $\A$, but clearly,
$\prod^{\D}F=\beta$ in $\D$.   This contradicts that $\A\subseteq_c \D$, and we are done.

Now we deal with $\PEA$s. Here representability for infinite dimensional algebras is more tricky because, unlike $\PA$s, not every $\PEA$ is representable.
The idea here is that when we truncate the dimension to be finite; to $n$, say,  
then the resulting algebra $\PEA_n$ (by the neat embedding theorem for $\QEA$s) becomes representable, 
since any $\PEA_{\omega}$ has a $\sf QEA_{\omega}$ reduct obtained by discarding infinitary substituitions.
Let $\A\subseteq_c{\bf Nr}_n\D$ where $\D\in \PEA_{\omega}$ is atomic.
We want to completely represent $\A$.  Let $a\in \A$ be non--zero. We will find  a complete
representation $f$ of $\A$, such that $f(a)\neq 0$.
We have $\A$ is atomic.  For brevity, let $X=\At\A$.
Let $\C=\Rd_{pa}\D\in \PA_{\omega}$, so $\C$ is obtained from $\D$ by discarding diagonal elements.
We use the argument in  \cite{au}, which freely uses the terminology in \cite{DM}, and which  we continue to use.
Let $\mathfrak{m}$ be the local degree of $\C$, $\mathfrak{c}$ its effective cardinality and
$\mathfrak{n}$ be any cardinal such that $\mathfrak{n}\geq \mathfrak{c}$
and $\sum_{s<\mathfrak{m}}\mathfrak{n}^s=\mathfrak{n}$.
Then there exists an atomic  $\B\in \PA_{\mathfrak{n}}$, such that
that $\C=\bf Nr_{\omega}\B$, cf. \cite[Theorem 3.10]{DM},
and the local degree of $\B$ is the same as that of $\C$.
Since $\B$ is a dilation of $\C$, which is a reduct of a $\PEA_{\omega}$, then one can define for all $i<j<\mathfrak{n}$, the diagonal element ${\sf d}_{ij}$ in $\B$, using
the diagonal elements in $\D$, as in
\cite[Theorem 5.4.17]{HMT2} satisfying the (abstract) axioms for $\sf PEA_{\mathfrak{n}}$.
Call the expanded structure $\B^*(\in \PEA_{\mathfrak{n}})$.
For a while we concentrate only on
$\B$;  we forget about diagonal elements. For $\tau\in {}^{\omega}\mathfrak{n}$, we write $\tau^+$ for $\tau\cup Id_{\mathfrak{n}\sim \omega}(\in {}^\mathfrak{n}\mathfrak{n}$).
Let $F$ be a Boolean principal ultrafilter in $\B$, that contains $a$ and preserves the following joins evaluated in $\B$,
where $p\in \B$, $\Gamma\subseteq \mathfrak{n}$ and
$\tau\in {}^{\omega}\mathfrak{n}$:
$${\sf c}_{(\Gamma)}p=\sum^{\B}\{{\sf s}_{\bar{\tau}}p: \tau\in {}^{\omega}\mathfrak{n},\ \  \tau\upharpoonright \alpha\sim\Gamma=Id\},$$ and (*):
$${\sf s}_{{\tau^+}}^{\B}X=1.$$
The first join exists \cite{au, DM}, and the second exists, because $\sum ^{\A}X=\sum ^{\D}X=\sum ^{\B}X=1$ and $\tau^+$ is completely additive, since
$\B^*\in \PEA_{\mathfrak{n}}$. The last equality of suprema follows from the fact that $\D=\sf Nr_{\omega}\B\subseteq_c \B$; this can
be proved exactly as above.
Such an  ultrafilter exists \cite{au}. In fact, any ultrafilter
generated by an atom below the non--zero $a$ will be as required. The underlying idea,  used also in \cite{au} 
is that the above joins give rise to nowhere dense sets in  the Stone topology $S$
of $\B$, any $F\in S$ preserves these joins $\iff\ F$ lies outside these sets, any principal ultrafilter lies outside nowhere dense sets, and finally
the principal ultrafilters are dense in $S$, since $\B$ is atomic.
For $i, j\in \mathfrak{n}$, set $iEj\iff {\sf d}_{ij}^{\B}\in F$.
Then by the equational properties of diagonal elements and properties of filters, it is easy to show that $E$ is an equivalence relation on $\mathfrak{n}$.
Define $f: \A\to \wp({}^n\mathfrak{n})$, via $x\mapsto \{\bar{t}\in {}^n(\mathfrak{n}/E): {\sf s}_{t\cup Id}^{\B}x\in F\},$
where $\bar{t}(i/E)=t(i)$ We show that $f$ is well--defined. Clearly  $f(a)\neq 0$. Showing that $f$ is an atomic
homomorphism (preserving diagonal elements, as well) is not hard \cite[Theorem 3.2.4]{Sayed}.
Let $V={}^{\mathfrak{n}}\mathfrak{n}^{(Id)}$.
To show that $f$ is well-defined, it suffices to show  that for all $\sigma, \tau\in V$,
if $(\tau(i), \sigma(i))\in E$ for all $i\in \mathfrak{n}$ and $a\in A$, then
${\sf s}_{\tau}^{\B}a\in F\iff {\sf s}_{\sigma}^{\B}a\in F.$  This can be proved by induction on
$|\{i\in \mathfrak{n}: \tau(i)\neq \sigma(i)\}| (<\omega)$.
If $J=\{i\in \mathfrak{n}: \tau(i)\neq \sigma(i)\}$ is empty, the result is obvious.
Otherwise assume that $k\in J$.
We introduce a helpful piece of notation.
For $\eta\in V (={}^{\mathfrak{n}}\mathfrak{n}^{(Id)})$, let
$\eta(k\mapsto l)$ stand for the $\eta'$ that is the same as $\eta$ except
that $\eta'(k)=l.$
Now take any
$\lambda\in  \{\eta\in \mathfrak{n}: (\sigma){^{-1}}\{\eta\}= (\tau){^{-1}}\{\eta\}=
\{\eta\}\}\smallsetminus \Delta a.$
Recall that $\Delta a=\{i\in \mathfrak{n}: {\sf c}_ix\neq x\}$ and that $\mathfrak{n}\setminus \Delta a$
is infinite because $\Delta a\subseteq n$, so such a $\lambda$ exists. Now
we freely use properties of substitutions for cylindric algebras.
We have by \cite[1.11.11(i)(iv)]{HMT2}
(a) ${\sf s}_{\sigma}x={\sf s}_{\sigma k}^{\lambda}{\sf s}_{\sigma(k\mapsto \lambda)}x,$
and (b)
${\sf s}_{\tau k}^{\lambda}({\sf d}_{\lambda, \sigma k}\cdot  {\sf s}_{\sigma} x)
={\sf d}_{\tau k, \sigma k} {\sf s}_{\sigma} x,$
and (c)
${\sf s}_{\tau k}^{\lambda}({\sf d}_{\lambda, \sigma k}\cdot {\sf s}_{\sigma(k\mapsto \lambda)}x)
= {\sf d}_{\tau k,  \sigma k}\cdot {\sf s}_{\sigma(k\mapsto \tau k)}x,$
and finally (d)
${\sf d}_{\lambda, \sigma k}\cdot {\sf s}_{\sigma k}^{\lambda}{\sf s}_{{\sigma}(k\mapsto \lambda)}x=
{\sf d}_{\lambda, \sigma k}\cdot {\sf s}_{{\sigma}(k\mapsto \lambda)}x.$
Then by (b), (a), (d) and (c), we get,
\begin{align*}
{\sf d}_{\tau k, \sigma k}\cdot {\sf s}_{\sigma} x
&=  {\sf s}_{\tau k}^{\lambda}({\sf d}_{\lambda,\sigma k}\cdot {\sf s}_{\sigma}x)\\
&={\sf s}_{\tau k}^{\lambda}({\sf d}_{\lambda, \sigma k}\cdot {\sf s}_{\sigma k}^{\lambda}
{\sf s}_{{\sigma}(k\mapsto \lambda)}x)\\
&={\sf s}_{\tau k}^{\lambda}({\sf d}_{\lambda, \sigma k}\cdot {\sf s}_{{\sigma}(k\mapsto \lambda)}x)\\
&= {\sf d}_{\tau k,  \sigma k}\cdot {\sf s}_{\sigma(k\mapsto \tau k)}x.
\end{align*}
By $F$ is a filter and $(\tau k, \sigma k)\in E,$ we conclude that
$${\sf s}_{\sigma}x\in F \iff{\sf s}_{\sigma(k\mapsto \tau k)}x\in F.$$
The conclusion follows from the induction hypothesis.

We check that $f$ is a homomorphism:
\begin{itemize}
\item Boolean join:  Since $F$ is maximal we have:
\begin{align*}  
\bar{\sigma}\in f(x+y)
&\iff {\sf s}_{\sigma^+}(x+y)\in F\\
&\iff {\sf s}_{\sigma^+}x+{\sf s}_{\sigma^+}y\in F\\
&\iff {\sf s}_{\sigma^+} x \text { or } {\sf s}_{\sigma^+} y\in F\\
&\iff  \bar{\sigma}\in f(x)\cup f(y).
\end{align*}

\item Complementation: $$\bar{\sigma} \in f(-x) \iff {\sf s}_{\sigma^+}(-x)\in F \iff-{\sf s}_{\sigma^+x}\in F 
\iff {\sf s}_{\sigma^+} x\notin F \iff \bar{\sigma}\notin f(x).$$

\item Diagonal elements: Let $k,l<n$. Then we have:   
$\sigma\in f{\sf d}_{kl}\iff {\sf s}_{\sigma^{+}}{\sf d}_{kl}\in F \iff
{\sf d}_{\sigma k, \sigma l}\in F
\iff (\sigma k , \sigma l)\in E \iff
\sigma k/E=\sigma l/E \iff 
\bar\sigma(k)=\bar\sigma(l)
\iff\bar{\sigma} \in {\sf d}_{kl}.$

\item Cylindrifications: Let $k<n$ and $a\in A$.  Let $\bar{\sigma}\in {\sf c}_kf(a)$. 
Then for some $\lambda\in \mathfrak{n}$, we have
$\bar{\sigma}(k \mapsto  \lambda/E)\in f(a)$ 
hence 
${\sf s}_{\sigma^+(k\mapsto \lambda)}a\in F$ 
It follows from  the inclusion $a\leq {\sf c}_ka$ that, 
${\sf s}_{\sigma^+(k\mapsto \lambda)}{\sf c}_ka \in F$
so ${\sf s}_{\sigma^+}{\sf c}_ka\in F.$
Thus ${\sf c}_kf(a)\subseteq f({\sf c}_ka.)$

We prove the other more difficult inclusion that uses the condition of eliminating cylindrifiers.
Let $a\in A$ and $k<n$. Let $\bar\sigma'\in f{\sf c}_ka$ and let $\sigma=\sigma'\cup Id_{\mathfrak{n}\sim n}$. Then 
${\sf s}_{\sigma}^{\D}{\sf c}_ka={\sf s}_{\sigma'}^{\D}{\sf c}_ka\in F.$
Let $\lambda\in \{\eta \in \mathfrak{n}:  \sigma^{-1}\{\eta\}=\{\eta\}\}\smallsetminus \Delta a$, such a $\lambda$ exists because $\Delta a$ is finite, and 
$|\{i\in \mathfrak{n}:\sigma(i)\neq i\}|<\omega.$
Let $\tau=\sigma\upharpoonright \mathfrak{n}\smallsetminus \{k,\lambda\}\cup \{
(k,\lambda), (\lambda,k)\}.$
Then in $\B$ we have
${\sf c}_{\lambda}{\sf s}_{\tau}a={\sf s}_{\tau}{\sf c}_ka={\sf s}_{\sigma}{\sf c}_ka\in F.$
By the construction of $F$, there is some $u(\notin \Delta({\sf s}_{\tau}^{\D}a))$  
such that ${\sf s}_{u}^{\lambda}{\sf s}_{\tau}a\in F,$ so ${\sf s}_{\sigma(k\mapsto u)}a\in F.$
Hence $\sigma(k\mapsto u)\in f(a),$ from which we 
get that  $\bar{\sigma}'\in {\sf c}_k f(a)$.

\item Substitutions: Direct since substitution operations are Boolean endomorphisms.

\end{itemize}

\end{proof}

\section{Positive results on $\sf OTT$ for $L_n$}

Unless otherwise specified, $n$ will denote a finite ordinal $>2$. 
Now we turn to proving omitting types theorems for certain (not all) $L_n$ theories. 
But first a definition:
\begin{definition}\label{definition} Let $\A\in \RCA_n$ and let $\lambda$ be a cardinal. 
\begin{enumerate} 
\item If $\bold X=(X_i: i<\lambda)$ is  family of subsets of $\A$, we say that {\it $\bold X$ is omitted in $\C\in {\sf Crs}_{n}$}, 
if there exists an isomorphism 
$f:\A\to \C$ such that $\bigcap f(X_i)=\emptyset$ for all $i<\lambda$.  When we want to stress the role of $f$, 
we say that $\bold X$ is omitted in $\C$ via $f$. 

\item If $X\subseteq \A$ and $\prod X=0$, 
then we refer to $X$ as a {\it non-principal type} of $\A$.
\end{enumerate}
\end{definition}

Observe that $\A\in \RCA_n$ is completely representable $\iff$ $\A$ is atomic, and the single non-principal type of co-atoms 
can be omitted in a ${\sf Gs}_n$.
We recall certain cardinals that play a key role in (positive) omitting types theorems for $L_{\omega, \omega}$.
Let $\sf covK$ be the cardinal used in \cite[Theorem 3.3.4]{Sayed}.
The cardinal $\mathfrak{p}$  satisfies $\omega<\mathfrak{p}\leq 2^{\omega}$
and has the following property:
If $\lambda<\mathfrak{p}$, and $(A_i: i<\lambda)$ is a family of meager subsets of a Polish space $X$ (of  which Stone spaces of countable Boolean algebras are examples)  
then $\bigcup_{i\in \lambda}A_i$ is meager. For the definition and required properties of $\mathfrak{p}$, witness \cite[pp. 3, pp. 44-45, corollary 22c]{Fre}. 
Both cardinals $\sf cov K$ and $\mathfrak{p}$  have an extensive literature.
It is consistent that $\omega<\mathfrak{p}<\sf cov K\leq 2^{\omega}$ \cite{Fre},
so that the two cardinals are generally different, but it is also consistent that they are equal; equality holds for example in the Cohen
real model of Solovay and Cohen.  Martin's axiom implies that  both cardinals are the continuum.
To prove the main result on positive omitting types theorems, we need the following lemma due to Shelah:
\begin{lemma} \label{sh} Assume that $\lambda$ is an infinite regular cardinal. 
Suppose that $T$ is a first order theory,
$|T|\leq \lambda$ and $\phi$ is a formula consistent with $T$,  then there exist models $\M_i: i<{}^{\lambda}2$, each of cardinality $\lambda$,
such that $\phi$ is satisfiable in each,  and if $i(1)\neq i(2)$, $\bar{a}_{i(l)}\in M_{i(l)}$, $l=1,2,$, $\tp(\bar{a}_{l(1)})=\tp(\bar{a}_{l(2)})$,
then there are $p_i\subseteq \tp(\bar{a}_{l(i)}),$ $|p_i|<\lambda$ and $p_i\vdash \tp(\bar{a}_ {l(i)})$ ($\tp(\bar{a})$ denotes the complete type realized by
the tuple $\bar{a}$)
\end{lemma}
\begin{proof} \cite[Theorem 5.16, Chapter IV]{Shelah}.
\end{proof}

In the next theorem $n<\omega$:
\begin{theorem}\label{i} Let $\A\in \bold S_c\Nr_n\CA_{\omega}$ be countable.  Let $\lambda< 2^{\omega}$ and let 
$\bold X=(X_i: i<\lambda)$ be a family of non-principal types  of $\A$.
Then the following hold:
\begin{enumerate}
\item If $\A\in \Nr_n\CA_{\omega}$ and the $X_i$s are maximal non--principal ultrafilters,  then $\bold X$ can be omitted in a ${\sf Gs}_n$.
Furthrmore, the condition of maximality cannot be dispensed with,
\item Every subfamily of $\bold X$ of cardinality $< \mathfrak{p}$ can be omitted in a ${\sf Gs}_n$; in particular, every countable 
subfamily of $\bold X$ can be omitted in a ${\sf Gs}_n$,
\item  If $\A$ is simple, then every subfamily 
of $\bold X$ of cardinlity $< \sf covK$ can be omitted in a ${\sf Cs}_n$, 
\item It is consistent, but not provable (in $\sf ZFC$), that $\bold X$ can be omitted in a ${\sf Gs}_n$,
\item If $\A\in \Nr_n\CA_{\omega}$ and $|\bold X|<\mathfrak{p}$, then $\bold X$ can be omitted $\iff$ every countable subfamily of $\bold X$ can be omitted.   
If $\A$ is simple, we can replace $\mathfrak{p}$ by $\sf covK$. 
\item If $\A$ is atomic, {\it not necessarily countable}, but have countably many atoms, 
then any family of non--principal types can be omitted in an atomic $\sf Gs_n$; in particular, 
$\bold X$ can be omitted in an atomic ${\sf Gs}_n$; if $\A$ is simple, we can replace $\sf Gs_n$ by 
$\sf Cs_n$.
\end{enumerate}
\end{theorem}
\begin{proof}

For the first item we prove the special case when $\kappa=\omega$. The general case follows from the fact that (*) below holds for any infinite regular cardinal.
We assume that $\A$ is simple (a condition that can be easily removed). We have $\prod ^{\B}X_i=0$ for all $i<\kappa$ because,
$\A$ is a complete subalgebra of $\B$. Since $\B$ is a locally finite,  we can assume 
that $\B=\Fm_T$ for some countable consistent theory $T$.
For each $i<\kappa$, let $\Gamma_i=\{\phi/T: \phi\in X_i\}$.
Let ${\bold F}=(\Gamma_j: j<\kappa)$ be the corresponding set of types in $T$.
Then each $\Gamma_j$ $(j<\kappa)$ is a non-principal and {\it complete $n$-type} in $T$, because each $X_j$ is a maximal filter in $\A=\mathfrak{Nr}_n\B$.
(*) Let $(\M_i: i<2^{\omega})$ be a set of countable
models for $T$ that overlap only on principal maximal  types; these exist by lemma \ref{sh}.
Asssume for contradiction that for all $i<2^{\omega}$, there exists
$\Gamma\in \bold F$, such that $\Gamma$ is realized in $\M_i$.
Let $\psi:{}2^{\omega}\to \wp(\bold F)$,
be defined by
$\psi(i)=\{F\in \bold F: F \text { is realized in  }\M_i\}$.  Then for all $i<2^{\omega}$,
$\psi(i)\neq \emptyset$.
Furthermore, for $i\neq j$, $\psi(i)\cap \psi(j)=\emptyset,$ for if $F\in \psi(i)\cap \psi(j)$, then it will be realized in
$\M_i$ and $\M_j$, and so it will be principal.
This implies that $|\bold F|=2^{\omega}$ which is impossible. Hence we obtain a model $\M\models T$ omitting $\bold X$
in which $\phi$ is satisfiable. The map $f$ defined from $\A=\Fm_T$ to ${\sf Cs}_n^{\M}$ (the set algebra based on $\M$ \cite[4.3.4]{HMT2})
via  $\phi_T\mapsto \phi^{\M},$ where the latter is the set of $n$--ary assignments in
$\M$ satisfying $\phi$, omits $\bold X$. Injectivity follows from the facts that $f$ 
is non--zero and $\A$ is simple. 
For the second part of (1), we construct an atomic $\B\in \Nr_n\CA_{\omega}$ with uncountably many atoms that is not completely representable. This implies that the maximality condition 
cannot be dispensed with; else the set  of co--atoms of $\B$ call it $X$ will be a non--principal type that cannot be omitted, 
because any ${\sf Gs}_n$ omitting $X$ yields a complete representation of 
$\B$, witness the last paragraph in \cite{Sayed}.  
The construction is taken from \cite{bsl}. 

For (2) and (3), we can assume that $\A\subseteq_c \Nr_n\B$, $\B\in \Lf_{\omega}$. 
We work in $\B$. Using the notation on \cite[p. 216 of proof of Theorem 3.3.4]{Sayed} replacing $\Fm_T$ by $\B$, we have $\bold H=\bigcup_{i\in \lambda}\bigcup_{\tau\in V}\bold H_{i,\tau}$
where $\lambda <\mathfrak{p}$, and $V$ is the weak space ${}^{\omega}\omega^{(Id)}$,  
can be written as a countable union of nowhere dense sets, and so can 
the countable union $\bold G=\bigcup_{j\in \omega}\bigcup_{x\in \B}\bold G_{j,x}$.  
So for any $a\neq 0$,  there is an ultrafilter $F\in N_a\cap (S\setminus \bold H\cup \bold G$)
by the Baire category theorem. This induces a homomorphism $f_a:\A\to \C_a$, $\C_a\in {\sf Cs}_n$ that omits the given types, such that
$f_a(a)\neq 0$. (First one defines $f$ with domain $\B$ as on p.216, then restricts $f$ to $\A$ obtaining $f_a$ the obvious way.) 
The map $g:\A\to \bold P_{a\in \A\setminus \{0\}}\C_a$ defined via $x\mapsto (g_a(x): a\in \A\setminus\{0\}) (x\in \A)$ is as required. 
In case $\A$ is simple, then by properties of $\sf covK$, $S\setminus (\bold H\cup \bold G)$ is non--empty,  so
if $F\in S\setminus (\bold H\cup \bold G)$, then $F$ induces a non--zero homomorphism $f$ with domain $\A$ into a $\Cs_n$ 
omitting the given types. By simplicity of $\A$, $f$ is injective.

To prove independence, it suffices to show that $\sf cov K$ many types may not be omitted because it is consistent that $\sf cov K<2^{\omega}$. 
Fix $2<n<\omega$. Let $T$ be a countable theory such that for this given $n$, in $S_n(T)$, the Stone space of $n$--types,  the isolated points are not dense.
It is not hard to find such theories. One such (simple) theory is the following: 
Let $(R_i:i\in \omega)$ be a countable family of unary relations and for each disjoint and finite subsets
 $J,I\subseteq \omega$, let $\phi_{I,J}$ be the formula expressing `there exists $v$ such that $R_i(v)$ holds for all $i\in I$ and $\neg R_j(v)$ holds for all
$j\in J$. Let $T$ be the following countable theory $\{\phi_{I, J}: I, J\text { as above }\}.$  
Using a simple compactness argument one can show that 
$T$ is consistent. Furthermore, for each $n\in \omega$, $S_n(T)$ does not have isolated types at all, 
hence of course the isolated types are not dense in $S_n(T)$ for all $n$. 
Algebraically, this means that if $\A=\Fm_T$, then for all $n\in \omega$, $\Nr_n\A$ is atomless.
(Another example, is the  theory of random graphs.)
This condition excludes the existence of a prime model for $T$ because $T$ has a prime model $\iff$ the isolated points in 
$S_n(T)$ are dense for all $n$. A prime model which in this context is an atomic model,  omits any family of non--principal types (see the proof of the last item).
We do not want this to happen.
Using exactly the same argument in \cite[Theorem 2.2(2)]{CF}, one can construct 
a family $P$ of non--principal $0$--types (having no free variable) of  $T$,
such that $|P|={\sf covK}$ and $P$ cannot be omitted. 
Let $\A=\Fm_T$ and for $p\in P$, let $X_p=\{\phi/T:\phi\in p\}$. Then $X_p\subseteq \Nrr_n\A$, and $\prod X_p=0$,
because $\Nrr_n\A$ is a complete subalgebra of $\A$.
Then we claim that for any $0\neq a$, there is no set algebra $\C$ with countable base
and $g:\A\to \C$ such that $g(a)\neq 0$ and $\bigcap_{x\in X_p}f(x)=\emptyset$.
To see why, let $\B=\Nrr_n\A$. Let $a\neq 0$. Assume for contradiction, that there exists
$f:\B\to \D'$, such that $f(a)\neq 0$ and $\bigcap_{x\in X_i} f(x)=\emptyset$. We can assume that
$B$ generates $\A$ and that $\D'=\Nrr_n\B'$, where $\B'\in \Lf_{\omega}$.
Let $g=\Sg^{\A\times \B'}f$.
We show that $g$ is a homomorphism with $\dom(g)=\A$, $g(a)\neq 0$, and $g$ omits $P$, and for this, 
it suffices to show by symmetry that $g$ is a function with domain $A$. It obviously has co--domain $B'$. 
Settling the domain is easy: $\dom g=\dom\Sg^{\A\times \B'}f=\Sg^{\A}\dom f=\Sg^{\A}\Nrr_{n}\A=\A.$ 
Let $\bold K=\{\A\in \CA_{\omega}: \A=\Sg^{\A}\Nrr_{n}\A\}(\subseteq \sf Lf_{\omega}$). 
We show that $\bold K$ is closed under finite direct products. Assume that $\C$, $\D\in \bold K$, then we have 
$\Sg^{\C\times \D}\Nrr_{n}(\C\times \D)= \Sg^{\C\times \D}(\Nrr_{n}\C\times \Nrr_{n}\D)=
\Sg^{\C}\Nrr_{n}\C\times  \Sg^{\D}\Nrr_{n}\D=\C\times \D$.
Observe that (*):
 $$(a,b)\in g\land \Delta[(a,b)]\subseteq n\implies f(a)=b.$$
(Here $\Delta[(a,b)]$ is the {\it dimension set} of $(a,b)$ defined via $\{i\in \omega: {\sf c}_i(a,b)\neq (a,b)\}$).
Indeed $(a,b)\in \Nrr_{n}\Sg^{\A\times \B}f=\Sg^{\Nrr_{n}(\A\times \B)}f=\Sg^{\Nrr_{\alpha}\A\times \Nrr_{n}\B}f=f.$
Now suppose that $(x,y), (x,z)\in g$. We need to show that $y=z$ proving that $g$ is a function.
Let $k\in \omega\setminus n.$ Let $\oplus$ denote `symmetric difference'. Then (1):
$$(0, {\sf c}_k(y\oplus z))=({\sf c}_k0, {\sf c}_k(y\oplus z))={\sf c}_k(0,y\oplus z)={\sf c}_k((x,y)\oplus(x,z))\in g.$$
Also (2),
$${\sf c}_k(0, {\sf c}_k(y\oplus z))=(0,{\sf c}_k(y\oplus z)).$$
From (2) by observing that $k$ is arbitrarly chosen in $\omega\sim n$, we get (3): 
$$\Delta[(0, {\sf c}_k(y\oplus z))]\subseteq n$$
From (1) and (3) and  (*),  we get  $f(0)={\sf c}_k(y\oplus z)$ for any  $k\in \omega\setminus n.$
Fix $k\in \omega\sim n$. Then by the above, upon observing that $f$ is a homomophism, so that in particular $f(0)=0$, we get 
${\sf c}_k(y\oplus z)=0$. But  $y\oplus z\leq {\sf c}_k(x\oplus z)$, so $y\oplus z=0$, thus $y=z$. 
We have shown that $g$ is a function, $g(a)\neq 0$,  $\dom g=\A$ and $g$ omits $P$.
This contradicts that  $P$, by its construction,  cannot be omitted.
Assuming Martin's axiom, 
we get $\sf cov K=\mathfrak{p}=2^{\omega}$. Together with the above arguments this proves (4).

We now prove (5). Let $\A=\Nr_n\D$, $\D\in {}\sf Lf_{\omega}$ is countable. Let $\lambda<\mathfrak{p}.$ Let $\bold X=(X_i: i<\lambda)$ be as in the hypothesis.
Let $T$ be the corresponding first order theory, so
that $\D\cong \Fm_T$. Let $\bold X'=(\Gamma_i: i<\lambda)$ be the family of non--principal types in $T$ corresponding to $\bold X$. 
If $\bold X'$ is not omitted, then there is a (countable) realizing tree
for $T$, hence there is a realizing tree for a countable subfamily of $\bold X'$ in the sense of \cite[Definition 3.1]{CF}, 
hence a countable subfamily of $\bold X'$ 
cannot be omitted. Let $\bold X_{\omega}\subseteq \bold X$ be the corresponding countable 
subset of $\bold X$. Assume that $\bold X_{\omega}$ can be omitted in a $\sf Gs_n$, via $f$ say. 
Then  by the same argument used in proving item (4)  $f$ can be lifted to $\Fm_T$ omitting $\bold X'$, 
which is a contradiction.  We leave the part when $\A$ is simple to the reader. 

For (6): If $\A\in {\bold S}_n\Nr_n\CA_{\omega}$, is atomic and has countably many atoms, 
then any complete representation of $\A$, equivalently, an atomic representation of $\A$, equivalently, a representation of $\A$ 
omitting the set of co--atoms is as required. 
If $\A$ is simple and completely representable, 
then it is completely represented  
by a ${\sf Cs}_n$, and we are done. 
\end{proof}
By observing that if $T$ is an $L_n$ theory that has quantifier elimination, then $\Fm_T\in \Nr_n\CA_{\omega}$, we conclude:
 \begin{corollary}\label{ii} If $T$ is a complete countable $L_n$ theory that has quantifier elimination, and $\bold X=(X_i: i< 2^{\omega})$ are non principal complete types. 
Then there is a countable model of $T$ that omits $\bold X$.
\end{corollary}
The condition of maximality (completeness of types) cannot be omitted due to the Theorem \ref{bsl} and its following consequence:
\begin{corollary} For any infinite cardinal $\kappa$, there exists an atomic complete theory $T$ with $|T|\leq 2^{\kappa}$, that is to say, $\Fm_T$ is n atomic algebra, 
such that the single   type consisting of co-atom of $\Fm_T$, namely, $\Gamma=\{\neg \phi: \phi_{\equiv_T}\in \At\Fm_T\}$, $\Gamma$ is a non principal type, but  $\Gamma$ cannot be omitted
in any model of $T$.
\end{corollary} 
\subsection{$\sf OTT$ for algebtaizable versions of $L_{\omega,\omega}$; the finitary logics  with infinitary relation}

For $\alpha\geq \omega$, recall that$\sf Dc_{\alpha}$ denote the class of {\it dimension complemented $\CA_{\alpha}$s}, 
so that $\A\in {\sf Dc}_{\alpha}\iff\ \alpha\setminus \Delta x$ 
is infinite for every $x\in \A$. The cardinals $\sf covK$ and $\mathfrak{p}$ addressed next are defined above right before Theorem \ref{i}. 
Recall that we deviated from the stndard notation in \cite{HMT2} by denoting the class of generalized weak set algebras of dimension $\alpha$ by $\sf GCAws_{\alpha}$ rather than $\sf Gws_{\alpha}$; the last  is the notation used in \cite{HMT2}.

\begin{theorem} \label{ii} Let $\alpha$ be a countable infinite ordinal.
\begin{enumerate} 
\item There exists a countable 
atomic $\A\in \RCA_\alpha$ such that the non--principal types of co--atoms cannot be omitted in a $\sf Gs_{\alpha}$ (in the sense of definition \ref{definition} by considering infinite ordinals), 
We can replace $\sf Gs_{\alpha}$ by any set algebra in ${\sf RCA}_{\alpha}=\bold I\sf Gs_{\alpha}$. In particular, there exists a countable atomic $\B\in \sf RCA_{\alpha}$ such that the non-principal type consisting of co-atoms cannot be omtted in a $\sf GCAws_{\alpha}$.
\item If $\A\in \bold S_c\Nr_{\alpha}\CA_{\alpha+\omega}$ is countable, $\lambda$ a cardinal $<\mathfrak{p}$ 
and $\bold X=(X_i: i<\lambda)$ is a family of non--principal types, 
then $\bold X$ can be omited in a $\sf Gws_{\alpha}$ (in the same sense as in the previous item). 
If $\A$ is simple, and $\A\in \Nr_{\alpha}\CA_{\alpha+\omega}$, 
then we can replace $\mathfrak{p}$ by $\sf covK$.
\end{enumerate}
\end{theorem}
\begin{proof}

(1) Using an argument in \cite{HH}, 
one shows that if $\C\in \CA_{\omega}$ is completely representable 
$\C\models {\sf d}_{01}<1$, then $\At\C=2^{\omega}$. 
The argument is as follows: 
Suppose that $\C\models {\sf d}_{01}<1$. Then there is $s\in h(-{\sf d}_{01})$ so that if $x=s_0$ and $y=s_1$, we have
$x\neq y$. Fix such $x$ and $y$. For any $J\subseteq \omega$ such that $0\in J$, set $a_J$ to be the sequence with
$i$th co-ordinate is $x$ if $i\in J$, and is $y$ if $i\in \omega\setminus J$.
By complete representability every $a_J$ is in $h(1^{\C})$ and so it is in
$h(x)$ for some unique atom $x$, since the representation is an atomic one.
Let $J, J'\subseteq \omega$ be distinct sets containing $0$.   Then there exists
$i<\omega$ such that $i\in J$ and $i\notin J'$. So $a_J\in h({\sf d}_{0i})$ and
$a_J'\in h (-{\sf d}_{0i})$, hence atoms corresponding to different $a_J$'s with $0\in J$ are distinct.
It now follows that $|\At\C|=|\{J\subseteq \omega: 0\in J\}|=2^{\omega}$.
Take $\D\in {\sf Cs}_{\omega}$ with universe $\wp({}^{\omega}2)$. 
Then $\D\models {\sf d}_{01}<1$ and plainly $\D$ is completely representable. 
Using the downward \ls--Tarski theorem, take a countable  elementary subalgebra $\B$ of $\D$.
(This is possible because the signature of $\CA_{\omega}$ is countable.)
Then in $\B$ we have $\B\models {\sf d}_{01}<1$ because $\B\equiv \C$. But $\B$ 
cannot be completely 
representable, because if it were then by the above argument, we get that $|\At\B|=2^{\omega}$,
which is impossible because $\B$ is countable. The last part folloes from Theorem \ref{infinitecan} above. 

(2) Now we prove the second item, which is a generalization of \cite[Theorem 3.2.4]{Sayed}. 
Though the generalization is strict, in the sense that $\sf Dc_{\omega}\subsetneq \bold S_c\Nr_{\omega}\CA_{\omega+\omega}$
\footnote{It is not hard to see that the full set algebra with universe $\wp({}^{\omega}\omega)$ 
is in $\Nr_{\omega}\CA_{\omega+\omega}\subseteq \bold S_c\Nr_{\omega}\CA_{\omega+\omega}$ but it is not
in $\sf Dc_{\omega}$ because for any $s\in {}^{\omega}U$, $\Delta\{s\}=\omega$.}
the proof is the same. 
Without loss, we can take $\alpha=\omega$. Let $\A\in \CA_{\omega}$ be as in the hypothesis. 
For brevity, let $\beta=\omega+\omega$.  By hypothesis, we have 
$\A\subseteq_c\mathfrak{Nr}_{\omega}\D$, with $\D\in \CA_{\beta}$.
We can also assume that $\D\in {\sf Dc}_{\beta}$ by replacing, if necessary, $\D$ 
by
$\Sg^{\D}\A$.   Since $\A$ is a complete sublgebra of $\mathfrak{Nr}_{\omega}\D$  which in turn is a complete subalgebra 
of $\D$, we have $\A\subseteq_c \D$. Thus given $< \mathfrak{p}$ non--principal types in $\A$ 
they stay non--principal in $\D$.  Next one proceeds like in \cite{Sayed}  since 
$\D\in {\sf Dc}_{\beta}$ is countable;  this way omitting any $\bold X$ consisting of $<\mathfrak{p}$ 
non--principal types.
For all non-zero $a\in \D$, there exists $\B\in \sf Ws_{\beta}$  and a homomorphism $f_a:\D\to \B$ (not necessarily injective) 
such that $f_a(a)\neq \emptyset$ and $f_a$ omits $\bold X$. 
Let $\C=\bold P_{a\in \D, a\neq 0}\B_a\in \sf Gws_{\beta}$. Define $g:\D\to  \C$ by 
$g(x)=(f_a(x): a\in \D\setminus \{0\})$, and then relativize $g$ to $\A$ as follows:
Let $W$ be the top element of $\C$. Then  $W=\bigcup_{i\in I} {}^{\beta}U_i^{(p_i)}$, where 
$p_i\in  {}^{\beta}U_i$ and $^{\beta}U_i^{(p_i)}\cap {}^{\beta}U_j^{(p_j)}=\emptyset$, for $i\neq j\in I$.
Let   $V=\bigcup _{i\in I}{}^{\alpha}U_i^{(p_i\upharpoonright \alpha)}$. 
For $s\in V$, $s\in {}^{\alpha}U_i^{(p_i\upharpoonright \alpha)}$ (for a unique $i$), let $s^+=s\cup p_i\upharpoonright \beta\setminus \alpha$. 
Now define $f:\A\to \wp(V)$, via $a\mapsto \{s\in V: s^+\in g(a)\}$.
Then $f$ is as required.
Assume now that $\A$ is simple, with $\A=\mathfrak{Nr}_{\omega}\D$ and $\D\in {\sf Dc}_{\beta}$.
It suffices to show that  $\D$ is simple, too. Consider the function $F(I)=I\cap \A$, $I$ an ideal in $\D$. 
It is straightforward to check that $F$ establishes an isomorphism between
the lattice of ideals in $\D$ and the lattice of ideals of $\A$ (the order here is  of course $\subseteq$),  
with inverse $G(J)= {\Ig}^{\D}(J)$,  $J$ an ideal in $\A$, cf. \cite[Theorem 2.6.71, Remark 2.6.72]{HMT2} 
where $\Ig^{\D}J$ denotes the ideal of $\D$  generated by $J$. 
Thus $\A$ is simple $\iff$ $\D$ is simple.
\end{proof}

\subsection{Other variants of $L_{\omega, \omega}$}

Now we prove an omitting types theorem 
for a countable version of the so--called $\omega$--dimensional 
cylindric polyadic algebras with equality, in symbols $\sf CPE_{\omega}$, 
as defined in \cite{Fer2}.Consider the semigroup $\sf T$ generated by the set of 
transformations 
$\{[i|j], [i,j], i, j\in \omega, \sf suc, \sf pred\}$ defined on $\omega$. Then $\T$  is a 
{\it strongly rich} subsemigroup of $(^\omega\omega, \circ)$ in the sense of \cite{AU},
where $\sf suc$ and $\sf pred$ are the successor and predecessor functions on $\omega$, respectively.
For a set $X$, let $\B(X)$ denote the Boolean set algebra $\langle \wp(X), \cup, \cap, \sim\rangle$. 
Let $\sf K_\T$ 
be the class of set algebras of the form 
$\langle\B(V), {\sf C}_i, {\sf S}_{\tau}\rangle_{i\in \omega, \tau\in \sf T},$
where $V\subseteq {}^{\omega}U,$  $V$ is a {\it compressed space}, that is $V=\bigcup_{i\in I}{}^{\alpha}U_i^{(p)}$ where for each $i, j\in I$, $U_i=U_j$ or $U_i\cap U_j=\emptyset$. 
Let $\Sigma_1$ be the set of equations defined
in \cite{AU} axiomatizing ${\sf K}_{\T}$; that is 
${\sf Mod}\Sigma_1=\K_\T$. Here we {\it do not} have diagonal elements in the signature; the corresponding logic is a conservative extension of $L_{\omega, \omega}$ {\it without} equality, and it is
a proper extension.
Let $\sf Gp_{\T}$ be the class of set algebras of the form 
$\langle\B(V), {\sf C}_i, {\sf D}_{ij}, {\sf S}_{\tau}\rangle_{i,j\in \omega, \tau\in \sf T},$
where $V\subseteq {}^{\omega}U,$  $V$ a non--empty union (not necessarily a disjoint one) 
of cartesian spaces.  Here we have  diagonal elements in the signature; the corresponding logic is a {\it variant} of $L_{\omega, \omega}$ where quantifiers do not 
necessarily commute, so $L_{\omega, \omega}$ does not `embed' in this logic its (square Tarskian) semantics are different.
Let $\Sigma_2$ be the set of equations defining $\sf CPE_{\omega}$ 
in \cite[Definition 6.3.7]{Fer2} 
restricted to the countable signature of ${\sf Gp}_{\T}$.
In the next theorem complete additivity is given explicitly in the second item only. 
Any algebra $\A$ satisifying $\Sigma_2$ is completely additive (due to the presence of diagonal elements), cf. \cite{Fer2}.
\begin{theorem}\label{three}
\begin{enumerate}
\item If $\A\models \Sigma_2$ is countable  
and $\bold X=(X_i: i<\lambda)$, $\lambda<\mathfrak{p}$ is a family of subsets of
$\A$, such that $\prod X_i=0$ for all $i<\lambda$, then there exists $\B\in \sf Gp_{\T}$ and an isomorphism $f:\A\to \B$ such that 
$\bigcap_{x\in X_i}f(x)=\emptyset$ for 
all $i<\lambda$.
\item  If $\A\models \Sigma_1$ is countable,  and completely additive 
and $\bold X=(X_i: i<\lambda)$, $\lambda<\mathfrak{p}$ is a family of subsets of
$\A$, such that $\prod X_i=0$ for all $i<\lambda$, then there exists $\B\in \sf K_{\T}$ and an isomorphism $f:\A\to \B$ such that 
$\bigcap_{x\in X_i}f(x)=\emptyset$ for 
all $i<\lambda$.
\item In particular, for both cases any countable atomic algebra is completely representable.
\end{enumerate}
\end{theorem}
\begin{proof}
For brevity, throughout the proof of the first two items,  let $\alpha=\omega+\omega$.
By strong richness of $\T$, it can be proved that $\A=\Nrr_{\omega}\B$ where  $\B$ is an $\alpha$--dimensional
dilation with substitution operators coming from a countable subsemigroup 
$\sf S\subseteq ({}^{\alpha}\alpha, \circ)$ \cite{AU}. 
It suffices to show that for any non--zero $a\in \A$, there exist a countable $\D\in {\sf Gp}_{\T}$ and a
homomorphism (that is not necessarily injective)
$f:\A\to \D$,  such that $\bigcap_{x\in X_i} f(x)=\emptyset$ for all $i\in \omega$ 
and $f(a)\neq 0$. 
So fix non--zero $a\in \A$. For $\tau\in \sf S$, set $\dom(\tau)=\{i\in \alpha: \tau(i)\neq i\}$ and $\rng(\tau)=\{\tau(i): i\in \dom(\tau)\}$.  
Let $\sf adm$ be  the set of admissible substitutions in $\sf S$, where 
now $\tau\in \sf adm$ if $\dom\tau\subseteq \omega$ and $\rng\tau\cap \omega=\emptyset$.  Since $\sf S$ is countable, we have
$|\sf adm|\leq\omega$; in fact it can be easily shown that $|\sf adm|=\omega$. 
Then for all $i< \alpha$, $p\in \B$ and $\sigma\in \sf adm$,
$\s_{\sigma}{\sf c}_{i}p=\sum_{j\in \alpha} \s_{\sigma}{\sf s}_j^ip.$
By $\A=\Nrr_{\omega}\B$ we also have, for each $i<\omega$, $\prod^{\B}X_i=0$, since $\A$ is a complete subalgebra of $\B$.
Because substitutions are completely additive, for all $\tau\in \sf adm$ and all $i<\lambda$,
$\prod {\sf s}_{\tau}^{\B}X_i=0$.
For better readability, for each $\tau\in \sf adm$, for each $i\in \omega$, let
$X_{i,\tau}=\{{\sf s}_{\tau}x: x\in X_i\}.$
Then by complete additivity, we have:
$(\forall\tau\in {\sf adm})(\forall  i\in \lambda)\prod {}^{\B}X_{i,\tau}=0.$
Let $S$ be the Stone space of $\B$, whose underlying set consists of all Boolean ultrafilters of
$\B$ and for $b\in B$, let $N_b$ denote the clopen set consisting of all ultrafilters containing $b$. 
Then from the suprema obtained above, it follows that for $x\in \B,$ $j<\alpha,$ $i<\lambda$ and
$\tau\in \sf adm$, the sets
$\bold G_{\tau,j,x}=N_{{\sf s}_{\tau}{\sf c}_jx}\setminus \bigcup_{i} N_{{\sf s}_{\tau}{\sf s}_i^jx}
\text { and } \bold H_{i,\tau}=\bigcap_{x\in X_i} N_{{\sf s}_{\tau}x}$
are closed nowhere dense sets in $S$.
Also each $\bold H_{i,\tau}$ is closed and nowhere
dense. Like before, we can assume that $\B$ is countable by assuming that $\A$ generates $\B$ is the presence of $|\\alpha|=(|A|=\omega$) many operations.
Let $\bold G=\bigcup_{\tau\in \sf adm}\bigcup_{i\in \alpha}\bigcup_{x\in B}\bold G_{\tau, i,x}
\text { and }\bold H=\bigcup_{i\in \lambda}\bigcup_{\tau\in  \sf adm}\bold H_{i,\tau}.$
Then $\bold H$ is meager, that is it can be written as a countable union of nowghere dense sets. This follows from the properties 
of $\mathfrak{p}$  
By the Baire Category theorem  for compact Hausdorff spaces,
we get that $X=S\smallsetminus \bold H\cup \bold G$ is dense in $S$,
since $\bold H\cup \bold G$ is meager, because $\bold G$ is meager, too, since
$\sf adm$, $\alpha$ and $\B$ are all countable.
Accordingly, let $F$ be an ultrafilter in $N_a\cap X$, then 
by its construction $F$ is a {\it perfect ultrafilter} \cite[p.128]{Sayedneat}.
Let $\Gamma=\{i\in \alpha:\exists j\in \omega: {\sf c}_i{\sf d}_{ij}\in F\}$.
Since $\c_i{\sf d}_{ii}=1$, then $\omega\subseteq \Gamma$. Furthermore the inclusion is proper, 
because for every $i\in \omega$, there is a $j\in \alpha\setminus \omega$ such that ${\sf d}_{ij}\in F$. 
Define the relation $\sim$ on $\Gamma$ via  $m\sim n\iff {\sf d}_{mn}\in F.$ 
Then $\sim$ is an equivalence relation because for all $i, j, k\in \alpha$, ${\sf d}_{ii}=1\in F$, ${\sf d}_{ij}={\sf d}_{ji}$,
${\sf d}_{ik}\cdot {\sf d}_{kj}\leq {\sf d}_{lk}$ and filters are closed upwards.
Now we show that the required  representation will be a $\Gp_{\T}$ with base
$M=\Gamma/\sim$. One defines the  homomorphism $f$ using the hitherto obtained perfect ultrafilter $F$ as follows:
For $\tau\in {}^{\omega}\Gamma$, 
such that $\rng(\tau)\subseteq \Gamma\setminus \omega$ (the last set is non--empty, because $\omega\subsetneq \Gamma$),
let  $\bar{\tau}: \omega\to M$ be defined by $\bar{\tau}(i)=\tau(i)/\sim$  
and write $\tau^+$ for $\tau\cup Id_{\alpha\setminus \omega}$. 
Then $\tau^+\in \sf adm$, because $\tau^+\upharpoonright \omega=\tau$, $\rng(\tau)\cap \omega=\emptyset,$ 
and $\tau^+(i)=i$ for all $i\in \alpha\setminus \omega$. 
Let $V=\{\bar{\tau}\in {}^{\omega}M: \tau: \omega\to \Gamma,  \rng(\tau)\cap \omega=\emptyset\}.$
Then $V\subseteq {}^{\omega}M$ is non--empty (because $\omega\subsetneq \Gamma$).
Now define $f$ with domain $\A$ via: 
$a\mapsto \{\bar{\tau}\in V: \s_{\tau^+}^{\B}a\in F\}.$ 
Then $f$ is well defined, that is, whenever $\sigma, \tau\in {}^{\omega}\Gamma$ and 
$\tau(i)\setminus \sigma(i)$ for all $i\in \omega$, then for any $x\in \A$, 
$\s_{\tau^+}^{\B}x\in F\iff \s_{\sigma^+}^{\B}x\in F$.
Furthermore $f(a)\neq 0$, since ${\sf s}_{Id}a=a\in F$ 
and $Id$ is clearly admissable.
The congruence relation just defined
on $\Gamma$  guarantees
that the hitherto defined homomorphism respects the diagonal elements.
As before, for the other operations,
preservation of cylindrifiers
is guaranteed by the condition that
$F\notin G_{\tau, i, p}$ for all $\tau\in {\sf adm}, i\in \alpha$
and all $p\in A$. 
For omitting the given family 
of non--principal types, we use that 
$F$ is outside $\bold H$, too.  This means (by definition) that for each $i<\lambda$ and each  $\tau\in  \sf adm$
there exists $x\in X_i$, such that
${\sf s}_{\tau}^{\B}x\notin F$. Let $i<\lambda$. If $\bar{\tau}\in V\cap \bigcap _{x\in X_i}f(x)$, 
then $\s_{\tau^+}^{\B}x\in F$ which  is impossible because $\tau^+\in \sf adm$. 
We have shown that for each $i<\omega$, $\bigcap_{x\in X_i}f(x)=\emptyset.$

For the second required one deals with all substitutions in the semigroup $\sf S$ determining the signature of the dilation not just $\sf adm$, namely, the admissable ones as defined 
above.
More succinctly, now {\it all} substitutions in $\sf S$ are admissable.  Other than that, the idea is essentially the same appealing to the Baire category theorem.
Let $\T$ be as above. Assume that  $\A\models \Sigma_1$ is countable, and fix non--zero 
$a\in \A$. Similarly to the first part we will construct a set algebra $\C$ in $\sf K_{\T}$ and a homomorphism 
$f:\A\to \C$ omitting the given non--principal types and satisfying that $f(a)\neq 0$.
By \cite{AU}, there exists $\B$ such that  $\A=\Nrr_{\omega}\B$ and the signature of $\B$ has, besides all the Boolean operations,  
all cylindrifiers ${\sf c}_i: i\in \alpha$, and
the substitutions are determined by a semigroup defined from the rich semigroup $\sf T$. Substitutions
in the signature of $\B$ are indexed by transformations in 
${\sf S}$; which we explicitly describe. The semigroup $\sf S$ is the subsemigroup of  $^{\alpha}\alpha$
generated by the set $\{\bar{\tau}: \tau\in \T\}$ together with all replacements and transpositions on $\alpha$.  
Here  $\bar{\tau}$ is the transformation 
that agrees with $\tau$ on $\omega$ and otherwise is the identity.
For all $i< \alpha$, $p\in \B$, we have 
${\sf c}_{i}p=\sum_{j\in \alpha} {\sf s}_j^ip.$
By  $\A=\Nrr_{\omega}\B$ we also have, for each $i<\omega$, $\prod^{\B}X_i=0$, since $\A$ is a complete subalgebra of $\B$.
Let $V$ be the generalized $\omega$-dimensional 
weak space $\bigcup_{\tau\in S}{}^{\omega}\alpha^{(\tau)}$.
Recall that $^{\omega}\alpha^{(\tau)}=\{s\in {}^{\omega}\alpha: |\{i\in \omega: {\sf s}_i\neq \tau_i\}|<\omega\}$.
For each $\tau\in V$ and for each $i\in \lambda$, let
$X_{i,\tau}=\{{\sf s}_{\bar{\tau}}^{\B}x: x\in X_i\}.$
Here we are using that for any $\tau\in V$, $\bar{\tau}\in {\sf S}.$
By complete additivity {\it which is given as an assumption}, it follows that 
$(\forall\tau\in V)(\forall  i\in \kappa)\prod{}^{\B}X_{i,\tau}=0.$
Let $S$ denote the Stone space of the boolean part of $\B$.
Like before, for $p\in B$, let $N_p$ be the clopen set of $S$ consisting of all ultrafilters
of the boolean part of $\B$ containing $p$. 
Then for $x\in \B,$ $j<\alpha,$ $i<\lambda$, $\tau\in \sf S$ (using the suprema just established) , the sets
$\bold G_{j,x}=N_{{\sf c}_jx}\setminus \bigcup_{i} N_{{\sf s}_i^jx}
\text { and } \bold H_{i,\tau}=\bigcap_{x\in X_i} N_{{\sf s}_{\tau}x}$
are closed nowhere dense sets in $S$.
Also each $\bold H_{i,\tau}$ is closed and nowhere
dense.
Let $\bold G=\bigcup_{i\in \alpha}\bigcup_{x\in B}\bold G_{i,x}
\text { and }\bold H=\bigcup_{i\in \lambda}\bigcup_{\tau\in  \sf S}\bold H_{i,\tau}.$
Then $\bold H$ is meager, since it is a countable union of nowhere dense sets.  
Once more by the Baire Category theorem  for compact Hausdorff spaces,
we get that $X=S\smallsetminus \bold H\cup \bold G$ is dense in $S$,
Let $F$ be an ultrafilter in $N_a\cap X$. 
One builds the required represention from $F$ as follows \cite{AU}: 
Let $\wp(V)$ be the full boolean set algebra with unit $V$. 
Let $f$ be the function with domain $A$
such that $f(a)=\{\tau\in V: {\sf s}_{\bar{\tau}}^{B}a\in F\}.$ 
Then 
$f$ is the desired homorphism from $\A$ into 
the set algebra 
$\langle \wp(V), {\sf c}_i, {\sf s}_{\tau}\rangle_{i\in \omega,\tau\in \T}.$ 
In particular,  $f(a)\neq 0$, because $Id\in f(a)$. That $f$ omits the given 
non--principal types is exactly like the first part, modulo 
replacing $\sf adm$ by (the whole of the semigroup) $\sf S$.

For the last part, given $\A$ as in the hypothesis, the required 
follows by omitting the non--principal type consisting  of co-atoms obtaining a complete representation of $\A$.
\end{proof}

\section{Omitting types, atom-canonicity and non-finite axiomatizability}

(1) A Theorem of Vaught in basic model theory, says that a countable atomic $L_{\omega,\omega}$ theory 
$T$ has a unique atomic (equivalently in this context prime) model. This can be proved by a direct application of the clssical Orey-Henkin Omitting Types Theorem.
The unique atomic atomic model is the 'smallest' models of $T$, in the sense that it   
elementary embeds into other models of $T$. The last theorem says that Keisler's logics which allow formulas of infinite length and quantification on infinitely many variables, 
enjoys a form of Vaught's theorem. And in Keisler's logics there is the additional advantage that there is no restrictions on the cardinality of atomic theories (algebras)  considered. For 
$L_{\omega, \omega}$, Vaught's theorem is known to fail for theories having 
uncountable cadinality. 
If $T$ is an atomic theory in Keisler's logic, and the Tarski-Lindenbaum atomic quotient algebra $\Fm_T$ happens to be completely additve, 
then $T$ has an atomic model. In contrast, we actually showed that Vaught's theorem fails for $L_n$ when we substantially broaden the class of permissable models; it fails 
even for `$n+3$-square models.' 
From Theorem \ref{can}, for $2<n<\omega$, there is a countable atomic $L_n$ theory that lacks even an atomic $n+3$-square model (let alone an ordinary atomic model), 
i.e a complete $n+3$-square representation 
of the Tarski Lindenbaum quotient algebra $\Fm_T(\in \RCA_n)$.

(2)  Let $2<n<l\leq m\leq \omega$. 
Consider the statemet ${\sf notVT}(l, m)$: {\it There exists a countable, complete and  atomic $L_n$ first order theory $T$ in a signature $L$ such that  the type 
$\Gamma$ consisting of co-atoms in the  cylindric Tarski-Lindenbaum quotient algebra $\Fm_T$ is realizable in every {\it $m$--square} model, 
but $\Gamma$ cannot be isolated using $\leq l$ variables, where $n\leq l<m\leq  \omega$.}
An $m$-square model of $T$ is an $m$-square represenation of $\Fm_T$.
The statement  $\sf notVT(l, m)$, short for Vaught's Theorem ({\it $\sf VT$) fails at (the parameters) $l$ and $m$.}
Let ${\sf VT}(l, m)$ stand for {\it {\sf VT} holds at $l$ and $m$}, so that by definition $\notVT(l, m)\iff \neg {\sf VT}(l, m)$. We also include $l=\omega$ in the equation by 
defining ${\sf VT}(\omega, \omega)$ as {\sf VT} holds for $L_{\omega, \omega}$: Atomic countable first order theories have atomic countable models. 
For $2<n<l\leq m\leq \omega$ and $l=m=\omega$,  it is likely and plausible that {\it (***): $\VT(l, m)\iff l=m=\omega$}.  
In other words:  {\it Vaught's theorem holds only in the limiting  case when $l\to \infty$ and $m=\omega$ and not `before'.} 
We give  sufficient condition for (***)  to happen.
The following definition to be used in the sequel is taken from \cite{ANT}:

Let $\R$ be a relation algebra, with non--identity atoms $I$ and $2<n<\omega$. Assume that  
$J\subseteq \wp(I)$ and $E\subseteq {}^3\omega$.
Recall that $(J, E)$ is a {\it strong $n$--blur} for $\R$ if it $(J, E)$ is an $n$--blur of $R$ in the sense of \cite[Definition 3.1]{ANT},  that is to say $J$ is a complex $n$ blur and $E$ is an index blur 
such that the complex 
$n$--blur  satisfies:
$$(\forall V_1,\ldots V_n, W_2,\ldots W_n\in J)(\forall T\in J)(\forall 2\leq i\leq n)
{\sf safe}(V_i,W_i,T).$$ 
\begin{theorem}\label{fl} For $2<n<\omega$ and $n\leq l<\omega$, $\notVT(n, n+3)$ and $\notVT(l, \omega)$ hold.
Furthermore, if for each $n<m<\omega$, there exists a finite relation algebra $\R_m$ having $m-1$ strong blur and no
$m$-dimensional relational basis, then
(***) above for ${\sf VT}$ holds.
\end{theorem}

\begin{proof}
We start by the last part. Let $\R_m$ be as in the hypothesis with strong $m-1$--blur $(J, E)$ and $m$-dimensional relational basis. 
We `blow up and blur' $\R_m$ in place of the Maddux algebra 
$\mathfrak{E}_k(2, 3)$ blown up and blurred in
\cite[Lemma 5.1]{ANT}, 
where $k<\omega$ is the number of non--identity atoms 
and  $k$ depends recursively on $l$, giving the desired`strong' $l$--blurness, cf. \cite[Lemmata  4.2, 4.3]{ANT}. 
The relation algebra ${\mathfrak Bb}(\R_m, J, E)$, obtained by blowing up and blurring $\R_m$ with respect to $(J, E)$,  is $\Tm\bf At$ (the term algebra). For brevity call it $\cal R$.
Now take $\A=\mathfrak{Bb}_n(\R_m, J, E)$ as defined in \cite{ANT} to be the $\CA_n$ obtained after blowing up and blurring $\R_m$ to a weakly representable relation algebra atom structure, 
namely,  ${\bf At}=\At\cal R$. 
Here by \cite[Theorem 3.2 9(iii)]{ANT},  ${\sf Mat}_n\At\cal R$ (the set of $n$-basic matrices on $\At\cal R$) is a $\CA_n$ atom structure and 
$\A$ is an atomic  subalgebra of $\Cm{\sf Mat}_n(\At\cal R)$ containing $\Tm{\sf Mat}_n(\At\cal R)$, cf. \cite{ANT}.
In fact,  by \cite[item (3) pp.80]{ANT},  $\A\cong \mathfrak{Nr}_n{\mathfrak Bb}_l(\R_m, J, E)$.The last algebra ${\mathfrak Bb}_l(\R_m, J, E)$ is defined 
and the isomorphism holds because $\R_m$ has a strong $l$-blur.
The embedding $h:\Rd_n\mathfrak{Bb}_l(\R_m, J, E)\to \A$ defined via 
$x\mapsto \{M\upharpoonright n: M\in x\}$ restricted to   
$\mathfrak{Nr}_n{\mathfrak Bb}_l(\R_m, J, E)$ 
is an isomorphism onto $\A$ \cite[p.80]{ANT}. Surjectiveness uses the displayed condition in Definition \ref{strongblur} of {\it strong $l$-blurness}.  
Then  $\A\in \RCA_n\cap \Nr_n\CA_l$, but $\A$ has no complete $m$-square representation.
For if it did, then this induces an $m$--square representation of $\Cm\At\A$,   
But $\Cm\At\A$ does not have an $m$-square representation, because $\R$ does not have an $m$-dimensional relational basis,  
and $\R\subseteq \Ra\Cm\At\A$. So an $m$-square representation of $\Cm\At\A$ induces one of $\R$ which  that $\R$ has no $m$-dimensional relational basis, a contradiction.

We prove $\notVT(m-1, m)$, hence the required, namely, (***).
By \cite[\S 4.3]{HMT2}, we can (and will) assume that $\A= \Fm_T$ for a countable, simple and atomic theory $L_n$ theory $T$.  
Let $\Gamma$ be the $n$--type consisting of co--atoms of $T$. Then $\Gamma$ is realizable in every $m$--square model, for if $\Mo$ is an $m$--square model omitting 
$\Gamma$, then $\Mo$ would be the base of a complete  $m$--square  representation of $\A$, and so $\A\in \bold S_c\Nr_n{\sf D}_m$ which is impossible.
Suppose for contradiction that $\phi$ is an $m-1$ witness, so that $T\models \phi\to \alpha$, for
all $\alpha\in \Gamma$, where recall that $\Gamma$ is the set of coatoms.
Then since $\A$ is simple, we can assume
without loss  that $\A$ is a set algebra with 
base $M$ say.
Let $\Mo=(M,R_i)_{i\in \omega}$  be the corresponding model (in a relational signature)
to this set algebra in the sense of \cite[\S 4.3]{HMT2}. Let $\phi^{\Mo}$ denote the set of all assignments satisfying $\phi$ in $\Mo$.
We have  $\Mo\models T$ and $\phi^{\Mo}\in \A$, because $\A\in \Nr_n\CA_{m-1}$.
But $T\models \exists x\phi$, hence $\phi^{\Mo}\neq 0,$
from which it follows that  $\phi^{\Mo}$ must intersect an atom $\alpha\in \A$ (recall that the latter is atomic).
Let $\psi$ be the formula, such that $\psi^{\Mo}=\alpha$. Then it cannot
be the case
that $T\models \phi\to \neg \psi$,
hence $\phi$ is not a  witness,
contradiction and we are done.
Finally, $\notVT(n, n+3)$ and $\notVT(l,  \omega)$ $(n\leq l<\omega)$ follow from Theorem  \ref{can} and \cite{ANT}. 
 \end{proof}

(3)  Let $2<n<\omega$. 
For any $m>n$
there exists an $n$--variable formula that cannot be proved using $m-1$ variables, but can 
be proved using $m$ variables \cite[Theorem 15.17]{HHbook}, using any standard Hilbert style proof system 
\cite[\S 4.3]{HMT2}.
To prove this,  for each $m>n+1$  Hirsch and Hodkinson constructed a finite relation algebra, 
such that $\R_m$ has an $m-1$ dimensional hyperbasis, but no $m$--dimensional hyperbasis 
\cite[\S 15.2-15.4]{HHbook}.
To prove that $\sf VT$ fails everywhere, as defined above,  one needs to construct, for each $n+1<m<\omega$, a finite relation algebra $\R_m$ having a strong $m-1$ blur, but no $m$--dimensional basis. 
In this case blowing up and blurring $\R_m$ gives a(n infinite) relation algebra having an $m-1$ dimensional cylindric basis, 
{\it whose \de\ completion} has no $m$--dimensional basis. 

(4) Coming back full circle we reprove strong non-finite axiomatizibility results refining Monk's obtained by Maddux and Biro.
Let $2<n\leq l<m\leq \omega$. In ${\sf VT}(l, m)$, while the parameter $l$ measures how close we are to $L_{\omega, \omega}$,  $m$ measures the {\it `degree' of squareness} of permitted models. 
Using elementary calculas terminology one can view ${\sf lim}_{l\to \infty}{\sf VT}(l, \omega)={\sf VT}(\omega, \omega)$ 
algebraically using ultraproducts as follows. Fix $2<n<\omega$. For each $2<n\leq l<\omega$, 
let $\R_l$ be the finite Maddux algebra $\mathfrak{E}_{f(l)}(2, 3)$ with strong $l$--blur
$(J_l,E_l)$ and $f(l)\geq l$ as specified in \cite[Lemma 5.1]{ANT} (denoted by $k$ therein).  
 Let ${\cal R}_l={{\mathfrak{Bb}}}(\R_l, J_l, E_l)\in \sf RRA$ and let 
$\A_l=\mathfrak{Nr}_n{{\mathfrak{Bb}}}_l(\R_l, J_l, E_l)\in \RCA_n$. 
Then $(\At{\cal R}_l: l\in \omega\sim n)$, and $(\At\A_l: l\in \omega\sim n)$ are sequences of weakly representable atom structures 
that are not strongly representable with a completely representable 
ultraproduct.
\begin{corollary}\label{maddux}(Monk, Maddux, Biro, Hirsch and Hodkinson)
Let $2<n<\omega$.  Then the  set of equations using only one variable that holds in each of the varieties ${\sf RCA}_n$ and $\sf RRA$,  
together with any finite first order definable expansion of each, 
cannot be derived from any finite set of equations valid in 
the variety \cite{Biro, maddux}. Furthermore,  ${\sf LCA}_n$ 
is not finitely axiomatizable.
\end{corollary}


\begin{thebibliography}{}
\bibitem{Andreka} H. Andr\'eka, {\it Complexity of equations valid in algebras of relations.} Annals of Pure and Applied  Logic {\bf 89}(1997), pp.149--209.

\bibitem{a} H. Andr\'eka {\it Finite axiomatizability of $\bold S\Nr_n\CA_{n+1}$ and non--finite axiomatizability of 
$\bold S\Nr_n\CA_{n+2}$.} Lecture notes,  Algebraic Logic Meeting, Oakland, 
CA (1990).











\bibitem{1} H. Andr\'eka, M. Ferenczi and I. N\'emeti (Editors), {\bf Cylindric-like Algebras 
and Algebraic Logic.} Bolyai Society Mathematical Studies {\bf 22} (2013).
 

\bibitem{ANT}  H. Andr\'eka,  I. N\'emeti and T. Sayed Ahmed,  {\it Omitting types for finite variable fragments and complete representations.}
Journal of Symbolic Logic. {\bf 73} (2008) pp. 65--89.






\bibitem {Biro} B. Bir\'o.  {\it Non-finite axiomatizability results in algebraic logic},
Journal of Symbolic Logic,   {\bf 57}(3)(1992), pp. 832--843.








\bibitem{AGMNS}
H. Andr\'eka, S. Givant, S. Mikulas, I. N\'emeti I. and A. Simon ,
{\it Notions of density
that imply representability in algebraic logic.}
Annals of Pure and Applied logic, {\bf 91}(1998), pp. 93--190.

\bibitem{modal} P. Blackburn, M. de Rijke, and Y. Venema {\it Modal logic} Cambridge University 
Press (2001).


\bibitem{b} J. Bulian and I. Hodkinson,
\emph{Bare canonicity of representable cylindric and polyadic algebras,}
Annals of Pure and Applied Logic \textbf{164} (2013) 884-906.

\bibitem{CF} E. Casanovas, R. Farre {\it Omitting types in incomplete theories},
Journal of Symbolic Logic, {\bf 61}(1)(1996), p. 236--245.




\bibitem{DM} A. Daigneault and J.D. Monk,
{\it Representation Theory for Polyadic algebras}.
Fund. Math. {\bf 52}(1963) p.151--176.




\bibitem{Erdos}
P. Erd\"{o}s, \emph{Graph theory and probability}. Canadian Journal of
Mathematics, vol. 11 (1959), pp. 34-38.


\bibitem{Fre} D.H. Fremlin {\it Consequences of Martin's axiom}. 
Cambridge University Press, 1984.




\bibitem{Fer2} M. Ferenczi, {\it A new representation theory, 
representing cylindric--like algebras by relativized set algebras.} In \cite{1}.

\bibitem{g} R. Goldblatt, {\it Algebraic polyamodal logic: a survey} 
Logic Journal of IGPL {\bf 8}(4) (2000), pp. 393--456.

\bibitem{g1}Goldblatt {\it Persistence and atomic generation for varieties of Boolean algebras with operators}
Sutdia Logica {\bf 68}(2) (2001) pp.155--157


\bibitem{HMT2}  L. Henkin, J.D. Monk and  A. Tarski {\it Cylindric Algebras Parts I, II}.
North Holland, 1971.



\bibitem{r} R. Hirsch, {\it Relation algebra reducts of cylindric algebras and complete representations},
Journal of Symbolic Logic, {\bf 72}(2) (2007), pp. 673--703.

\bibitem{r2} R. Hirsch {\it Corrigendum to `Relation algebra reducts of cylindric algebras and complete representations'} Journal of Symbolic Logic,
{\bf 78}(4) (2013), pp. 1345--1348.



\bibitem{HH} R. Hirsch and I. Hodkinson {\it Complete representations in algebraic logic},
Journal of Symbolic Logic, {\bf 62}(3)(1997) p. 816--847.

\bibitem{HHbook}  R. Hirsch and I. Hodkinson,  {\it Relation algebras by games.}
Studies in Logic and the Foundations of Mathematics, {\bf 147} (2002).


\bibitem{strongca} R. Hirsch and I. Hodkinson {\it Strongly  representable atom structures  of cylindic  algebras}, Journal of Symbolic Logic 
{\bf 74}(2009), pp. 811--8
28

\bibitem{HHbook2} R. Hirsch and I. Hodkinson {\it  Completions and complete representations}, in \cite{1} pp. 61--90.

\bibitem{k2} R. Hirsch, I. Hodkinson and 
 Kurucz, {\it On modal logics between $K\times K\times K$ and $S5\times S5\times S5$.} Journal of Symbolic Logic \textbf{68} (3-4) (2012) pp. 257-285.




\bibitem{t} R. Hirsch and T. Sayed Ahmed, {\it The neat embedding problem for algebras other than cylindric algebras
and for infinite dimensions.} Journal of Symbolic Logic {\bf 79}(1) (2014), pp .208--222.




\bibitem{Hodkinson} I. Hodkinson  I., {\it Atom structures of relation and cylindric algebras.} Annals of pure and applied logic,
{\bf  89} (1997), pp. 117--148.





\bibitem{k} A. Kurucz {\it Representable cylindric algebras and many dimensional modal logics}. In \cite{1}, pp.185--204













\bibitem{MT} M. Khaled and T. Sayed Ahmed {\it On complete representations of algebras in  logic} Logic Journal of $IGPL$ {\bf 17}(2009)pp 267-272.




\bibitem{maddux}  R. Maddux  {\it Non finite axiomatizability results for cylindric and relation algebras},
Journal of Symbolic Logic (1989) {\bf 54}, pp. 951--974.





\bibitem{ST} Sain I., Thompson, R. J., {\it Strictly finite Schema 
Axiomatization of Quasi polyadic algebras.} 
In ` Algebraic Logic' North Holland, Editors 
Andr\'eka, H., Monk, D., N\'emeti, I. p.539--572. 





\bibitem{Shelah} S. Shelah, {\it Classification theory: and the number of non-isomorphic models}.
Studies in Logic and the Foundations of Mathematics  (1990).

\bibitem{bsl} T. Sayed Ahmed,  {\it Neat embedding is nt sufficient for complete representations}
Bulletin Section of Logic {\bf 36}(1) (2007) pp. 29--36.




\bibitem{Basim} B. Samir and T. Sayed Ahmed {\it A Neat Embedding 
Theorem for expansions of cylindric algebras.} Logic Journal of $IGPL$ {\bf 15}(2007) pp.41--51.


\bibitem{IGPL} T. Sayed Ahmed,  {\it The class of neat reducts is not elementary.} 
Logic Journal of $IGPL$  {\bf 9}(2001)pp. 593--628.

\bibitem{Fm}  T. Sayed Ahmed  {\it The class of $2$-dimensional polyadic algebras is not elementary},
Fundamenta Mathematica, 
{\bf 172} (2002), pp. 61--81.

\bibitem{note} T. Sayed Ahmed, {\it A note on neat reducts.} Studia Logica, {\bf 85} (2007), pp. 139-151.



\bibitem{MLQ}  T. Sayed Ahmed,  {\it A model-theoretic solution to a problem of Tarski.} Math Logic Quarterly. {\bf 48}(2002), pp. 343--355.



\bibitem{AU} T. Sayed Ahmed,  {\it Amalgamation for reducts of polyadic algebras.} Algebra Universalis, {\bf 51} (2004), p. 301--359.










\bibitem{Sayed}  T. Sayed Ahmed  {\it Completions, Complete representations and Omitting types}, in \cite{1}, pp. 186--205.


\bibitem{Sayedneat}  T. Sayed Ahmed, {\it Neat reducts and neat embeddings in cylindric algebras}, in \cite{1}, pp. 105--134.


\bibitem{au} T. Sayed Ahmed  {\it The class of completely representable polyadic algebras of infinite dimensions is elementary} Algebra universalis 72(1) (2014), pp. 371--390. 



\bibitem{adjoint} T. Sayed Ahmed {\it Neat embeddings as adjoint situations} Synthese {\bf 192} (2015) pp.2223--259

\bibitem{mlq} T. Sayed Ahmed {\it On notions of representability for cylindric--polyadic algebras and a solution to the finitizability problem for first order logic with equality}. 
Mathematical Logic Quarterly
 {\bf 61}(6) (2015) pp. 418--447.






\bibitem{SL}  T. Sayed Ahmed  and I. N\'emeti,  {\it On neat reducts of algebras of logic}, Studia Logica. {\bf 68(2)} (2001), pp. 229--262.





\bibitem{Venema} Y. Venema {\it Atom structures and Sahlqvist equations.}


\end{thebibliography}
\end{document}